\documentclass{compositio}
\usepackage{listings}\usepackage{enumerate}

\RequirePackage{amsmath}
\RequirePackage{amssymb}
\RequirePackage{amsxtra}
\RequirePackage{amsfonts}
\RequirePackage{latexsym}
\RequirePackage{euscript}
\RequirePackage{amscd}
\RequirePackage{amsthm}
\RequirePackage{xypic}

\newcommand{\pigmonkey}{\lambda}
\DeclareMathOperator{\cInd}{c-Ind}
\usepackage{stmaryrd}
\usepackage{mathtools}
\xyoption{all}
\newcommand{\Eins}{\boldsymbol{1}}

\newcommand{\Proj}{\operatorname{Proj}}
\usepackage{tikz}
\makeatletter
\newlength\xvec@height\newlength\xvec@depth\newlength\xvec@width\newcommand{\xvec}[2][]{\ifmmode\settoheight{\xvec@height}{$#2$}\settodepth{\xvec@depth}{$#2$}\settowidth{\xvec@width}{$#2$}\else\settoheight{\xvec@height}{#2}\settodepth{\xvec@depth}{#2}\settowidth{\xvec@width}{#2}\fi\def\xvec@arg{#1}\def\xvec@dd{:}\def\xvec@d{.}\raisebox{.2ex}{\raisebox{\xvec@height}{\rlap{\kern.05em\begin{tikzpicture}[scale=1]
    \pgfsetroundcap
    \draw (.05em,0)--(\xvec@width-.05em,0);
    \draw (\xvec@width-.05em,0)--(\xvec@width-.15em, .075em);
    \draw (\xvec@width-.05em,0)--(\xvec@width-.15em,-.075em);
    \ifx\xvec@arg\xvec@d\fill(\xvec@width*.45,.5ex) circle (.5pt);\else\ifx\xvec@arg\xvec@dd\fill(\xvec@width*.30,.5ex) circle (.5pt);\fill(\xvec@width*.65,.5ex) circle (.5pt);\fi\fi\end{tikzpicture}}}}#2}
\makeatother

\usepackage{dsfont}

\def\A{\mathbb A}

\def\C{\mathbb C}
\def\F{\mathbb F}

\def\Q{\mathbb{Q}}

\def\Z{\mathbb{Z}}
\def\Fbar{\overline{\F}}

\def\m{\mathfrak m}

\def\Mod{\mathrm{Mod}}

\def\chibar{\overline{\chi}}

\def\thetahat{\widehat{\theta}}

\def\unif{\varpi}

\def\id{\mathrm{id}}

\def\et{\mathrm{\acute{e}t}}

\def\alg{\mathrm{alg}}
\def\lalg{\mathrm{l.alg}}
\def\ladm{\mathrm{l.adm}}

\def\an{\mathrm{an}}

\def\ab{\mathrm{ab}}

\def\sm{\mathrm{sm}}

\def\cont{\mathrm{cont}}

\def\GL{\operatorname{GL}}

\def\Gal{\mathrm{Gal}}

\def\Sym{\mathrm{Sym}}
\def\Ext{\mathrm{Ext}}

\def\End{\mathrm{End}}

\def\Art{\mathop{\mathrm{Art}}\nolimits}

\def\Hom{\mathop{\mathrm{Hom}}\nolimits}

\def\Spec{\mathop{\mathrm{Spec}}\nolimits}
\def\Spf{\mathop{\mathrm{Spf}}\nolimits}
\def\Frob{\mathop{\mathrm{Frob}}\nolimits}

\def\Ind{\mathop{\mathrm{Ind}}\nolimits}

\def\Fil{\mathop{\mathrm{Fil}}\nolimits}

\def\soc{\mathop{\mathrm{soc}}\nolimits}

\def\rhobar{\overline{\rho}}

\def\cotimes{\widehat{\otimes}}

\def\WD{\mathrm{WD}}

\def\m{\mathfrak{m}}

\def\iso{\buildrel \sim \over \longrightarrow}

\def\triv{\mathds{1}}

\def\cosoc{\mathop{\mathrm{cosoc}}\nolimits}

\swapnumbers

\newcommand{\onto}{\twoheadrightarrow}

\newcommand{\into}{\hookrightarrow}

\newcommand{\To}{\longrightarrow}
\newcommand{\isoto}{\stackrel{\sim}{\To}}

\newcommand{\Dcris}{D_{\cris}}

\newlength{\ownl}

\DeclareMathOperator{\cind}{c-Ind}

\newcommand{\rec}{{\operatorname{rec}}}

\newcommand{\tr}{{\operatorname{tr}\,}}

\newcommand{\PGL}{\operatorname{PGL}}

\newcommand{\cris}{{\operatorname{cris}}}

\newcommand{\op}{{\operatorname{op}}}

\newcommand{\univ}{{\operatorname{univ}}}

\newcommand{\OO}{{\mathcal{O}}}

\renewcommand{\cH}{\mathcal{H}}

\newcommand{\cO}{\mathcal{O}}

\newcommand{\cV}{\check{\mathbf{V}}}

\newcommand{\ga}{{\mathfrak{a}}}

\newcommand{\barR}{\overline{{R}}}

\newcommand{\tP}{\widetilde{{P}}}

\newcommand{\tR}{\widetilde{{R}}}

\newcommand{\tw}{\mathrm{tw}}

\newcommand{\varepsilonbar  }{\overline{\varepsilon}}

 \newcommand{\sigmabar   }{\overline{\sigma}}

\def\RCS$#1: #2 ${\expandafter\def\csname RCS#1\endcsname{#2}}
\RCS$Revision: 85 $
\RCS$Date: 2011-12-16 18:54:17 -0600 (Fri, 16 Dec 2011) $

\DeclareMathOperator{\mSpec}{m-Spec}

\newcommand{\bb}{\mathbb}

\newcommand{\mf}{\mathfrak}

\newcommand{\TT}{\bb{T}}

\DeclareMathOperator{\ssg}{ss}

\newcommand{\rbar}{{\bar{r}}}

\newcommand{\HT}{\operatorname{HT}}

 \newcommand{\Qp}{{\Q_p}}
\newcommand{\Qptimes}{{\Q_p^\times}}
\newcommand{\Qpbartimes}{\overline{\mathbb{Q}}_p^{\times}}
\newcommand{\Zptimes}{{\Z_p^\times}}

\newcommand{\Zp}{{\Z_p}}
\newcommand{\Ql}{\Q_l} 
\newcommand{\Qpbar}{{\overline{\Q}_p}}
\newcommand{\Zpbar}{{\overline{\Z}_p}}

\newcommand{\Fpbar}{{\overline{\F}_p}}

\newcommand{\Fp}{\F_p}

\newcommand{\md}{\mathrm{m}}

\newcommand{\wE}{\widetilde{E}}\DeclareMathOperator{\Tor}{Tor}  \newcommand{\pp}{\mathfrak p} \newcommand{\ur}{\mathrm{ur}} \newcommand{\ps}{\mathrm{ps}}\newcommand{\gal}{G_{\Qp}}

\usepackage{hyperref}

\newtheorem{thm}[subsection]{Theorem}
\newtheorem{lemma}[subsection]{Lemma}
\newtheorem{lem}[subsection]{Lemma}

\newtheorem{cor}[subsection]{Corollary}

\newtheorem{prop}[subsection]{Proposition}

\theoremstyle{definition}

\newtheorem{defn}[subsection]{Definition}
\newtheorem{defi}[subsection]{Definition}

\theoremstyle{remark}
\newtheorem{remark}[subsection]{Remark}
\newtheorem{rem}[subsection]{Remark}

\newtheorem{example}[subsection]{Example}

\newtheorem{assumption}[subsection]{Assumption}

\newcommand{\dualcat}{\mathfrak C} \DeclareMathOperator{\wtimes}{\widehat{\otimes}} \newcommand{\mm}{\mathfrak{m}} 
\newcommand{\wP}{\widetilde{P}}
\newcommand{\sigmao}{\sigma^{\circ}}

\def\numequation{\addtocounter{subsection}{1}\begin{equation}}
\def\nummultline{\addtocounter{subsubsection}{1}\begin{multline}}
\def\anumequation{\addtocounter{subsection}{1}\begin{equation}}
\renewcommand{\theequation}{\arabic{section}.\arabic{subsection}}

\begin{document}

\title{Patching and the $p$-adic Langlands program for~$\GL_2(\Qp)$}

\author[A. Caraiani]{Ana Caraiani}\email{a.caraiani@imperial.ac.uk}
\address{Department of Mathematics, Imperial College London,
  London SW7 2AZ, UK}

\author[M. Emerton]{Matthew Emerton}\email{emerton@math.uchicago.edu}
\address{Department of Mathematics, University of Chicago,
5734 S.\ University Ave., Chicago, IL 60637, USA}

\author[T. Gee]{Toby Gee} \email{toby.gee@imperial.ac.uk} \address{Department of
  Mathematics, Imperial College London,
  London SW7 2AZ, UK}
  
  \author[D. Geraghty]{David Geraghty}
\email{david.geraghty@bc.edu}\address{Department of Mathematics, 301
  Carney Hall,
  Boston College, Chestnut Hill, MA 02467, USA}

\author[V. Pa\v{s}k\=unas]{Vytautas
  Pa\v{s}k\=unas}\email{paskunas@uni-due.de}\address{
   Fakult\"at f\"ur Mathematik, Universit\"at Duisburg-Essen,  45117 Essen, Germany}

\author[S. W. Shin]{Sug Woo Shin}\email{sug.woo.shin@berkeley.edu}\address{Department of Mathematics, UC Berkeley, Berkeley, CA 94720, USA / Korea Institute for Advanced Study, 85 Hoegiro,
Dongdaemun-gu, Seoul 130-722, Republic of Korea}

\shortauthors{ A. Caraiani et al.}

\thanks{A.C.\ was partially
  supported by the NSF Postdoctoral Fellowship DMS-1204465 and NSF Grant DMS-1501064. M.E.\ was
  partially supported by NSF grants DMS-1003339, DMS-1249548, and DMS-1303450. T.G.\ was
  partially supported by  a Leverhulme Prize, EPSRC grant EP/L025485/1, Marie Curie Career
  Integration Grant 303605, and by
  ERC Starting Grant 306326. D.G.\ was partially supported by NSF grants
  DMS-1200304 and DMS-1128155. V.P.\ was partially supported by the DFG,
  SFB/TR45. S.W.S.\ was partially supported by NSF grant DMS-1162250
  and a Sloan Fellowship.}

\begin{abstract}
We present a new construction of the $p$-adic local Langlands correspondence
for $\GL_2(\Q_p)$ via the patching method of Taylor--Wiles and Kisin.
This construction sheds light on the relationship between the
various other approaches to both the local and global
aspects of the $p$-adic Langlands program; in particular, it gives a
new proof of many cases of the second author's local-global compatibility
theorem, and relaxes a hypothesis on the local mod~$p$ representation in that theorem. \end{abstract}

\maketitle

\section{Introduction}

The primary goal of this paper is to explain how (under mild technical
hypotheses) the patching construction
of \cite{Gpatch}, when applied to the group $\GL_2(\Q_p)$, gives rise to the
$p$-adic local Langlands correspondence for $\GL_2(\Q_p)$, as constructed
in~\cite{MR2642409}, and as further analyzed in \cite{paskunasimage} and
\cite{CDP}. As a by-product, we obtain a new proof of many cases of the local-global
compatibility theorem of~\cite{emerton2010local} (and of some cases
not treated there).

\subsection{Background} We start by recalling the main results of~\cite{Gpatch} and the role we
expect them to play in the (hypothetical) $p$-adic local Langlands correspondence.
Let $F$ be a finite extension of $\Qp$, and let $G_F$ be its absolute Galois group.
One would like to have an analogue of the local Langlands correspondence
for all finite-dimensional, continuous, $p$-adic representations of $G_F$.
Let $E$ be another finite extension of $\mathbb{Q}_p$, which will be our field
of coefficients, assumed large enough,
with ring of integers $\OO$, uniformizer $\varpi$ and residue field $\F$. To a continuous
Galois representation $r: G_F \rightarrow \GL_n(E)$ one would like to attach
an admissible unitary $E$-Banach space representation
$\Pi(r)$ of $G:=\GL_n(F)$ (or possibly a family of such Banach space representations).
Ideally, such a construction should be compatible with deformations,
should encode the classical local Langlands correspondence and should be
compatible with a global $p$-adic correspondence, realized in the completed cohomology
of locally symmetric spaces.

It is expected that the Banach spaces $\Pi(r)$ should encode
the classical local Langlands correspondence in the following way:
if $r$ is potentially semi-stable with regular Hodge--Tate weights,
then the subspace of locally algebraic vectors $\Pi(r)^{\lalg}$ in $\Pi(r)$ should be
isomorphic to $\pi_{\sm}(r)\otimes \pi_{\alg}(r)$ as a
$G$-representation, where
$\pi_{\sm}(r)$ is the smooth representation of $G$ corresponding via classical local Langlands
to the Weil--Deligne representation obtained from
$r$ by Fontaine's recipe, and $\pi_{\alg}(r)$ is an algebraic representation
of $G$, whose highest weight vector is determined by the Hodge--Tate weights of $r$.

\begin{example} If $F=\Qp$, $n=2$ and $r$ is crystalline with Hodge--Tate weights $a< b$,
then $\pi_{\sm}(r)$ is a smooth unramified principal series representation,
whose Satake parameters can be calculated in terms of the trace and
determinant of Frobenius on $D_{\mathrm{cris}}(r)$, and
$\pi_{\alg}(r)= \Sym^{b-a-1} E^2 \otimes\det^{1-a}$.
(We note that in the literature different normalisations
lead to different twists by a power of $\det$.)
\end{example}

Such a correspondence has been established in the case of $n=2$
and $F=\Qp$ by the works of Breuil, Colmez and others,
see~\cite{MR2493214}, \cite{MR2642402} as well as the
introduction to \cite{MR2642409}. Moreover,
when $n=2$ and $F=\Qp$, this correspondence has
been proved (in most cases) to satisfy local-global compatibility with the
$p$-adically completed cohomology of modular curves, see~\cite{emerton2010local}.
However, not much is known beyond this case. In~\cite{Gpatch} we
have constructed a candidate for such a correspondence
using the Taylor--Wiles--Kisin patching method,
which has been traditionally employed to prove modularity lifting theorems for Galois representations.
We now describe the end product of the paper \cite{Gpatch}.

Let  $\rbar: G_F\rightarrow \GL_n(\F)$ be a continuous representation and let $R^{\square}_p$ be its universal framed deformation ring. Under the assumption that $p$ does not divide $2n$
we construct an $R_{\infty}[G]$-module $M_{\infty}$, which is finitely generated as a module over the completed group algebra $R_{\infty}[[ \GL_n(\OO_F)]]$, where $R_\infty$ is a complete local noetherian
$R_p^{\square}$-algebra with residue field $\F$. If  $y\in \Spec R_\infty$ is an $E$-valued point then $$\Pi_y:=\Hom_{\OO}^{\cont}( M_{\infty}\otimes_{R_\infty, y} \OO, E)$$ is an admissible
unitary $E$-Banach space representation of $G$. The composition $R_p^{\square}\rightarrow R_{\infty}\overset{y}{\rightarrow} E$ defines an $E$-valued point $x\in \Spec R_p^{\square}$ and
thus a continuous Galois representation $r_x: G_F \rightarrow
\GL_n(E)$. We expect that the Banach space representation $\Pi_y$
depends only on $x$ and that it should be related to $r_x$ by the hypothetical
$p$-adic Langlands correspondence; see \S 6 of \cite{Gpatch} for a detailed discussion. We show in
\cite[Theorem 4.35]{Gpatch} that
if $\pi_{\sm}(r_x)$ is generic and $x$ lies on an automorphic component of a potentially crystalline deformation ring of $\rbar$ then
$\Pi_y^{\lalg}\cong \pi_{\sm}(r_x)\otimes \pi_{\alg}(r_x)$ as
expected; moreover, the points $x$ such that $\pi_{\sm}(r_x)$ is
generic are Zariski dense in every irreducible component of a
potentially crystalline deformation ring. (It is expected that every irreducible component of a
potentially crystalline deformation ring is automorphic; this expectation is motivated by the
Fontaine--Mazur and Breuil--M\'ezard conjectures. However, it is intrinsic to our method that we would not be able to access these
non-automorphic components even if they existed.)

However, there are many natural questions regarding our construction for $\GL_n(F)$
that we cannot answer at the moment and that appear to be genuinely
deep, as they are intertwined with questions about local-global compatibility
for $p$-adically completed cohomology, with the Breuil--M\'ezard conjecture
on the geometry of local deformation rings and with the
Fontaine--Mazur conjecture for global Galois representations.
For example, it is not clear that $\Pi_y$ depends only on $x$,
it is not clear that $\Pi_y$ is non-zero for an arbitrary $y$, and
that furthermore~$\Pi_y^\lalg$ is non-zero if~$r_x$ is potentially
semistable of regular weight,
and it is not at all clear that $M_{\infty}$ does
not depend on the different choices made during the patching process.

\subsection{The present paper} In this paper, we specialize
the construction of \cite{Gpatch} to the case $F=\Qp$ and $n=2$
(so that $G:=\GL_2(\Qp)$ and $K:=\GL_2(\Zp)$ from now on)
to confirm our expectation that, firstly,
$M_{\infty}$ does not depend on any of the choices made during the patching process
and, secondly, that it does recover the $p$-adic local
Langlands correspondence as constructed by Colmez.

We achieve the first part
\emph{without} appealing to Colmez's construction
(which relies on the theory of $(\varphi,\Gamma)$-modules).
The proof that $M_\infty$ is uniquely determined highlights
some key features of the $\GL_2(\mathbb{Q}_p)$ setting beyond
the use of $(\varphi,\Gamma)$-modules:
the classification of irreducible mod $p$ representations of
$\GL_2(\Q_p)$ in terms of Serre weights and Hecke operators,
and the fact that the Weil--Deligne representation and
the Hodge--Tate weights determine a (irreducible) $2$-dimensional crystalline
representation of $G_{\Qp}$ uniquely (up to isomorphism).

When combined with the results of~\cite{paskunasimage}
(which \emph{do} rely on Colmez's functor $\cV$), we obtain that $M_{\infty}$ realizes the $p$-adic Langlands correspondence
as constructed by Colmez.

We also obtain a new proof of
local-global compatibility, which helps clarify the relationship
between different perspectives and approaches to $p$-adic local Langlands.

\subsection{Arithmetic actions} In the body of the paper we restrict the representations $\rbar$
we consider by assuming that they have only scalar endomorphisms, so that $\End_{G_{\Qp}}(\rbar)=\F$, and that
 $\rbar\not \cong \bigl ( \begin{smallmatrix} \omega & \ast \\ 0 & 1 \end{smallmatrix}\bigr)\otimes\chi$
 for any character $\chi: G_{\Qp} \rightarrow \F^{\times}$.
 For simplicity, let us assume in this introduction that $\rbar$ is irreducible
 and let $R_p$ be its universal deformation ring. Then $R_p^{\square}$ is formally
 smooth over $R_p$. Moreover, (as $F=\Qp$ and $n=2$)
 we may also assume that $R_{\infty}$ is formally smooth over
 $R_p^{\square}$, and thus over~$R_p$.

The following definition is meant to axiomatize the key properties of the patched module $M_{\infty}$.

\begin{defi} Let $d$ be a non-negative integer, let $R_{\infty}:=R_p[[x_1, \ldots, x_d]]$
and let $M$ be a non-zero $R_\infty[G]$-module. We say that the action of $R_{\infty}$ on $M$ is \textit{arithmetic}
if the following conditions hold:
\begin{itemize}
\item[(AA1)] $M$ is a finitely generated module over the completed group algebra $R_{\infty}[[K]]$;
\item[(AA2)] $M$ is projective in the category of pseudo-compact $\OO[[K]]$-modules;
\item[(AA3)] for each pair of integers $a< b$, the action of $R_{\infty}$ on
$$M(\sigma^{\circ}):=\Hom^{\cont}_{\OO[[K]]}(M, (\sigma^{\circ})^d)^d$$ factors through the action
of $R_{\infty}(\sigma):=R_p(\sigma)[[x_1,\ldots, x_d]]$. Here $R_p(\sigma)$ is the
quotient of $R_p$ constructed by Kisin, which parameterizes crystalline
representations with Hodge--Tate weights $(a, b)$, $\sigma^{\circ}$ is a
$K$-invariant $\OO$-lattice in $\sigma:=\Sym^{b-a-1} E^2 \otimes \det^{1-a}$ and $(\ast)^d:=\Hom^{\cont}_{\OO}(\ast, \OO)$ denotes
the Schikhof dual.

 Moreover,
$M(\sigma^{\circ})$ is maximal Cohen--Macaulay over $R_\infty(\sigma)$ and
the $R_{\infty}(\sigma)[1/p]$-module $M(\sigma^{\circ})[1/p]$ is locally
free of rank $1$ over its support.
\item[(AA4)] for each $\sigma$ as above and each maximal ideal $y$ of $R_{\infty}[1/p]$ in the support of
$M(\sigma^{\circ})$,  there is a non-zero $G$-equivariant map $$ \pi_{\sm}(r_x)\otimes
\pi_{\alg}(r_x)\rightarrow \Pi_y^{\lalg}$$ where~$x$ is the image of $y$ in $\Spec R_p$.\end{itemize}
\end{defi}
The last condition says that $M$ encodes the classical
local Langlands correspondence. This is what motivated us to call such actions arithmetic.
(In fact in the main body of the paper we use a reformulation of condition (AA4),
see \S\ref{subsec: axioms for arithmetic action} and Remark \ref{rem: AA4 alternatives}.)
To motivate (AA3), we note for the sake of the reader familiar with Kisin's proof of the
Fontaine--Mazur conjecture \cite{kisinfmc} that the modules $M(\sigma^\circ)$ are analogues of the patched modules
denoted by $M_{\infty}$ in \cite{kisinfmc}, except that Kisin patches algebraic automorphic forms for definite quaternion
algebras and in this paper we will ultimately be making use of patching
arguments for algebraic automorphic forms on forms of~$U(2)$.

\subsection{Uniqueness of $M_\infty$} As already mentioned,
the patched module $M_{\infty}$ of \cite{Gpatch}
carries an arithmetic action of $R_{\infty}$ for some $d$. In order to prove that $M_\infty$ is uniquely
determined, it is enough to show that for any given $d$, any $R_{\infty}[G]$-module $M$ with an arithmetic action of $R_\infty$
is uniquely determined. The following is our main result, which for simplicity we state under the assumption that $\rbar$ is irreducible.

\begin{thm}\label{main_intro} Let $M$ be an $R_{\infty}[G]$-module
with an arithmetic action of $R_{\infty}$.

\begin{enumerate}
\item If $\pi$ is any irreducible $G$-subrepresentation of the Pontryagin
dual $M^{\vee}$ of $M$  then $\pi$ is isomorphic to the representation
of $G$ associated to $\rbar$ by the mod $p$ local Langlands correspondence for $\GL_2(\Qp)$.
\item Let $\pi\hookrightarrow J$ be an injective
envelope of the above $\pi$ in the category of smooth locally
admissible representations of $G$ on $\OO$-torsion modules.
Let $\wP$ be the Pontryagin dual of $J$. Then $\wP$ carries a unique arithmetic
action of $R_p$ and, moreover,
$$ M \cong \wP\wtimes_{R_p} R_{\infty}$$
as $R_{\infty}[G]$-modules.
\end{enumerate}
\end{thm}

\noindent The theorem completely characterizes modules with an arithmetic action
and shows that $M_{\infty}$ does not depend on the choices made in the patching process.
A further consequence is that the Banach space $\Pi_y$ depends
only on the image of $y$ in $\Spec R_p$, as expected.

Let us sketch the proof of Theorem~\ref{main_intro} assuming for simplicity that $d=0$.
The first step is to show that
$M^{\vee}$ is an injective object
in the category  of smooth locally admissible representations of
$G$ on $\OO$-torsion modules and that its $G$-socle is isomorphic to $\pi$.
This is done by computing $\Hom_G(\pi', M^{\vee})$ and
showing that $\Ext^1_G(\pi', M_{\infty}^{\vee})$ vanishes for all
irreducible $\F$-representations $\pi'$ of $G$; see
Proposition \ref{prop: abstract Serre weight freeness and Hecke compatibility}
and Theorem \ref{thm: Minfty is projective}.
The arguments here use the foundational
results of Barthel--Livn\'e \cite{barthel-livne} and Breuil \cite{breuil1}
on the classification of irreducible mod~$p$ representations of $G$,
arguments related to the weight part of Serre's conjecture,
and the fact that the rings $R_p(\sigma)$ are formally smooth over $\OO$, whenever $\sigma$ is of
the form $\Sym^{b-a-1} E^2 \otimes \det^{1-a}$ with $1\le b-a\le p$.

This first step allows
us to conclude that $M^{\vee}$ is an injective envelope of $\pi$,
which depends only on $\rbar$. Since injective envelopes
are unique up to isomorphism, we conclude that any two modules with
an arithmetic action of $R_p$ are isomorphic as $G$-representations.
Therefore, it remains to show that any two arithmetic actions of $R_p$ on $\wP$ coincide.
As $R_p$ is $\OO$-torsion free, it is enough to show that two such actions
induce the same action on the unitary $E$-Banach space
$\Pi:=\Hom_{\OO}^{\cont}(M, E)$. Since $M$ is a projective $\OO[[K]]$-module by (AA2) one may show using
 the ``capture'' arguments that appear in \cite[\S 2.4]{CDP},
 \cite[Prop.\ 5.4.1]{emerton2010local} that the subspace
 of $K$-algebraic vectors in $\Pi$ is dense.
 Since the actions of $R_p$ on $\Pi$ are continuous it is enough to show that
 they agree on this dense subspace. Since the subspace of
 $K$-algebraic vectors is semi-simple as a $K$-representation,
 it is enough to show that the two actions agree on $\sigma$-isotypic subspaces in $\Pi$ for all irreducible
 algebraic $K$-representations $\sigma$.
 These are precisely the representations $\sigma$ in axiom (AA3).
 Taking duals one more time, we are left with showing
 that any two arithmetic actions induce the same action
 of $R_p$ on $M(\sigma^{\circ})[1/p]$ for all $\sigma$ as above.

 At this point we use another special feature of $2$-dimensional
 crystalline representations of $G_{\Qp}$: the associated Weil--Deligne representation
together with Hodge--Tate weights determine a $2$-dimensional crystalline
representation of $G_{\Qp}$ up to isomorphism.
Using this fact and axioms (AA3) and (AA4) for the arithmetic
action we show that the action of the Hecke algebra
$\mathcal H(\sigma):=\End_G(\cInd_K^G \sigma)$ on
$M(\sigma^{\circ})[1/p]$ completely determines the
action of $R_p(\sigma)$ on $M(\sigma^{\circ})[1/p]$; see the proof
of Theorem \ref{unique_arithmetic} as well as the key
Proposition \ref{prop: generic fibre of cristabelline deformation ring}.
Since the action of $\cH(\sigma)$ on $M(\sigma^\circ)[1/p]$
depends only on the $G$-module structure of $M$, we are
able to conclude that the two arithmetic actions are the same.
The reduction from the case when $d$ is arbitrary
to the case when $d=0$ is carried out in \S \ref{subsec: Vytas' argument to factor out patching variables}.

\begin{rem} As we have already remarked, the arguments up to this point
  make no use of $(\varphi, \Gamma)$-modules.
 Indeed the
proof of Theorem \ref{main_intro} does not use them.
One of the objectives of this project was to find out how much
of the $p$-adic Langlands for $\GL_2(\Qp)$ correspondence can
one recover from the patched module $M_{\infty}$
without using Colmez's functors,
as these constructions are not available
for groups other than $\GL_2(\Qp)$, while our patched module $M_{\infty}$ is.
Along the same lines, in section \S \ref{sec: Berger Breuil stuff} we
show that to a large extent we can recover a fundamental theorem of
Berger--Breuil \cite{berger-breuil} on the
uniqueness of unitary completions of locally algebraic principal
series without making use of $(\varphi, \Gamma)$-modules,
see Theorem \ref{thm:Berger--Breuil completions via Minfty} and Remark \ref{we_can_do_crystabelline}.
\end{rem}

\begin{rem} As already explained, Theorem \ref{main_intro} implies that
$\Pi_y$ depends only on the image of $y$ in $\Spec R_p$.
However, we are still not able to deduce using only our methods that
$\Pi_y$ is non-zero for an arbitrary $y\in \mSpec R_{\infty}[1/p]$.
Since $M_{\infty}$ is not a finitely generated module over $R_{\infty}$ theoretically it could happen that
$\Pi_y\neq 0$ for a dense subset of $\mSpec R_{\infty}[1/p]$, but
$\Pi_y=0$ at all other maximal ideals. We can only prove that this
pathological situation does not occur after combining
Theorem~\ref{main_intro} with the results of~\cite{paskunasimage}.
\end{rem}

In \S \ref{sec:comparison} we relate the arithmetic action of $R_p$ on
$\wP$ to the results of \cite{paskunasimage}, where an action of $R_p$ on an injective envelope of $\pi$
 in the subcategory of representations with a fixed central character is constructed using Colmez's functor;
 see Theorem \ref{final_comparison}.
 Then by appealing to the results of \cite{paskunasimage} we show that $\Pi_y$ and $r_x$
correspond to each other under the $p$-adic Langlands correspondence as defined by Colmez \cite{MR2642409}, for
all $y\in \mSpec R_{\infty}[1/p]$, where $x$ denotes the image of $y$ in $\mSpec R_p[1/p]$.

It follows from the construction of $M_{\infty}$ that after quotienting out by a certain ideal
of $R_{\infty}$ we obtain a dual of completed cohomology, see \cite[Corollary 2.11]{Gpatch}.
This property combined with Theorem \ref{main_intro} and with the results in \S \ref{sec:comparison} enables us to obtain a new proof of local-global compatibility as in \cite{emerton2010local} as well as obtaining a genuinely new result, when $\rhobar|_{G_{\Qp}}$ is isomorphic to
$\bigl ( \begin{smallmatrix} 1 & \ast \\ 0 &
    \omega \end{smallmatrix}\bigr)\otimes\chi$, where $\omega$ is the
  cyclotomic character modulo $p$. (See Remark~\ref{rem: we've done something new}.)

\subsection{Prospects for generalization}

      Since our primary goal in this paper is  to build some new connections between various existing ideas
      related to the $p$-adic Langlands program, we have not striven
      for maximal generality, and we expect that some of our
      hypotheses on~$\rbar:G_{\Qp}\to\GL_2(\Fpbar)$ could be
      relaxed. In particular, it should be possible to prove results
      when~$p=2$ by using results of  Thorne~\cite{thornep2} to redo the patching in
      \cite{Gpatch}. It may also be possible to extend our results to
      cover more general~$\rbar$ (recall that we assume that
      $\rbar$~has only scalar endomorphisms, and that it is not a
      twist of an extension of the trivial character by the mod~$p$
      cyclotomic character). In Section~\ref{subsec: wild
        speculations} we discuss the particular case where~$\rbar$ has scalar
      semisimplification; as this discussion (and the arguments
      of~\cite{paskunasimage}) show, while it may well be possible to
      generalise our arguments, they will necessarily be considerably more
      involved in cases where~$\rbar$ does not satisfy the hypotheses
      that we have imposed.

      Since the patching construction in~\cite{Gpatch} applies equally
      well to the case of~$\GL_2(F)$ for any finite extension~$F/\Qp$,
      or indeed to~$\GL_n(F)$, it is natural to ask whether any of our
      arguments can be extended to such cases (where there is at
      present no construction of a $p$-adic local Langlands
      correspondence). As explained in Remark~\ref{rem:axioms are
        natural}, the natural analogues of our axioms (AA1)-(AA4)
      hold, even in the generality of~$\GL_n(F)$. Unfortunately, the
      prospects for proving analogues of our main theorems are less
      rosy, as it seems that none of the main inputs to our arguments
      will hold. Indeed, already for the case of~$\GL_2(\Q_{p^2})$
      there is no analogue available of the classification
      in~\cite{breuil1} of the irreducible $\F$-representations
      of~$\GL_2(\Qp)$, and it is clear from the results
      of~\cite{Breuil_Paškūnas_2012} that any such classification
      would be much more complicated.

      Furthermore, beyond the case of~$\GL_2(\Qp)$ it is no longer the
      case that crystalline representations are (essentially)
      determined by their underlying Weil--Deligne representations, so
      there is no possibility of deducing that a $p$-adic
      correspondence is uniquely determined by the classical
      correspondence in the way that we do here, and no hope that an
      analogue of the results of~\cite{berger-breuil} could
      hold. Finally, it is possible to use the constructions
      of~\cite{MR2128381} to show that for~$\GL_2(\Q_{p^2})$ the
      patched module~$M_\infty$ is not a projective
      $\GL_2(\Q_{p^2})$-module.

\subsection{Outline of the paper}\label{subsec:outline of the paper}In Section~\ref{sec:Galois
  deformation rings and Hecke algebras} we recall some well-known
results about Hecke algebras and crystalline deformation rings
for~$\GL_2(\Qp)$. The main result in this section is
Proposition~\ref{lem:rbar has a unique Serre weight, and the
  deformation ring is smooth, and given by Tp}, which describes the
crystalline deformation rings corresponding to Serre weights as
completions of the corresponding Hecke algebras. In Section~\ref{sec:
  recalling from Gpatching} we explain our axioms for a module with an arithmetic action, and show how the
results of~\cite{Gpatch} produce patched modules~$M_\infty$ satisfying these
axioms.

Section~\ref{sec:
  proof that arithmetic action is unique} proves that the axioms
determine~$M_\infty$ (essentially) uniquely, giving a new construction
of the $p$-adic local Langlands correspondence for~$\GL_2(\Qp)$. It
begins by showing that~$M_\infty$ is a projective $\GL_2(\Qp)$-module
(Theorem~\ref{thm: Minfty is projective}), before making a
category-theoretic argument that allows us to ``factor out'' the
patching variables (Proposition~\ref{cut_it_out}). We then use the
``capture'' machinery to complete the proof.

In Section~\ref{sec: Berger Breuil stuff} we explain how our results
can be used to give a new proof (not making use of
$(\varphi,\Gamma)$-modules) that certain  locally algebraic principal series
representations  admit at most one unitary
completion. Section~\ref{sec:comparison} combines our results with
those of~\cite{paskunasimage} to show that our construction is
compatible with Colmez's correspondence, and as a byproduct extends
some results of~\cite{paskunasimage} to a situation where the central
character is not fixed.

Finally, in Section~\ref{sec: local global} we explain how our results
give a new proof of the second author's local-global compatibility theorem, and
briefly explain how such results can be extended to quaternion
algebras over totally real fields (Remark~\ref{rem: division algebra
  case}).

\subsection{Notation}\label{subsec:notation}We fix an odd prime $p$,
an algebraic closure $\Qpbar$ of $\Qp$, and a
finite extension $E/\Qp$ in $\Qpbar$, which will be our coefficient field. We
write $\cO=\cO_E$ for the ring of integers in $E$,
$\varpi=\varpi_E$ for a uniformiser, and $\F:=\cO/\varpi$ for the residue field. We will assume without comment
that $E$ and $\F$ are sufficiently large, and in particular that if we
are working with representations of the absolute Galois group of a $p$-adic
field $K$, then the images of all embeddings $K\into \Qpbar$ are
contained in $E$.

  \subsubsection{Galois-theoretic notation}
If $K$ is a field, we let $G_K$ denote its absolute Galois group.  Let $\varepsilon$ denote the $p$-adic cyclotomic character,
and $\varepsilonbar=\omega$ the mod $p$ cyclotomic character. If $K$ is a
finite extension of $\bb{Q}_p$ for some $p$, we write $I_K$ for the
inertia subgroup of $G_K$. If $R$ is a local ring we write
$\mf{m}_{R}$ for the maximal ideal of $R$. If $F$ is a number field
and $v$ is a finite place of $F$ then we let $\Frob_v$ denote a
geometric Frobenius element of~$G_{F_v}$.

If $K/\Qp$ is a finite extension, we write
$\Art_K:K^\times \iso W_K^\ab$ for the Artin map normalized to send
uniformizers to geometric Frobenius elements. To avoid cluttering up
the notation, we will use~$\Art_{\Qp}$ to regard characters of
$\Qptimes$, $\Zptimes$ as characters of $G_{\Qp}$, $I_{\Qp}$
respectively, without explicitly mentioning~$\Art_{\Qp}$ when we do
so.

If $K$ is a $p$-adic field and $\rho$ a de Rham
representation of $G_K$ over $E$ and if $\tau:K \into E$
then we will write $\HT_\tau(\rho)$ for the multiset of Hodge--Tate
numbers of $\rho$ with respect to $\tau$.  By definition, the multiset
$\HT_\tau(\rho)$ contains $i$ with multiplicity
$\dim_{E} (\rho\otimes_{\tau,K} \widehat{\overline{K}}(i))^{G_K}
$. Thus for example $\HT_\tau(\varepsilon)=\{ -1\}$.
If $\rho$ is moreover crystalline then we have the associated
filtered $\varphi$-module $\Dcris(\rho):=(\rho\otimes_{\Q_p} B_{\mathrm{cris}})^{G_K}$,
where $B_{\mathrm{cris}}$ is Fontaine's crystalline period ring.

\subsubsection{Local Langlands correspondence}
 Let $n\in \Z_{\ge 1}$, let $K$ be a finite extension of $\Qp$, and let $\rec$ denote the local
 Langlands correspondence from isomorphism classes of irreducible
 smooth representations of $\GL_n(K)$ over $\C$ to isomorphism classes
 of $n$-dimensional Frobenius semisimple Weil--Deligne representations
 of $W_K$ defined in \cite{ht}. Fix an isomorphism $\imath:\Qpbar\to\C$. We define the
 local Langlands correspondence $\rec_p$ over $\Qpbar$ by $\imath \circ
 \rec_p = \rec \circ \imath$. Then $r_p(\pi):=\rec_p(\pi\otimes |\det|^{(1-n)/2})$
 is independent of the choice of $\imath$.
 In this paper we are mostly concerned with the case that $n=2$ and $K=\Qp$.

\subsubsection{Notation for duals}
\label{subsubsec:duals}
If $A$ is a topological  $\cO$-module, we write $A^\vee:=\Hom_\cO^\cont(A,E/\cO)$ for the
Pontryagin dual of $A$. We apply this to $\cO$-modules that are either
discrete or profinite, so that the usual formalism of Pontryagin duality
applies.

If $A$ is a pseudocompact $\cO$-torsion free $\cO$-module,
we write $A^d:=\Hom_\cO^\cont(A,\cO)$ for its Schikhof dual.

If $F$ is a free module of finite rank over a ring $R$,
then we write $F^\ast := \Hom_R(F,R)$ to denote its $R$-linear dual,
which is again a free $R$-module of the same rank over $R$ as $F$.

If $R$ is a commutative $\cO$-algebra, and if $A$
is an $R$-module that is
pseudocompact and $\cO$-torsion free as an $\cO$-module,
then we may form its Schikhof dual $A^d$,
which has a natural $R$-module structure via the transpose action,
extending its $\cO$-module structure.
If $F$ is a finite rank free $R$-module,
then $A\otimes_R F$ is again an $R$-module that is pseudocompact
as an $\cO$-module
(if $F$ has rank $n$
then it is non-canonically isomorphic to a direct sum of $n$ copies
of $A$), and there is a canonical isomorphism of $R$-modules
$(A\otimes_R F)^d\iso
A^d \otimes_R F^\ast.$

\subsubsection{Group-theoretic notation}
Throughout the paper
we write $G=\GL_2(\Qp)$ and $K=\GL_2(\Zp)$, and let $Z=Z(G)$ denote
the centre of $G$.  We also let $B$ denote the Borel subgroup of $G$
consisting of upper triangular matrices, and $T$ denote the diagonal
torus contained in $B$.

If $\chi: T \to E^{\times}$ is a continuous character,
then we define the continuous induction
$(\Ind_B^G \chi)_{\cont}$ to be the $E$-vector space
of continuous functions $f:G \to E$ satisfying the condition
$f(bg) = \chi(b) f(g)$ for all $b \in B$ and $g \in G$; it forms a
$G$-representation with respect to the right regular action.
If $\chi$ is in fact a smooth character,
then we may also form the smooth induction $(\Ind_B^G \chi)_{\sm}$;
this is the $E$-subspace of $(\Ind_B^G\chi)_{\cont}$ consisting of smooth
functions, and is a $G$-subrepresentation of the continuous induction.

If $\chi_1$ and $\chi_2$ are continuous characters of $\Qptimes$,
then the character $\chi_1\otimes \chi_2 : T \to E^{\times}$ is defined
via $\bigl ( \begin{smallmatrix} a & 0 \\ 0 & d \end{smallmatrix} \bigr ) \mapsto
\chi_1(a)\chi_2(d)$.  Any continuous $E$-valued character $\chi$ of $T$
is of this form, and $\chi$ is smooth if and only if $\chi_1$ and $\chi_2$
are.

\section{Galois deformation rings and Hecke algebras}\label{sec:Galois
  deformation rings and Hecke algebras}

\subsection{Galois deformation rings}\label{subsec: local Galois
  deformation rings}Recall that we assume throughout the paper that
$p$ is an odd prime.  Fix a continuous representation
$\rbar:G_\Qp\to\GL_2(\F)$, where as before $\F/\Fp$ is a finite
extension. Possibly enlarging $\F$, we fix a sufficiently large extension
$E/\Qp$ with ring of integers $\cO$ and residue field $\F$.

We will make the following assumption from now on.
\begin{assumption}
  \label{assumption: rbar is generic enough}
  Assume that
  $\End_{G_{\Qp}}(\rbar)=\F$, and that
  $\rbar\not \cong \bigl ( \begin{smallmatrix} \omega & \ast \\ 0 &
    1 \end{smallmatrix}\bigr)\otimes\chi$,
  for any character $\chi: G_{\Qp} \rightarrow \F^{\times}$.
\end{assumption}
\noindent In particular this assumption implies that $\rbar$ has a universal deformation
$\cO$-algebra ~$R_p$,
and that either $\rbar$ is
(absolutely) irreducible, or $\rbar$ is a non-split extension of characters.

We begin by recalling
the relationship between crystalline deformation rings of~$\rbar$, and
the representation theory of~$G := \GL_2(\Q_p)$ and~$K:= \GL_2(\Z_p)$.
Given a pair of integers $a\in \Z$ and $b\in \Z_{\ge 0}$, we let $\sigma_{a,b}$ be the absolutely irreducible
$E$-representation $\det^a\otimes\Sym^{b}E^2$ of
$K$. Note that this is just the algebraic representation of
highest weight $(a+b,a)$ with respect to the Borel subgroup given by the upper-triangular matrices in $G$.

We say that a representation $r:G_{\Qp}\to \GL_2(\Qpbar)$ is crystalline
of Hodge type $\sigma=\sigma_{a,b}$ if it is crystalline with Hodge--Tate
weights $(1-a,-a-b)$,\footnote{Note that this convention agrees with those
  of~\cite{Gpatch} and~\cite{emerton2010local}.} we write
$R_p(\sigma)$, for the reduced, $p$-torsion free
quotient of $R_p$ corresponding to crystalline deformations of Hodge type~$\sigma$.

\subsection{The morphism from the Hecke algebra to the deformation
  ring}\label{subsec: morphism from Hecke to Galois}We briefly recall some
results from~\cite[\S4]{Gpatch}, specialised to the case
of crystalline representations of~$\GL_2(\Qp)$.

Set $\sigma=\sigma_{a,b}$, and let
$\cH(\sigma):=\End_G(\cind_K^G\sigma)$. The action of $K$ on $\sigma$ extends to the action of $G$. This gives rise
to the  isomorphism of $G$-representations:
$$ (\cInd_K^G \Eins)\otimes \sigma\cong \cInd_K^G \sigma, \quad f\otimes v\mapsto [ g\mapsto f(g) \sigma(g) v].$$
The map
\numequation\label{eqn: Hecke algebra iso}\cH(\triv)\rightarrow \cH(\sigma), \quad \phi\mapsto \phi\otimes \id_{\sigma}\end{equation}
is an isomorphism of $E$-algebras  by Lemma 1.4 of \cite{MR2290601}. Using the above isomorphism we will
identify elements of $\cH(\sigma)$ with $E$-valued $K$-biinvariant functions
on~$G$, supported on finitely many double cosets.

\begin{prop}
  \label{prop: explicit description of the Hecke algebra of a principal series}
    Let $S\in \cH(\sigma)$ be the function supported on the double coset of
   $\left(\begin{smallmatrix}
      p&0\\0& p
    \end{smallmatrix}\right)$, with value $p^{2a+b}$ at   $\left(\begin{smallmatrix}
      p&0\\0&p
    \end{smallmatrix}\right)$, and let
    $T\in \cH(\sigma)$ be the function supported on the double coset\footnote{The function supported on the double coset
 $KgK$ with value $1$ at $g$, viewed as an element of $\mathcal H(\Eins)$, acts on $v\in V^K$ by the formula $[KgK] v=\sum_{h K \subset KgK} h v$.}  of $\left(\begin{smallmatrix}
      p&0\\0&1
    \end{smallmatrix}\right)$, with value  $p^{a+b}$ at     $\left(\begin{smallmatrix}
      p&0\\0&1
    \end{smallmatrix}\right)$. Then $\cH(\sigma) =E[S^{\pm 1},T]$ as an $E$-algebra.

    \end{prop}
\begin{proof}
This is immediate from~(\ref{eqn: Hecke algebra iso}) and the Satake isomorphism.
\end{proof}

Let  $r, s$ be integers with $r < s$, and let $t, d\in E$ with $d\in E^\times$. We let $D:=D(r,s, t,d)$ be the two-dimensional filtered $\varphi$-module
that has $e_1$, $e_2$ as a basis of its underlying $E$-vector space,
has its $E$-linear Frobenius endomorphism $\varphi$ being given by
$$\varphi(e_1)=e_2, \, \varphi(e_2)=-d e_1 + t e_2,$$
and has its
Hodge filtration given by
$$\Fil^i D= D \text{ if }i\le r , \,  \Fil^i D= E e_1
\text{ if }r+1\le i\le s, \, \text{ and }\Fil^i D=0 \text{ if }i> s.$$
We note that $t$ is the trace and
$d$ is the determinant of $\varphi$ on~$D$, and both are therefore
determined uniquely by~$D$.
The same construction works if $E$ is replaced with an $E$-algebra $A$. We will still write $D(r,s,t,d)$ for the resulting $\varphi$-module with $A$-coefficients if the coefficient algebra is clear from the context.

\begin{lem}\label{well_known} If $V$ is an indecomposable
  $2$-dimensional crystalline representation of $G_{\Qp}$ over $E$
  with distinct Hodge--Tate weights $(s,r)$ then there exists a unique pair $(t, d)\in E\times E^{\times}$ such that $\Dcris(V)\cong D(r,s,t,d)$. Moreover,
$v_p(d)= r+s$ and $v_p(t) \ge r$.
\end{lem}
\begin{proof} This is well known, and is a straightforward computation
	using the fact that $\Dcris(V)$ is weakly admissible.
For the sake of completeness, we sketch the proof; the key fact one employs
is that $V \mapsto \Dcris(V)$ is a fully faithful embedding of the category
of crystalline representations of $G_{\Q_p}$ into the category
of weakly admissible filtered $\varphi$-modules.  (Indeed,
it induces an equivalence between these two categories,
but that more difficult fact isn't needed for this computation.)
We choose $e_1$ to be a basis for $\Fil^{s} \Dcris(V)$; the assumption
that $V$ is indecomposable implies that $\Fil^{s} \Dcris(V)$ is not
stable under $\varphi,$ and so if we write $e_2:= \varphi(e_1)$ then
$e_1,e_2$ is a basis for $\Dcris(V)$, and $\varphi$ has a matrix
of the required form for a uniquely determined $t$ and $d$.
The asserted relations between $v_p(t)$, $v_p(d)$, $r$, and $s$ follow
from the weak admissibility of $\Dcris(V)$.
\end{proof}

In fact, it will be helpful to state a generalisation of the previous
result to the context of finite dimensional $E$-algebras.
(Note that the definition of $D(r,s,t,d)$ extends naturally
to the case when $t$ and $d$ are taken to lie in such a finite-dimensional
algebra.)

\begin{lem}\label{also_well_known} If $A$ is an Artinian local $E$-algebra
	with residue field $E'$, and if $V_A$ is a crystalline representation
	of rank two over $A$ whose associated residual representation
	$V_{E'} := E' \otimes_{A} V_A$ is indecomposable with distinct
	Hodge--Tate weights $(s,r)$, and if $\Dcris(V_A)$ denotes the
	filtered $\varphi$-module  associated to $V_A$, then there exists
	a  unique pair $(t,d) \in A\times A^{\times}$
       	such that $\Dcris(V_A) \cong D(r,s,t,d)$.
\end{lem}
\begin{proof}
	Choose a basis
        $\overline{e}_1$ for~$\Fil^s\Dcris(V_{E'})$, and
        choose~$e_1\in\Fil^s\Dcris(V_A)$  lifting~$\overline{e}_1$. By
        Na\-ka\-ya\-ma's lemma, $e_1$ generates~$\Fil^s\Dcris(V_A)$, and
        by considering the length of~$\Fil^s\Dcris(V_A)$ as an $E$-vector
        space, we see that~$\Fil^s\Dcris(V_A)$ is a free $A$-module of
        rank one.

        Let $e_2=\phi(e_1)$, and write~$\overline{e}_2$ for the image
        of~$e_2$ in~$\Dcris(V_{E'})$. As in the proof of
        Lemma~\ref{well_known}, $\overline{e}_1,\overline{e}_2$ is a
        basis of~$\Dcris(V_{E'})$, and thus by another application of
        Nakayama's lemma, $e_1,e_2$ are an $A$-basis
        of~$\Dcris(V_A)$. The matrix of~$\phi$ in this basis is
        evidently of the required form.
\end{proof}

\begin{cor}
	\label{cor_well_known}
	If $V$ is an indecomposable $2$-dimensional crystalline representation
	of $G_{\Q_p}$ over $E$ with distinct Hodge--Tate weights $(s,r)$,
	for which $\Dcris(V) \cong D(r,s,t_0,d_0)$,
	then the formal crystalline deformation ring of $V$ is naturally isomorphic
	to $E[[t-t_0,d-d_0]].$
\end{cor}
\begin{proof}This is immediate from Lemma~\ref{also_well_known}, and
  the fact that $V\mapsto\Dcris(V)$ is an equivalence of categories.
  \end{proof}

Suppose that $\rbar$ has a crystalline lift of Hodge type~$\sigma$. By~\cite[Thm.\ 4.1]{Gpatch} there is a natural $E$-algebra
homomorphism $\eta:\cH(\sigma)\to R_p(\sigma)[1/p]$ interpolating
a normalized local Langlands correspondence $r_p$ (introduced in Section \ref{subsec:notation}).
In order to characterise this map, one considers the composite
$$\cH(\sigma)\buildrel \eta\over \longrightarrow R_p(\sigma)[1/p]
\hookrightarrow \bigl(R_p(\sigma)\bigr)^{\an},$$
where $\bigl(R_p(\sigma)\bigr)^{\an}$ denotes the ring of rigid analytic
functions on the rigid analytic generic fibre of $\Spf R_p(\sigma)$.
Over $\bigl(R_p(\sigma)\bigl)^{\an}$, we may consider the universal
filtered $\varphi$-module, and the underlying universal Weil group
representation (given by forgetting the filtration).  The trace
and determinant of Frobenius on this representation are certain
elements of $\bigl(R_p(\sigma)\bigl)^{\an}$ (which in fact
lie in $R_p(\sigma)[1/p]$), and $\eta$ is characterised by the fact
that it identifies appropriately chosen generators of $\cH(\sigma)$
with this universal trace and determinant.

It is straightforward to give explicit formulas for these generators
of $\cH(\sigma)$, but we have found it interesting (in part with an eye
to making arguments in more general contexts) to also derive the facts
that we need without using such explicit formulas.

Regarding explicit formulas, we have the following result.

\begin{prop}
  \label{prop: explicit local Langlands for GL_2(Q_p)}The elements $\eta(S),\eta(T)\in
  R_p(\sigma)[1/p]$ are characterised by the following property:
  if  $x:R_p(\sigma)[1/p]\to \Qpbar$ is an $E$-algebra morphism, and $V_x$ is the
corresponding two-dimensional $\Qpbar$-representation of~$G_{\Qp}$,
then $x(\eta(T))=p^{a+b}t$,
and $x(\eta(S))=p^{2a+b-1}d$, where $t$, $d$ are respectively the
trace and the determinant of~$\varphi$ on $\Dcris(V_x)$. \end{prop}\begin{proof} Lemma \ref{also_well_known} implies that there are uniquely determined $t, d\in \Qpbar$, such that
$\Dcris(V_x)\cong D(r,s,t,d)$, where $r= -a-b$ and $s=1-a$. The Weil--Deligne representation associated
to $D(r,s,t,d)$ is an unramified $2$-dimensional representation of $W_{\Qp}$, on which the geometric Frobenius
$\Frob_p$ acts by the matrix of crystalline Frobenius on $D(r,s,t,d)$, which is $\bigl( \begin{smallmatrix} 0 & -d\\ 1 & t\end{smallmatrix}
\bigr)$. Thus $$\WD(\Dcris(V_x))= \rec_p(\chi_1)\oplus \rec_p(\chi_2),$$ where $\chi_1, \chi_2:\Qptimes\rightarrow \Qpbartimes$
are unramified characters, such that $\chi_1(p)+\chi_2(p)= t$ and $\chi_1(p)\chi_2(p)=d$.

If $\pi= (\Ind_B^G |\centerdot| \chi_1 \otimes \chi_2)_{\sm}$ then $\pi\otimes |\det|^{-1/2}\cong \iota_B^G(\chi_1\otimes\chi_2)$,
where $\iota_B^G$ denotes smooth normalized parabolic induction. Then
$$ r_p(\pi)=
\rec_p ( \iota_B^G (\chi_1 \otimes \chi_2))= \rec_p(\chi_1)\oplus \rec_p(\chi_2).$$
The action of $\cH(\Eins)$ on $\pi^K$ is given by sending $[K \bigl( \begin{smallmatrix} p & 0 \\ 0 & 1 \end{smallmatrix} \bigr) K]$
to $p|p| \chi_1(p)+ \chi_2(p)= t$ and  $[K \bigl( \begin{smallmatrix} p & 0 \\0 & p \end{smallmatrix} \bigr) K]$
to  $|p|\chi_1(p)\chi_2(p)= p^{-1} d$. By ~\cite[Thm.\ 4.1]{Gpatch} and the fact that the evident isomorphism between $\pi^K=\Hom_K(\Eins,\Eins\otimes \pi)$ and $\Hom_K(\sigma,\sigma\otimes \pi)$ is equivariant with respect to the actions by $\cH(\Eins)$ and $\cH(\sigma)$ via the isomorphism \eqref{eqn: Hecke algebra iso}, we see that
\numequation\label{oh_yes_baby}
x(\eta(T))= p^{-r} t = p^{a+b} t, \quad x(\eta(S))= p^{-r-s} d= p^{2a+b-1} d.
\end{equation}
Since $R_p(\sigma)[1/p]$ is a reduced Jacobson ring, the formulas determine $\eta(T)$ and $\eta(S)$ uniquely.
\end{proof}

\begin{cor}
  \label{cor: S T in normalisation}$\eta(S)$ and $\eta(T)$ are
  contained in the normalisation of $R_p(\sigma)$ in $R_p(\sigma)[1/p]$.
\end{cor}
\begin{proof}It follows from \eqref{oh_yes_baby} and Lemma \ref{well_known}
that for all closed points
  $x:R_p(\sigma)[1/p]\to \Qpbar$, we have $x(\eta(S)),
  x(\eta(T))\in\Zpbar$. The result follows from~\cite[Prop.\ 7.3.6]{deJ}.
\end{proof}

A key fact that we will use, which is special to our context of
$2$-dimensional crystalline representations of $G_{\Q_p}$,
is that the morphism
$$(\Spf R(\sigma))^{\an} \to \bigl(\Spec \cH(\sigma)\bigr)^{\an}$$
induced by $\eta$ is an open immersion of rigid analytic spaces, where the superscript
$\an$ signifies the associated rigid analytic space.
We prove this statement (in its infinitesimal form) in the following result.

\begin{lem}\label{fibre_hecke} Let $\sigma=\sigma_{a,b}$ with $a\in \Z$ and $b\in \Z_{\ge 0}$. Then
 $$\dim_{\kappa(y)} \kappa(y)\otimes_{\cH(\sigma)} R_p(\sigma)[1/p] \le 1, \quad \forall y\in \mSpec \cH(\sigma).$$
\end{lem}
\begin{proof} Let us assume that $A:=  \kappa(y)\otimes_{\cH(\sigma)} R_p(\sigma)[1/p]$ is non-zero.
If  $x, x'\in \mSpec A$ then  Frobenius on $\Dcris(V_x)$ and $\Dcris(V_{x'})$ will have the same trace and determinant (since,
by Proposition \ref{prop: explicit local Langlands for GL_2(Q_p)},
these are determined by the images of $T$ and $S$ in $\kappa(y)$);
denote them by $t$ and $d$.
It follows from Lemma \ref{well_known} that $\Dcris(V_x)\cong \Dcris(V_{x'})$ and hence $x=x'$.
Since $D(r,s,t,d)$ can be
constructed over $\kappa(y)$ (as $t$ and $d$ lie in $\kappa(y)$),
so can $V_x$ and thus $\kappa(x)=\kappa(y)$.
To complete the proof of the lemma
it is enough to show that the map $\mm_y \rightarrow \mm_x/\mm_x^2$ is surjective. Since we know that $R_p(\sigma)[1/p]$ is a regular ring of dimension $2$ by \cite[Thm. 3.3.8]{kisindefrings}, it is enough to construct a $2$-dimensional family of deformations
of $\Dcris(V_x)$ to the ring of dual numbers $\kappa(y)[\epsilon]$, which induces a non-trivial deformation
of the images of $S, T$. That this is possible is immediate from
Corollary~\ref{cor_well_known} and Proposition~\ref{prop: explicit local Langlands for GL_2(Q_p)}. \end{proof}

\begin{prop}\label{prop: generic fibre of cristabelline deformation ring}  Let $\sigma=\sigma_{a,b}$ with $a\in \Z$ and $b\in \Z_{\ge 0}$. Let $y\in \mSpec \cH(\sigma)$ be the image of $x\in \mSpec R_p(\sigma)[1/p]$ under the morphism
	induced by $\eta: \cH(\sigma)\rightarrow
R_p(\sigma)[1/p]$. Then $\eta$ induces an isomorphism of completions:
$$ \widehat{\cH(\sigma)}_{\mm_y}\overset{\cong}{\longrightarrow} \widehat{R_p(\sigma)[1/p]}_{\mm_x}.$$
\end{prop}
\begin{proof}
This can be proved by explicit computation, taking into account
Corollary~\ref{cor_well_known} and
Proposition~\ref{prop: explicit local Langlands for GL_2(Q_p)}.

We can also deduce it in more pure thought manner as follows:
	Since $\cH(\sigma)\cong E[T, S^{\pm 1}]$ by Proposition \ref{prop: explicit description of the Hecke algebra of a principal series} and $R_p(\sigma)[1/p]$ is a regular ring of dimension $2$ as in the preceding proof,
both completions are regular rings of dimension $2$. It follows from Lemma \ref{fibre_hecke} that $\kappa(y)=\kappa(x)$ and the map induces a surjection on tangent spaces. Hence the map is an isomorphism.
\end{proof}

If $0\le b\le p-1$, then $\sigma_{a,b}$ has a unique (up to homothety) $K$-invariant lattice $\sigmao_{a,b}$, which
is isomorphic to $\det^a \otimes \Sym^b \cO^2$ as a $K$-representation.
We let $\sigmabar_{a,b}$ be its reduction modulo $\varpi$. Then $\sigmabar_{a,b}$ is
the absolutely irreducible $\F$-representation
$\det^a\otimes\Sym^{b}\F^2$ of $\GL_2(\Fp)$; note that every
(absolutely) irreducible $\F$-representation of $\GL_2(\Fp)$ is of
this form for some uniquely determined $a,b$ with $0\le a<p-1$. We
refer to such representations as \emph{Serre weights}.

If $\sigmabar=\sigmabar_{a,b}$ is a Serre weight with the property that $\rbar$ has a
lift $r:G_{\Qp}\to\GL_2(\Zpbar)$ that is crystalline of Hodge type $\sigma=\sigma_{a,b}$,
then we say that $\sigmabar$ is a Serre weight of~$\rbar$.

Again we consider $\sigma=\sigma_{a,b}$ with any $a\in \Z$ and $b\in \Z_{\ge 0}$. Let
$$\sigma^{\circ}:=
\det{}\!^a\otimes\Sym^{b}\cO^2 ,\quad \sigmabar:=  \det{}\!^a\otimes\Sym^{b} \F^2$$ so that $\sigma^{\circ}/\varpi=\sigmabar$.
We let
$$ \cH(\sigmao):=\End_G(\cInd_K^G \sigmao), \quad
\cH(\sigmabar):=\End_G(\cInd_K^G \sigmabar).$$
Note that $\cH(\sigmao)$ is $p$-torsion free, since $\cInd_K^G \sigmao$ is.

\begin{lem}\label{lem: reducing Hecke algebra mod p}
	{\em (1)}
	For any $\sigma$,
	there is a natural isomorphism $\cH(\sigma)\cong \cH(\sigmao)[1/p]$,
	and a natural inclusion
	$\cH(\sigmao)/\varpi \hookrightarrow \cH(\sigmabar).$
	Furthermore, the $\cO$-subalgebra $\cO[S^{\pm 1}, T]$ of $\cH(\sigma)$
	is contained in $\cH(\sigmao)$.

{\em (2)} If, in addition
	$\sigma=\sigma_{a,b}$ with $0\le b\le p-1$,
	then $\cO[S^{\pm 1}, T] = \cH(\sigmao)$,
	and there is
	a natural isomorphism $\cH(\sigmao)/\varpi \cong \cH(\sigmabar).$
\end{lem}
\begin{proof}The isomorphism of~(1) follows immediately from the fact  that $\cInd_K^G \sigmao$ is a finitely generated $\OO[G]$-module.
To see the claimed inclusion, apply
$\Hom_G(\cInd_K^G \sigmao, -)$ to the exact sequence
$$0\rightarrow \cInd_K^G \sigmao \overset{\varpi}{\rightarrow} \cInd_K^G \sigmao\rightarrow \cInd_{K}^G \sigmabar\rightarrow 0$$
so as to obtain an injective map
$$\cH(\sigmao)/\varpi\hookrightarrow  \Hom_G(\cInd_K^G \sigmao, \cInd_K^G \sigmabar)\cong \cH(\sigmabar).$$
	To see the final claim of~(1), we recall that
  from \eqref{eqn: Hecke algebra iso} and Frobenius reciprocity we have natural isomorphisms
  $$\cH(\triv)\simeq \cH(\sigma) \simeq \Hom_K(\sigma,\cInd^G_K\sigma);$$
  the image of $\phi\in \cH(\triv)$ under the composite map sends
   $v\in \sigma$ to the function $g\mapsto \phi(g^{-1})\sigma(g)v$.
  A direct computation of the actions of $S,T$ on the standard basis of~$\sigma^\circ$ then verifies
that $S^{\pm 1}$ and $T$
  lie in the $\OO$-submodule $\Hom_K(\sigma^\circ,\cInd^G_K\sigma^\circ)$
  of $\cH(\sigma)$.

To prove~(2), we note that it
follows from \cite[\S 2]{breuil2} and \cite{barthel-livne}
that the composite
$\F[S^{\pm 1},T] \to \cH(\sigmao)/\varpi \to \cH(\sigmabar)$ is
an isomorphism.  Since the second of these maps is injective,
by~(1), we conclude that each of these maps is in fact an isomorphism,
confirming the second claim of~(2).
Furthermore, this shows that
 the inclusion $\cO[S^{\pm 1},T] \hookrightarrow \cH(\sigmao)$
of~(1) becomes an isomorphism both after reducing modulo $\varpi$
as well as after inverting $\varpi$ (because
$\cH(\sigma)$ is generated by~$S^{\pm 1}$ and~$T$ by
Proposition~\ref{prop: explicit description of the Hecke algebra of a
  principal series}). Thus it is an isomorphism, completing the proof of~(2).
\end{proof}

The following lemma is well known, but for lack of a convenient
reference we sketch a proof.\begin{lem}
  \label{lem:rbar has a unique Serre weight, and the deformation ring is smooth, and given by Tp}
  Assume that $\rbar$ satisfies
Assumption~{\em \ref{assumption: rbar is generic enough}}.
Then $\rbar$ has at most two Serre weights. Furthermore, if we let
$\sigmabar=\sigmabar_{a,b}$ be a Serre weight of $\rbar$ then the following hold:
\begin{enumerate}

\item The deformation ring $R_p(\sigma)$ is formally smooth of relative dimension $2$ over~$\cO$.
\item The morphism of $E$-algebras $\eta: \cH(\sigma)\to R_p(\sigma)[1/p]$ induces a
  morphism of $\cO$-algebras $\cH(\sigma^\circ)\to R_p(\sigma)$.

\item The character  $\omega^{1-2a-b}\det\rbar$ is unramified, and if we let
\begin{itemize}
\item $\mu=(\omega^{1-2a-b}\det\rbar)(\Frob_p)$, and
\item if $\rbar$ is irreducible, then $\lambda=0$, and
\item if $\rbar$ is reducible, then we can write $\rbar\cong \omega^{a+b}\otimes
\begin{pmatrix}
  \chibar_1&*\\0&\chibar_2\omega^{-b-1}
\end{pmatrix}$
for unramified characters $\chibar_1$, $\chibar_2$, and let
$\lambda=\chibar_1(\Frob_p)$, \end{itemize}
then the composition $$\alpha: \cH(\sigmao)\rightarrow R_p(\sigma)\rightarrow \F$$
 maps $T\mapsto \lambda$ and $S\mapsto \mu$.

\item Let $\widehat{\cH(\sigmao)}$ be the completion of $\cH(\sigmao)$ with respect to the kernel of $\alpha$. Then the map $\cH(\sigmao)\rightarrow R_p(\sigma)$
induces an isomorphism of local $\OO$-algebras
$\widehat{\cH(\sigmao)}\overset{\cong}{\rightarrow} R_p(\sigma)$. In
coordinates, we have
$R_p(\sigma)=\cO[[S-\widetilde{\mu},T-\widetilde{\lambda}]]$, where
the tilde denotes the Teichm\"uller lift.
\item If we set
$\pi:=(\cInd_K^G\sigmabar)\otimes_{\cH(\sigmabar),\alpha}\F$,
then~$\pi$ is an absolutely irreducible representation of~$G$, and
is independent of the choice of Serre weight~$\sigmabar$
of~$\rbar$.
\end{enumerate}

\end{lem}
\begin{proof}
The claim that $\rbar$ has at most two Serre weights is immediate from the proof
of~\cite[Thm.\ 3.17]{bdj}, which explicitly describes the Serre weights of
~$\rbar$. Concretely, in the case at hand these weights are as
follows (see also the discussion of~\cite[\S3.5]{emerton2010local},
which uses the same conventions as this paper). If~$\rbar$ is irreducible, then we may write \[\rbar|_{I_{\Qp}}\cong \omega^{m-1}\otimes
\begin{pmatrix}
  \omega_2^{n+1}&0\\0&\omega_2^{p(n+1)}
\end{pmatrix},
\]where $\omega_2$ is a fundamental character of niveau~$2$, and $0\le m<p-1$,
$0\le n\le p-2$. Then the Serre weights of~$\rbar$ are $\sigmabar_{m,n}$ and
$\sigmabar_{m+n,p-1-n}$ (with $m+n$ taken modulo $p-1$). If ~$\rbar$ is reducible, then  we may write \[\rbar|_{I_{\Qp}}\cong \omega^{m+n}\otimes
\begin{pmatrix}
  1 &*\\0&\omega^{-n-1}
\end{pmatrix},
\]where $0\le m<p-1$, $0\le n<p-1$. Then (under Assumption~\ref{assumption: rbar is generic enough}) if $n\ne 0$, the
unique Serre weight of~$\rbar$ is $\sigmabar_{m,n}$, while if $n=0$,
then $\sigmabar_{m,0}$ and $\sigmabar_{m,p-1}$ are the two Serre
weights of~$\rbar$.

Part~(2) follows from~(1) by Lemma \ref{lem: reducing Hecke algebra mod p}~(2)
and Corollary~\ref{cor: S T in normalisation}.
We prove parts~(1), (3) and~(4) simultaneously. If $\sigmabar$ is not of
the form $\sigmabar_{a,p-1}$, the claims about $R_p(\sigma)$  are a standard consequence of (unipotent)
Fontaine--Laffaille theory; for example, the irreducible case with $\cO=\Zp$ is~\cite[Thm.\
B2]{MR1363495}, and the reducible case follows in the same way. The
key point is that the corresponding weakly admissible modules are
either reducible, or are uniquely determined by the trace and determinant
of~$\varphi$, by Lemma~\ref{well_known}. Concretely, if $\rbar$ is irreducible, then the
crystalline lifts of $\rbar$ of Hodge type $\sigma_{a,b}$  correspond exactly
to the weakly
admissible modules $D(-(a+b),1-a,t,d)$ where  $v_p(t)>-a-b$ and  $\overline{p^{2a+b-1}d}=\mu$. The claimed description of the deformation ring then follows.

Similarly, if $\rbar$ is reducible, then it follows from Fontaine--Laffaille
theory (and Assumption~\ref{assumption: rbar is generic enough}) that any
crystalline lift of Hodge type~$\sigmabar_{a,b}$ is necessarily reducible and
indecomposable, and one finds that these crystalline lifts correspond precisely
to those weakly admissible modules with $D(-(a+b),1-a,t,d)$ where
$v_p(t)=-a-b$, $\overline{p^{a+b}t}=\lambda$, and  $\overline{p^{2a+b-1}d}=\mu$.

This leaves only the case that $\sigmabar$ is of
the form $\sigmabar_{a,p-1}$. In this case the result is immediate from the
main result of~\cite{berger-li-zhu}, which shows that the above
description of the weakly admissible modules continues to hold.

Finally, (5) is immediate from the main results
of~\cite{barthel-livne,breuil1}, together with  the explicit description
of~$\sigmabar$, $\lambda$, established above. More precisely, in the
case that $\rbar$ is irreducible, the absolute irreducibility of~$\pi$
is~\cite[Thm.\ 1.1]{breuil1}, and its independence of the choice
of~$\sigmabar$ is~\cite[Thm.\ 1.3]{breuil1}. If $\rbar$ is reducible
and has only a single Serre weight, then the absolute irreducibility
of~$\pi$ is~\cite[Thm.\ 33(2)]{barthel-livne}. In the remaining
case that $\rbar$ is reducible and has two Serre weights, then
Assumption~\ref{assumption: rbar is generic enough} together with the
explicit description of $\lambda,\mu$ above implies that
$\lambda^2\ne\mu$, and the absolute irreducibility of~$\pi$ is
again~\cite[Thm.\ 33(2)]{barthel-livne}. The independence of~$\pi$
of the choice of~$\sigmabar$ is~\cite[Cor.\ 36(2)(b)]{barthel-livne}. \end{proof}

\begin{rem}\label{rem: pi is not special or one dimensional}
  It follows from the explicit description of~$\pi$ that it is either
  a principal series representation or supersingular, and neither
  one-dimensional nor an element of the special series. (This would no
  longer be the case if we allowed $\rbar$ to be a twist of an
  extension of the trivial character by the mod~$p$ cyclotomic
  character, when in fact $\pi$ would be an extension of a
  one-dimensional representation and a special
  representation, which would also depend on the Serre weight if $\rbar$ is peu ramifi\'e.)\end{rem}
\begin{rem}
  \label{rem: comparting our convention on pi to the literature}If
  $\pi$ has central character~$\psi$, then
  $\det\rbar=\psi\omega^{-1}$.\end{rem}

\section{Patched modules and arithmetic actions}\label{sec: recalling
  from Gpatching}We now introduce the notion of an arithmetic action of (a power series ring over)
$R_p$ on an $\cO[G]$-module. It is not obvious from the definition
that any examples exist, but we will explain later in this section how
to deduce the existence of an example from the results
of~\cite{Gpatch} (that is, from the Taylor--Wiles patching
method). The rest of the paper is devoted to showing a uniqueness
result for such actions, and thus deducing that they encode the
$p$-adic local Langlands correspondence for~$\GL_2(\Qp)$. We
anticipate that the axiomatic approach taken here will be useful in
other contexts (for example, for proving local-global compatibility in
the $p$-adic Langlands correspondence for~$\GL_2(\Qp)$ in global
settings other than those considered in~\cite{emerton2010local} or~\cite{Gpatch}).

\subsection{Axioms}\label{subsec: axioms for arithmetic action}Fix an integer $d\ge 0$, and set
$R_\infty:=R_p\widehat{\otimes}_{\cO}\cO[[x_1,\dots,x_d]]$. Then an
\emph{$\cO[G]$-module with an arithmetic action of $R_\infty$} is by definition a non-zero $R_\infty[G]$-module $M_\infty$
satisfying the following axioms.
\begin{itemize}
\item[(AA1)] $M_\infty$ is a finitely generated
  $R_\infty[[K]]$-module.
\item[(AA2)] $M_\infty$ is projective in the category of pseudocompact
  $\cO[[K]]$-modules.
\end{itemize}
Let $\sigma^\circ$ be a $K$-stable
  $\cO$-lattice in
  $\sigma=\sigma_{a,b}$. Set
  \[M_\infty(\sigma^\circ):=\left(\Hom^{\mathrm{cont}}_{\cO[[K]]}(M_\infty,(\sigma^\circ)^d)\right)^d, \]where
  we are considering continuous homomorphisms for the profinite
  topology on $M_\infty$ and the $p$-adic topology on
  $(\sigma^\circ)^d$. This is a finitely generated $R_\infty$-module
  by (AA1) and~\cite[Cor.\ 2.5]{paskunasBM}.
\begin{itemize}
\item[(AA3)]For any $\sigma$, the action of $R_\infty$ on
  $M_\infty(\sigma^\circ)$ factors through
  $R_\infty(\sigma):=R_p(\sigma)[[x_1,\dots,x_d]]$. Furthermore,
  $M_\infty(\sigma^\circ)$ is maximal Cohen--Macaulay over
  $R_\infty(\sigma)$, and the $R_\infty(\sigma)[1/p]$-module
  $M_\infty(\sigma^\circ)[1/p]$ is locally free of rank one over its
  support.
\end{itemize}
For each $\sigma^\circ$, we have a natural action
of $\cH(\sigma^\circ)$ on $M_\infty(\sigma^\circ)$, and thus of $\cH(\sigma)$ on $M_\infty(\sigma^\circ)[1/p]$.

\begin{itemize}
\item[(AA4)] For any $\sigma$, the action of $\cH(\sigma)$ on
  $M_\infty(\sigma^\circ)[1/p]$ is given by the
  composite \[\cH(\sigma)\stackrel{\eta}{\to}R_p(\sigma)[1/p]\to R_p(\sigma)[[x_1,\dots,x_d]][1/p],\]
  where $\cH(\sigma)\stackrel{\eta}{\to}R_p(\sigma)[1/p]$ is defined in \cite[Thm. 4.1]{Gpatch}.
   \end{itemize}
\begin{rem}\label{rem:axioms are natural}
  While these axioms may appear somewhat mysterious, as we will see in
  the next subsection they arise very
  naturally in the constructions of~\cite{Gpatch}. (Indeed, those
  constructions give modules~$M_\infty$ satisfying obvious analogues of the above
  conditions for $\GL_n(K)$ for any finite extension $K/\Qp$; however,
  our arguments in the rest of the paper will only apply to the case
  of~$\GL_2(\Q_p)$.)

In these examples, axioms (AA1) and (AA2) essentially follow from the
facts that spaces of automorphic forms are finite-dimensional, and
that the cohomology of zero-dimensional Shimura varieties is
concentrated in a single degree (degree zero). Axioms (AA3) and (AA4)
come from the existence of Galois representations attached to
automorphic forms on unitary groups, and from local-global
compatibility at~$p$ for automorphic forms of level prime to~$p$. 
\end{rem}
The following remark explains how axiom~(AA4) is related to Breuil's
original formulation of the $p$-adic Langlands correspondence in terms
of unitary completions of locally algebraic vectors; see also the
proof of Proposition~\ref{AA4_holds} below.
\begin{rem}
  \label{rem: AA4 alternatives}Axiom (AA4) is, in the presence of
  axioms~(AA1)-(AA3), equivalent to an alternative axiom~(AA4'), which
  expresses a pointwise compatibility with the classical local
  Langlands correspondence, as we now explain. Write~$R_\infty(\sigma):=R_p(\sigma)[[x_1,\dots,x_d]]$. If~$y$ is a maximal ideal
  of~$R_\infty(\sigma)[1/p]$ which lies in the support of $M_\infty(\sigma^\circ)[1/p]$, then we
  write \[\Pi_y:=\Hom_\cO^\cont\bigl(M_\infty \otimes_{R_\infty,y}\cO_{\kappa(y)},E\bigr).\]
We write~$x$ for  the corresponding maximal ideal
  of~$R_p(\sigma)[1/p]$, $r_x$ for the deformation of~$\rbar$ corresponding to~$x$,
and set
$\pi_{\sm}(r_x):=r_p^{-1}(\WD(r_x)^{F-\mathrm{ss}})$, which is the smooth representation of $G$
corresponding to the Weil--Deligne representation associated to $r_x$ by the classical Langlands
correspondence~$r_p$ (normalised as in~Section~\ref{subsec:notation}).
We write $\pi_{\alg}(r_x)$ for the algebraic
representation of $G$ whose restriction to $K$ is equal to $\sigma$.
\begin{itemize}
\item[(AA4')] For any $\sigma$ and for any $y$ and $x$ as above,
  there is a non-zero $G$-equivariant map
  \[\pi_{\sm}(r_x)\otimes \pi_{\alg}(r_x)
\rightarrow \Pi_y^{\lalg}.\]
\end{itemize}
That (AA1)-(AA4) imply (AA4') is a straightforward consequence of the
defining property of the map~$\eta$. Conversely, assume (AA1)-(AA3)
and~(AA4'), and write $$\barR_\infty(\sigma):=R_\infty(\sigma)/\mathrm{Ann}(M_\infty(\sigma^\circ)).$$ It follows from (AA3) that  the
natural map
$\barR_\infty(\sigma)[1/p]\to\End_{R_\infty(\sigma)[1/p]}(M_\infty(\sigma^\circ)[1/p])$
is an isomorphism, as it is injective and the cokernel is not supported on any maximal ideal of $R_{\infty}(\sigma)[1/p]$.
In particular the action
of~$\cH(\sigma^\circ)$ on~$M_\infty(\sigma^\circ)$ induces a
homomorphism $\eta':\cH(\sigma)\to \barR_\infty(\sigma)[1/p]$. We have to
show that this agrees with the map induced
by~$\eta$.

It follows from (AA4') and the defining property of $\eta$ that $\eta$ and $\eta'$ agree modulo
every maximal ideal of $R_p(\sigma)[1/p]$ in the support of $M_{\infty}(\sigmao)$.
It follows from (AA3) that $\barR_\infty(\sigma)[1/p]$ is  a union of irreducible components
of $R_{\infty}(\sigma)[1/p]$. Since $R_{\infty}(\sigma)[1/p]$ is reduced we conclude that
$\barR_\infty(\sigma)[1/p]$ is reduced, thus the intersection of all maximal ideals  is equal to zero.  Hence (AA4) holds.
\end{rem}

\subsection{Existence of a patched module $M_\infty$}\label{subsec: M infty}

We now briefly recall some of
main results of~\cite{Gpatch}, specialised to the case of
two-dimensional representations. We emphasise that these results use only the
Taylor--Wiles--Kisin patching method, and use nothing about the
$p$-adic Langlands correspondence
for $\GL_2(\Qp)$. (We should  perhaps remark, though, that we do make implicit use of the
results of~\cite{BLGGT} in the globalisation part of the argument, and thus of
the Taylor--Wiles--Kisin method for unitary groups of rank~$4$, and
not just for~$U(2)$.) We freely use the notation of~\cite{Gpatch}.

Enlarging $\F$ if necessary, we see from~\cite[Lem.\
2.2]{Gpatch} that the hypotheses on $\rbar$ at the start
of~\cite[\S 2.1]{Gpatch} are automatically satisfied. We fix the choice of
weight~$\xi$ and inertial type~$\tau$ in~\cite[\S 2.3]{Gpatch} in
the following way: we take~$\tau$ to be trivial, and we take~$\xi$ to be the
weight corresponding to a Serre weight of~$\rbar$, as in Lemma~\ref{lem:rbar has a unique Serre weight, and the deformation ring is smooth,
  and given by Tp}.

With this choice, the modification of the
Taylor--Wiles--Kisin method carried out in~\cite[\S2.6]{Gpatch}
produces for some $d>0$ an $R_\infty$-module $M_\infty$ with an action
of~$G$. (Note that for our choice of $\rbar$, $\xi$ and
$\tau$, the various framed deformation rings appearing
in~\cite{Gpatch} are formal power series rings over~$\cO$, and
the framed deformation ring of~$\rbar$ is formally smooth over~$R_p$, so all of
these rings are absorbed into the power series ring
$\cO[[x_1,\dots,x_d]]$. The module~$M_\infty$ is patched from the
cohomology of a definite unitary group over some totally real field in
which~$p$ splits completely.)

This $R_\infty[G]$-module automatically satisfies the axioms
(AA1)--(AA4) above. Indeed, (AA1) and (AA2) follow from~\cite[Prop.\ 2.10]{Gpatch}, and (AA3) follows from~\cite[Lem.\ 4.17(1),
4.18(1)]{Gpatch}. Finally, (AA4) is~\cite[Thm.\ 4.19]{Gpatch}.

\section{Existence and uniqueness of arithmetic actions}\label{sec:
  proof that arithmetic action is unique}We fix an $\cO[G]$-module
$M_\infty$ with an arithmetic action of $R_\infty$ in the sense of Section~\ref{subsec: axioms for arithmetic action}.

\subsection{Serre weights and cosocles}\label{subsec: Serre weights theory}

Now let $\sigmabar=\sigmabar_{a,b}$ be a Serre weight, and let
$\sigma^\circ$ be a $K$-stable $\cO$-lattice in $\sigma_{a,b}$, so
that $\sigma^\circ/\varpi\sigma^\circ=\sigmabar$. We define
$M_\infty(\sigmabar)=\Hom^{\mathrm{cont}}_{\cO[[K]]}(M_\infty,(\sigmabar)^\vee)^{\vee}$,
so that by (AA2) we have
$M_\infty(\sigmabar)=M_\infty(\sigma^\circ)/\varpi
M_\infty(\sigma^\circ)$. By definition, the deformation ring
$R_p(\sigma)=R_p(\sigma_{a,b})$ is non-zero if and only if
$\sigmabar$ is a Serre weight of~$\rbar$. Set $R_{\infty}(\sigma)=R_p(\sigma)[[x_1,\dots,x_d]]$.

We let~$\pi$ denote the
absolutely irreducible smooth $\F$-representation of~$G$ associated to $\rbar$ via 
Lemma~\ref{lem:rbar has a unique Serre weight, and the deformation ring is smooth, and given by Tp} (5).

\begin{prop}\label{prop: abstract Serre weight freeness and Hecke
    compatibility}
  \begin{enumerate}
  \item We have $M_\infty(\sigma^\circ)\ne 0$ if and only if
    $\sigmabar$ is a Serre weight of~$\rbar$, in which case
    $M_{\infty}(\sigma^{\circ})$ is a free $R_{\infty}(\sigma)$-module
    of rank one.\item If $\sigmabar$ is a Serre weight of~$\rbar$ then the action of
  $\cH(\sigmabar)$ on $M_{\infty}(\sigmabar)$ factors through the
  natural map $R_p(\sigma)/\varpi\to R_\infty(\sigma)/\varpi$, and
  $M_{\infty}(\sigmabar)$ is a flat $\cH(\sigmabar)$-module.
  \item If $\pi'$ is an  irreducible  smooth  $\F$-representation of~$G$
then we have
  \[\Hom_G(\pi',M_\infty^\vee)\ne 0\]
  if and only if  $\pi'$ is isomorphic to $\pi$.
  \end{enumerate}
\end{prop}
\begin{proof}It follows from (AA3) that $M_{\infty}(\sigma^{\circ})\neq 0$ only if $R_p(\sigma)[1/p]\neq 0$, which is equivalent
  to $\sigmabar$ being a Serre weight of $\rbar$.
   In this case, since $\sigma=\sigma_{a,b}$ with $0\le b\le p-1$, $R_\infty(\sigma)$
  is formally smooth over~$\cO$ by Lemma~\ref{lem:rbar has a unique Serre weight, and the deformation ring is smooth,
   and given by Tp}, so it follows from~(AA3) and the
  Auslander--Buchsbaum theorem that $M_{\infty}(\sigmao)$ is a free
  $R_{\infty}(\sigma)$-module of finite rank, and that $M_{\infty}(\sigmao)[1/p]$ is a locally free
  $R_{\infty}(\sigma)[1/p]$-module of rank $1$. Thus
  $M_{\infty}(\sigmao)$ is free of rank one over
  $R_{\infty}(\sigma)$.
  This proves the ``only if'' direction of~(1).

For~(2), note that $M_\infty(\sigma^\circ)\ne 0$ if and only if
$M_\infty(\sigmabar)\ne 0$, so we may assume that~ $M_\infty(\sigma^\circ)\ne 0$.  The first part of (2) follows from~(AA4) together with
  Lemmas~\ref{lem:rbar has a unique Serre weight, and the deformation ring is smooth, and given by Tp}~(2)
  and~\ref{lem: reducing Hecke algebra mod p}~(2).
For the remaining part of~(2), note that $R_p(\sigma)/\varpi$ is flat over~$\cH(\sigmabar)$
by Lemmas~\ref{lem: reducing Hecke algebra mod p} and
\ref{lem:rbar has a unique Serre weight, and the deformation
  ring is smooth, and given by Tp}, and $M_\infty(\sigmabar)$ is flat
over~$R_p(\sigma)/\varpi$ by the only if part of~(1), as required.

 To prove (3) we first note that it is enough to prove the statement for absolutely irreducible $\pi'$ as we may enlarge the field $\F$.
 Let us assume that $\pi'$ is absolutely irreducible and let $\sigmabar'$ be an irreducible representation of $K$ contained in the socle
 of $\pi'$.  It follows from \cite{barthel-livne,breuil1} that the surjection $\cind_K^G \sigmabar' \twoheadrightarrow \pi'$ factors through the map
 \numequation\label{oh_tg_i_love_you_dearly}
  (\cind_K^G \sigmabar')\otimes_{\cH(\sigmabar') , \alpha'}\F\twoheadrightarrow \pi'
 \end{equation}
 where $\alpha': \cH(\sigmabar')\rightarrow \F$ is given by the action of $\cH(\sigmabar')$ on the one dimensional $\F$-vector space
 $\Hom_K(\sigmabar', \pi')$. Moreover, \eqref{oh_tg_i_love_you_dearly} is an isomorphism unless $\pi'$ is a character or special series.
 Since
 $$M_{\infty}(\sigmabar')\cong \Hom_K(\sigmabar', M_{\infty}^{\vee})^{\vee}\cong  \Hom_G(\cind_K^G \sigmabar', M_{\infty}^{\vee})^{\vee},$$
 from \eqref{oh_tg_i_love_you_dearly} we obtain a surjection of $R_{\infty}(\sigmabar')$-modules
 $$ M_{\infty}(\sigmabar')\otimes_{\cH(\sigmabar'), \alpha'} \F\cong   \Hom_G(\cind_K^G \sigmabar' \otimes_{\cH(\sigmabar'), \alpha'} \F, M_{\infty}^{\vee})^{\vee}\twoheadrightarrow \Hom_G(\pi', M_{\infty}^{\vee})^{\vee},$$
 which moreover is an isomorphism if $\pi'$ is not a character or special series. Thus if $\Hom_G(\pi', M_{\infty}^{\vee})$ is non-zero  then we deduce
 from the previous displayed expression that
 $ M_{\infty}(\sigmabar')\otimes_{\cH(\sigmabar'), \alpha'} \F  \neq 0.$
 In particular,
 $M_{\infty}(\sigmabar')\neq 0$ and hence $\sigmabar'$ is a Serre
 weight for $\rbar$ by the only if part of (1).

We claim that $\alpha'$ coincides with the morphism $\alpha$
of Lemma~\ref{lem:rbar has a unique Serre weight, and the deformation ring
    is smooth, and given by Tp}~(3)
(with $\sigmabar'$ in place of $\sigma$).
To see this,
note that by the only if part of~(1),  we have that $\bigl(R_p(\sigma')/\varpi\bigr)
\otimes_{\cH(\sigmabar'),\alpha'} \F \neq 0,$
and hence by
Lemma~\ref{lem:rbar has a unique Serre weight, and the deformation ring
    is smooth, and given by Tp}~(4),
we find that $\widehat{\cH(\sigmabar')}\otimes_{\cH(\sigmabar'),\alpha'}\F\neq 0,$
where $\widehat{\cH(\sigmabar')}$ denotes the completion of $\cH(\sigmabar')$
with respect to the kernel of the morphism $\alpha$.
This proves that $\alpha$ and $\alpha'$ coincide.

Part~(5) of Lemma \ref{lem:rbar has a unique Serre weight, and the deformation ring
    is smooth, and given by Tp}
now implies that $\pi\cong (\cind_K^G \sigmabar')\otimes_{\cH(\sigmabar') , \alpha'}\F$. Hence,
 \eqref{oh_tg_i_love_you_dearly} gives us a $G$-equivariant surjection $\pi\twoheadrightarrow \pi'$, which is an isomorphism as
 $\pi$ is irreducible.

 Conversely, it follows from from (AA1) that there is an irreducible smooth
$\F$-representation $\pi'$ of $G$ such that $\Hom_G(\pi',
M_{\infty}^{\vee})$ is non-zero; we have just seen
that~$\pi\cong\pi'$, so that $\Hom_G(\pi, M_{\infty}^{\vee})\neq 0$, as required.

Finally, suppose that~$\sigmabar$ is a Serre weight of~$\rbar$. Then
as above we have an isomorphism of $R_\infty(\sigmabar)$-modules  $$
M_{\infty}(\sigmabar)\otimes_{\cH(\sigmabar), \alpha} \F\cong
\Hom_G(\pi, M_{\infty}^{\vee})^{\vee}\ne 0$$ so that
$M_\infty(\sigmabar)\ne 0$. This completes the proof of the ``if''
direction of~(1).
\end{proof}

\subsection{Smooth and admissible representations} We record a few definitions, following Section 2
of~\cite{paskunasimage}. Let $(R,\mathfrak{m})$ be a complete local
noetherian $\cO$-algebra with residue field $\F$. Then
$\Mod_{G}^\sm(R)$ is the full subcategory of the category of
$R[G]$-modules consisting of \emph{smooth} objects. More precisely,
these are objects $V$ such
that \[V=\bigcup_{H,n}V^H[\mathfrak{m}^n],\]
where the union is taken over open compact subgroups $H\subset G$ and
over positive integers $n$.

We say that an object $V$ of $\Mod_{G}^\sm(R)$ is \emph{admissible} if
$V^H[\mathfrak{m}^n]$ is a finitely generated $R$-module for every
compact open subgroup $H\subset G$ and every $n\geq 1$. Moreover, $V$
is called \emph{locally admissible} if, for every $v\in V$, the
smallest $R[G]$-submodule of $V$ containing $v$ is admissible. We let
$\Mod_G^{\ladm}(R)$ denote the full subcategory of $\Mod_{G}^\sm(R)$
consisting of locally admissible representations.

The categories $\Mod_{G}^\sm(R)$ and $\Mod_G^{\ladm}(R)$ are abelian (see~\cite{emertonord1} for the second one) and have enough injectives.

\begin{defn}\label{injective envelope}\begin{enumerate}
  \item A monomorphism $\iota: N\hookrightarrow M$ in an abelian
    category is called \emph{essential} if, for every non-zero subobject
    $M'\subset M$, $\iota(N)\cap M'$ is non-zero.
  \item An \emph{injective envelope} of an object $N$ of an abelian category
    is an essential monomorphism $\iota: N\hookrightarrow I$ with $I$
    an injective object of the abelian category.
\end{enumerate}
\end{defn}
\noindent If they exist, injective envelopes are unique up to
(non-unique) isomorphism. By Lemma 2.3 of~\cite{paskunasimage}, the
category $\Mod_{G}^\sm(R)$ admits injective envelopes.   The category $\Mod_G^{\ladm}(R)$ also
admits injective envelopes.  (This follows from
\cite[Lem.~3.2]{unitpask} and
the fact that the inclusion of $\Mod_G^{\ladm}(R)$ into $\Mod_G^{\sm}(R)$
has a right adjoint, namely the functor to which any smooth $G$-representation
associates its maximal locally admissible subrepresentation.)

\begin{lemma}
	\label{lem:socles are essential}
	If $V$ is a locally admissible representation of $G$,
	then the inclusion $\soc_G(V) \hookrightarrow V$ is essential.
\end{lemma}
\begin{proof} Any non-zero subrepresentation of $V$ contains a non-zero
	finitely generated subrepresentation.   Thus it suffices to show that
	any non-zero finitely generated subrepresentation $W$ of $V$
	has a non-zero intersection with $\soc_G(V)$.   Since $\soc_G(V)
	\cap W = \soc_G(W)$, it suffices to show that any such subrepresentation
	has a non-zero socle.  This follows
	from the fact that every finitely generated admissible
representation of $G$ is of finite length by \cite[Thm. 2.3.8]{emertonord1}.
\end{proof}

\begin{defn}\label{projective envelope}
\begin{enumerate}
\item An epimorphism $q: M\twoheadrightarrow N$ in an abelian category
  is called \emph{essential} if a morphism $s: M'\to M$ is an epimorphism
  whenever $q\circ s$ is an epimorphism.
\item A \emph{projective envelope} of an object $N$ of an abelian category is
  an essential epimorphism $q: P\twoheadrightarrow N$ with $P$ a
  projective object in the abelian category.
\end{enumerate}
\end{defn}
\noindent  Pontryagin duality reverses arrows, so it exchanges injective and projective objects as well as injective and projective envelopes.

\subsection{Projectivity of $M_{\infty}$} Our first aim is to show that
$M_{\infty}^{\vee}$ is an injective locally admissible representation
of~$G$.

\begin{lem} $M_{\infty}^{\vee}$ is an admissible object of $\Mod^{\sm}_G(R_\infty)$,
 and thus lies in $\Mod^{\ladm}_G(\cO)$.
\end{lem}
\begin{proof} Dually it is enough by~\cite[Lem.\ 2.2.11]{emertonord1} to show that $M_{\infty}$ is a
  finitely generated $R_{\infty}[[K]]$-module, which is (AA1).
\end{proof}

\begin{lem}\label{tree} Let $\mm$ be a maximal ideal of
  $\cH(\sigmabar)$ with residue field $\kappa(\mm)$. Then
$$\Tor^{\cH(\sigmabar)}_i(\cInd_K^G \sigmabar , \kappa(\mm))=0, \quad \forall i>0.$$
\end{lem}
\begin{proof} Since the map $\F\rightarrow \Fbar$ is faithfully flat,
  we can and do assume that  $\F$ is algebraically closed. Since $\cH(\sigmabar)=\F[S^{\pm 1}, T]$, we have $\mm=(S-\mu, T-\lambda)$ for some $\mu\in \F^{\times}$, $\lambda\in \F$. Since the sequence $S-\mu, T-\lambda$ is regular in
$\cH(\sigmabar)$, the Koszul complex $K_{\bullet}$ associated to it is a resolution
of $\kappa(\mm)$ by free $\cH(\sigmabar)$-modules, \cite[Thm. 16.5 (i)]{Matsumura}.  Thus the complex $K_{\bullet}\otimes_{\cH(\sigmabar)} \cInd_K^G\sigmabar$ computes the $\Tor$-groups we are after, and to verify the claim it is enough to show that the sequence $S-\mu, T-\lambda$ is regular
on $\cInd_K^G\sigmabar$.

If $f\in \cInd_K^G \sigmabar$ then
$(S f)(g)= f(gz)$, where $z=\bigl (\begin{smallmatrix} p & 0 \\ 0 & p\end{smallmatrix}\bigr)$.
Since such an $f$ is supported only on finitely many cosets $K\backslash G$, we deduce that
the map
$$ \cInd_K^G \sigmabar \overset{S-\mu}{\longrightarrow}  \cInd_K^G \sigmabar$$
is injective. The quotient is isomorphic to $\cInd_{ZK}^G \sigmabar$, where $z$ acts on $\sigmabar$ by $\mu$.
It follows from the proof of \cite[Thm. 19]{barthel-livne} that  $\cInd_{ZK}^G \sigmabar$ is a free
$\F[T]$-module. Thus the map
$$ \cInd_{ZK}^G \sigmabar \overset{T-\lambda}{\longrightarrow}  \cInd_{ZK}^G \sigmabar$$
is injective, and the sequence  $S-\mu, T-\lambda$ is regular on
$\cInd_K^G\sigmabar$, as required.
\end{proof}

\begin{lem}\label{ext_vanish} Let $\mm$ be a maximal ideal of $\cH(\sigmabar)$. Then
$$\Ext^i_G\bigl(\kappa(\mm)\otimes_{\cH(\sigmabar)} \cInd_K^G \sigmabar, M_{\infty}^{\vee}\bigr)=0, \quad \forall i\ge 1,$$
where the $\Ext$-groups are computed in $\Mod^{\sm}_G(\OO)$.
\end{lem}

\begin{proof} We first prove that
$\Ext^i_G\bigl(\cInd_K^G \sigmabar, M_{\infty}^{\vee}\bigr)=0$, for all $i\ge 1$.
Let $M_{\infty}^{\vee} \hookrightarrow J^{\bullet}$ be an injective resolution of $M_{\infty}^{\vee}$ in $\Mod^{\sm}_G(\OO)$.
Since $$\Hom_K(\tau, J|_K)\cong \Hom_G(\cInd_K^G \tau, J)$$
and the functor $\cInd_K^G$ is exact,
the restriction of an injective object in $\Mod^{\sm}_G(\OO)$ to $K$ is injective in $\Mod^{\sm}_K(\OO)$.
Thus $(J^{\bullet})|_K$ is an injective resolution of $M_{\infty}^{\vee}|_K$ in $\Mod^{\sm}_K(\OO)$.
Since $\Hom_G(\cInd_K^G \sigmabar, J^{\bullet})\cong \Hom_K(\sigmabar, (J^{\bullet})|_K)$, we conclude that
we have natural isomorphisms
$$\Ext^i_G\bigl(\cInd_{K}^G \sigmabar, M_{\infty}^{\vee}\bigr)\cong \Ext^i_K\bigl(\sigmabar, M_{\infty}^{\vee}\bigl), \quad \forall i\ge 0.$$
Since $M_{\infty}$ is a projective $\OO[[K]]$-module by~(AA2), $M_{\infty}^{\vee}$ is injective in
$\Mod^{\sm}_K(\OO)$, and thus the $\Ext$-groups vanish as claimed.

Let $F_{\bullet}\twoheadrightarrow \kappa(\mm)$ be a resolution of $\kappa(\mm)$ by finite free $\cH(\sigmabar)$-modules. Lemma \ref{tree} implies that the complex $F_{\bullet}\otimes_{\cH(\sigmabar)} \cInd_K^G \sigmabar$ is a resolution
of  $\kappa(\mm)\otimes_{\cH(\sigmabar)} \cInd_K^G\sigmabar$ by acyclic objects for the functor $\Hom_G(\ast, M_{\infty}^{\vee})$.
We conclude that the cohomology of the complex $$\Hom_G(F_{\bullet}\otimes_{\cH(\sigmabar)} \cInd_K^G \sigmabar, M_{\infty}^{\vee})$$ computes the groups $\Ext^i_G\bigl(\kappa(\mm)\otimes_{\cH(\sigmabar)} \cInd_K^G \sigmabar, M_{\infty}^{\vee}\bigr)$. We may think of the transition maps in $F_{\bullet}$ as matrices with entries in
$\cH(\sigmabar)$. The functor $\Hom_G(\ast, M_{\infty}^{\vee})^{\vee}$ transposes these matrices twice, thus we get an isomorphism of complexes:
$$ \Hom_G(F_{\bullet}\otimes_{\cH(\sigmabar)} \cInd_K^G \sigmabar, M_{\infty}^{\vee})^{\vee}\cong
F_{\bullet}\otimes_{\cH(\sigmabar)} \Hom_G(\cInd_K^G \sigmabar, M_{\infty}^{\vee})^{\vee}\cong F_{\bullet}\otimes_{\cH(\sigmabar)} M_{\infty}(\sigmabar).$$
The above isomorphism induces a natural isomorphism
$$ \Bigl(\Ext^i_{G}\bigl(\kappa(\mm)\otimes_{\cH(\sigmabar)} \cInd_K^G \sigmabar, M_{\infty}^{\vee}\bigr)\Bigr)^{\vee}\cong \Tor^{\cH(\sigmabar)}_i\bigl(\kappa(\mm), M_{\infty}(\sigmabar)\bigr), \quad \forall i\ge 0.$$
The isomorphism implies the assertion, as $M_{\infty}(\sigmabar)$ is a flat $\cH(\sigmabar)$-module
by Proposition~\ref{prop: abstract Serre weight freeness and Hecke
    compatibility}~(2).
\end{proof}

\begin{lem}\label{rationality} Let $y: \cH(\sigmabar)\rightarrow \F'$ be a homomorphism of $\F$-algebras, where $\F'$ is a
finite field extension of $\F$. Let $\pigmonkey':= \F'\otimes_{\cH(\sigmabar), y} \cInd_K^G \sigmabar$ and
let $\pigmonkey$ be an absolutely irreducible $\F$-representation of $G$, which is either principal series or supersingular.
If $\pigmonkey$ is a subquotient of $\pigmonkey'$  then $\pigmonkey'$ is isomorphic to a direct sum of finitely many copies of $\pigmonkey$.
\end{lem}
\begin{proof} In the course of the proof we will use the following fact repeatedly: if $A$ and $B$ are $\F$-representations of $G$ and $A$ is finitely generated as an $\F[G]$-module then:
\numequation\label{extend_scalars}
\Hom_G(A, B)\otimes_{\F} \Fbar \cong \Hom_G(A\otimes_{\F}\Fbar, B\otimes_{\F}\Fbar),
\end{equation}
where $\Fbar$ denotes the algebraic closure of $\F$, see \cite[Lem. 5.1]{paskunasimage}. Then
\numequation\label{extend_scalars2}
\Fbar\otimes_{\F} \pigmonkey'\cong \Fbar\otimes_\F \F'\otimes_{\cH(\sigmabar), y}\cInd_K^G \sigmabar\cong \bigoplus_{\iota: \F'\rightarrow \Fbar}\Fbar \otimes_{\cH(\sigmabar), \iota\circ y}\cInd_K^G \sigmabar,
\end{equation}
where the sum is taken over $\F$-algebra homomorphisms $\iota: \F'\rightarrow \Fbar$. By the classification
theorems of Barthel--Livne \cite{barthel-livne} and Breuil
\cite{breuil1}, each representation  $\Fbar \otimes_{\cH(\sigmabar),
  \iota\circ y}\cInd_K^G \sigmabar$ is either irreducible, an
extension of a special series by a character, or an extension of a character by
a special series. Since $\pigmonkey$ is a subquotient of $\pigmonkey'$ by assumption, $\pigmonkey\otimes_{\F}\Fbar$ is a subquotient
of $\pigmonkey'\otimes_{\F}\Fbar$, and since $\pigmonkey$ is neither special series nor a character, we deduce that
$$\pigmonkey\otimes_{\F}\Fbar\cong \Fbar \otimes_{\cH(\sigmabar), \iota\circ y}\cInd_K^G \sigmabar,$$
for some embedding $\iota: \F'\hookrightarrow \Fbar$. For every $\tau\in \Gal(\Fbar/\F)$, we have
$$  \Fbar \otimes_{\cH(\sigmabar), \tau\circ \iota\circ y}\cInd_K^G \sigmabar\cong
\Fbar\otimes_{\Fbar, \tau}(  \Fbar \otimes_{\cH(\sigmabar), \iota\circ y}\cInd_K^G \sigmabar)
\cong \Fbar\otimes_{\Fbar, \tau} (\Fbar\otimes_{\F} \pigmonkey)\cong \Fbar \otimes_{\F} \pigmonkey.$$
Hence all the summands in \eqref{extend_scalars2} are isomorphic to $\pigmonkey\otimes_{\F} \Fbar$. It follows
from \eqref{extend_scalars} that $\pigmonkey'$ is isomorphic to a direct sum
of copies of $\pigmonkey$, as required.
\end{proof}

\begin{thm}\label{thm: Minfty is projective} $M_{\infty}^{\vee}$ is an injective object in $\Mod^{\ladm}_G(\cO)$.
\end{thm}
\begin{proof} Let $M_{\infty}^{\vee}\hookrightarrow J$ be an injective
  envelope of $M_{\infty}^{\vee}$ in $\Mod^{\ladm}_G(\OO)$.
Lemma~\ref{lem:socles are essential} shows that
the composition $\soc_G M_{\infty}^{\vee}\hookrightarrow
M_{\infty}^{\vee}\hookrightarrow J$ is an essential monomorphism,
and thus induces an isomorphism between $\soc_G M_{\infty}^{\vee}$ and
$\soc_G J$.
Proposition~\ref{prop: abstract Serre weight freeness and Hecke
    compatibility}~(3)
shows that
$\soc_G M_{\infty}^{\vee}$, and thus also $\soc_G J$,
is isomorphic to a direct sum of copies of the representation $\pi$
associated to $\rbar$ via Lemma~\ref{lem:rbar has a unique Serre weight, and the deformation ring is smooth, and given by Tp} (5).

  Let us
  assume that the quotient $J/M_{\infty}^{\vee}$ is non-zero; then there is a smooth irreducible $K$-subrepresentation $\sigmabar \subset J/M_{\infty}^{\vee}$. Let $\kappa$ be the $G$-subrepresentation of $J/M_{\infty}^{\vee}$ generated by $\sigmabar$. Since $J/M_{\infty}^{\vee}$ is locally admissible, and $\sigmabar$ is finitely generated as a $K$-representation, $\kappa$ is an admissible representation
of $G$. Thus $\Hom_G(\cInd_K^G \sigmabar, \kappa)\cong \Hom_K(\sigmabar, \kappa)$ is a finite dimensional
$\F$-vector space.

Let $\md$ be any irreducible $\cH(\sigmabar)$-submodule of $\Hom_G(\cInd_K^G \sigmabar, \kappa)$. Since 
$\md$ is finite dimensional over $\F$, Schur's lemma implies that 
$\F':=\End_{\cH(\sigmabar)}(\md)$ is a finite dimensional division algebra over $\F$. Since $\F$ is a finite field we deduce that $\F'$ is a finite field extension of $\F$.  Since $\cH(\sigmabar)$ is commutative we further deduce that 
$\md$ is a one dimensional $\F'$-vector space and thus obtain a surjective homomorphism
of $\F$-algebras $y: \cH(\sigmabar)\twoheadrightarrow \F'$. Moreover, by the construction of $\md$, we obtain a
non-zero $G$-equivariant map:
$$\pi':= \F'\otimes_{\cH(\sigmabar),y} \cInd_K^G \sigmabar \rightarrow \kappa \subset J/M_{\infty}^{\vee}.$$
Since $\Ext^1_G(\pi', M_{\infty}^{\vee})=0$ by Lemma \ref{ext_vanish}, by applying $\Hom_G(\pi', \ast)$ to
the exact sequence $0\rightarrow M_{\infty}^{\vee}\rightarrow J\rightarrow J/M_{\infty}^{\vee}\rightarrow 0$,
we obtain a short exact sequence
$$ 0\rightarrow \Hom_G(\pi', M_{\infty}^{\vee})\rightarrow \Hom_G(\pi', J)\rightarrow \Hom_G(\pi', J/M_{\infty}^{\vee})\rightarrow 0.$$
Moreover, we know that $\Hom_G(\pi', J/M_{\infty}^{\vee})$ is non-zero. Hence, $\Hom_G(\pi', J)$ is non-zero.

Fix a non-zero $G$-equivariant map $\varphi: \pi' \rightarrow J$; then
$\varphi(\pi')\cap \soc_G J\neq 0$.
Since $\soc_G J$ is isomorphic to a direct sum of copies of $\pi$, we find that
$\pi$ is an irreducible subquotient of $\pi'$.
It follows from Lemma \ref{rationality} that $\pi'$ is then isomorphic to a finite direct sum of copies of $\pi$, and so in particular is semi-simple.
As we've already noted,
the map $M_{\infty}^{\vee}\hookrightarrow J$ induces an isomorphism $\soc_G M_{\infty}^{\vee}\cong \soc_G J$, and so
the map $\Hom_G(\pi', M_{\infty}^{\vee})\rightarrow \Hom_G(\pi', J)$ is an isomorphism. This implies
$\Hom_G(\pi', J/M_{\infty}^{\vee})=0$, contradicting the assumption
$J/M_{\infty}^{\vee}\neq 0$. Hence $M_{\infty}^{\vee}=J$ is injective,
as required.
\end{proof}

\subsection{Removing the patching variables}\label{subsec: Vytas' argument to
  factor out patching variables} We now show that we can pass from
$M_\infty$ to an arithmetic action of~$R_p$ on a projective envelope
of~$\pi^\vee$, where as always~$\pi$ is the representation.associated
to $\rbar$ via Lemma~\ref{lem:rbar has a unique Serre weight, and the
  deformation ring is smooth, and given by Tp}~(5).
Let $(A, \mm)$ be a complete local noetherian $\OO$-algebra. Let $\dualcat(A)$ be the Pontryagin dual of $\Mod^{\ladm}_G(A)$, where, for the moment, we
  allow $G$ to be any $p$-adic analytic group.
  There is a forgetful functor from $\dualcat(A)$ to $\dualcat(\cO)$.
   In this subsection we prove a structural result about objects
    $P$ of $\dualcat(A)$ that are projective in $\dualcat(\OO)$. We will apply this result to
   $P=M_{\infty}$ and $G=\GL_2(\Qp)$.

\begin{lem}\label{lem: factoring Vytas out}
 Let $(A,\mathfrak{m})$ be a complete local noetherian $\F$-algebra with residue field~$\F$. Let $P \in \dualcat(A)$ be such that $P$ is projective in $\dualcat(\F)$ and the map $P\twoheadrightarrow \cosoc_{\dualcat(\F)}P$ is essential. Assume that all irreducible
subquotients of $\cosoc_{\dualcat(\F)} P$ are isomorphic to some given
object $S$, for which
$\End_{\dualcat(\F)}(S)=\F$. If $\Hom_{\dualcat(\F)}( P, S)^{\vee}$  is a free $A$-module of rank $1$ then there is an
  isomorphism $A\wtimes_\F \Proj(S) \cong P$ in $\dualcat(A)$,
  where $\Proj (S)\twoheadrightarrow S$ is a projective envelope
  of $S$ in $\dualcat(\F)$.
\end{lem}
\begin{proof}

  The assumption on the cosocle of $P$ implies that
  $(\cosoc P)^{\vee}$ is isomorphic to a direct sum of copies of
  $\pigmonkey:= S^{\vee}$.
 This means that we have natural isomorphisms:
$$(\cosoc P)^{\vee}\cong  \pigmonkey \otimes_\F \Hom_G(\pigmonkey, (\cosoc P)^{\vee})\cong \pigmonkey \otimes_\F \Hom_{\dualcat(\F)} (\cosoc P, S).$$
Taking Pontryagin duals we get  a natural isomorphism in $\dualcat(\F)$:
$$\Hom_{\dualcat(\F)} ( \cosoc P, S)^{\vee} \wtimes_\F S \cong  \cosoc P.$$

Since the isomorphism is natural, it is an isomorphism in
$\dualcat(A)$ with the trivial action of $A$ on $S$. Hence, we get a
surjection in $\dualcat(A)$:
$$P\twoheadrightarrow  \Hom_{\dualcat(\F)} ( \cosoc P, S)^{\vee} \wtimes_\F S.$$
The surjection $\Proj(S)\twoheadrightarrow S$ induces a surjection
$$\Hom_{\dualcat(\F)} ( \cosoc P, S)^{\vee} \wtimes_\F \Proj (S) \twoheadrightarrow  \Hom_{\dualcat(\F)} ( \cosoc P, S)^{\vee} \wtimes_\F S,$$
with trivial $A$-action on $\Proj(S)$. The source of this surjection is projective in $\dualcat(A)$, since in
general for a compact $A$-module $\mathrm{m}$
\numequation\label{adjointness_tensor} \Hom_{\dualcat(A)} ( \mathrm{m}
\wtimes_\F \Proj (S), -)\cong \Hom_A ( \mathrm{m},
\Hom_{\dualcat(\F)}( \Proj(S), -)),
\end{equation}
and $\Hom_{\dualcat(\F)} ( \cosoc P, S)^{\vee}= \Hom_{\dualcat(\F)}(P, S)^{\vee}$ is
projective since it is a free $A$-module of rank $1$.

Hence there is a map $\Hom_{\dualcat(\F)} ( \cosoc P, S)^{\vee} \wtimes_\F \Proj (S) \rightarrow  P$ in $\dualcat(A)$, such that the diagram
\begin{displaymath}
\xymatrix@1{\Hom_{\dualcat(\F)} ( \cosoc P, S)^{\vee} \wtimes_\F \Proj (S)\ar@{->>}[d]\ar[r] &   P \ar@{->>}[d]\\
 \Hom_{\dualcat(\F)} ( \cosoc P, S)^{\vee} \wtimes_\F S \ar[r]^-{\cong}&  \cosoc P }
\end{displaymath}
commutes.  If we forget the $A$-action then we obtain a map in
$\dualcat(\F)$ between projective objects, which induces an
isomorphism on their cosocles.  Hence the map is an isomorphism in
$\dualcat(\F)$, and hence also an isomorphism in $\dualcat(A)$.
\end{proof}

\begin{prop}\label{factoring patching variables over O} Let $A$ be a complete local noetherian $\cO$-algebra with residue field $\F$, that is $\cO$-flat. Let $P$ in $\dualcat(A)$ be such that $P$ is
  projective in $\dualcat(\cO)$ and the map $P\twoheadrightarrow \cosoc_{\dualcat(\cO)}P$ is essential.  Assume that all the irreducible
  subquotients of $\cosoc_{\dualcat(\cO)} P$ are isomorphic to $S$,
  and $\End_{\dualcat(\cO)}(S)=\F$. If $\Hom_{\dualcat(\cO)}( P, S)^{\vee}$ is a free
  $A/\varpi A$-module of rank $1$ then there is an isomorphism in
  $\dualcat(A)$:
$$A \wtimes_\cO \Proj(S) \cong  P,$$ where $\Proj(S)\twoheadrightarrow S$ is a projective envelope of $S$ in $\dualcat(\cO)$.
\end{prop}
\begin{proof}
  A special case of Lemma \ref{lem: factoring Vytas out} implies that the reductions of
  $A\wtimes_{\cO} \Proj(S)$ and $P$ modulo $\varpi$ are isomorphic in
  $\dualcat(\F)$.  Arguing as in \eqref{adjointness_tensor} we deduce
  that $A\wtimes_{\cO} \Proj(S)$ is projective in $\dualcat(A)$. Thus
  there is a map in $\dualcat(A)$,
  $$A\wtimes_{\cO} \Proj(S)\rightarrow P,$$
  which is an isomorphism
  modulo $\varpi$. If $V$ is a cokernel of this map then
  $V/\varpi V=0$, and Nakayama's lemma for compact $\cO$-modules
  implies that $V=0$. Since $P$ is projective this surjection must
  split. Since the map is an isomorphism modulo $\varpi$, if $U$ is
  the kernel of this map then the decomposition
  $A\wtimes_{\cO} \Proj(S)\cong U\oplus P$ implies that
  $U/\varpi U=0$, and so $U=0$, as required.
\end{proof}

\begin{cor}\label{get_proj} Let $A$, $P$, and $S$ be as in
Proposition~{\em \ref{factoring patching variables over O}}. Then for any
$\OO$-algebra homomorphism $x: A\rightarrow \OO$, $P\wtimes_{A, x} \OO$ is a projective envelope of
$S$ in $\dualcat(\OO)$.
\end{cor}
\begin{proof} It follows from Proposition~\ref{factoring patching variables over O} that
\[P\wtimes_{A, x}\OO\cong (\Proj(S)\wtimes_\OO A)\wtimes_{A, x}\OO\cong \Proj(S).\qedhere\]
\end{proof}

\begin{rem}
\label{rem: the factoring patching variables out lemmas work ok in rank n}
If we replace the assumption in Proposition~\ref{factoring patching variables over
O} that $\Hom_{\dualcat(\cO)}( P, S)^{\vee}$ is a free
$A/\varpi A$-module of rank $1$ with the assumption that it is free
of rank~$n$, then we have an isomorphism
$A^{\oplus n} \wtimes_\cO \Proj(S) \cong  P$  in
 $\dualcat(A)$.

Indeed,
a generalization of Lemma~\ref{lem: factoring Vytas out} to the rank $n$ case gives an isomorphism in $\dualcat(\F)$ of $A^{\oplus n}\wtimes_{\cO}\Proj(S)$ and
$P$ modulo $\varpi$. (In fact, the statement of Lemma~\ref{lem: factoring Vytas out}
can be strengthened as follows: if $\Hom_{\dualcat(\F)}(P,S)^\vee$ is a projective object in the category of compact $A$-modules, then we have an isomorphism $\Hom_{\dualcat(\F)}(P,S)^\vee\wtimes_\F\Proj S\simeq P$. The key is again the projectivity of the completed tensor product $\Hom_{\dualcat(\F)}(P,S)^\vee\wtimes_\F\Proj S$, which follows from our assumption and from~\eqref{adjointness_tensor}.)
We then upgrade the isomorphism modulo $\varpi$ to an isomorphism in $\dualcat(A)$ as in the proof of Proposition~\ref{factoring patching variables over O}, again relying on~\eqref{adjointness_tensor}.
\end{rem}

We now apply the results above in the special case of $P=M_{\infty}$ and $G=\GL_2(\Qp)$.
\begin{prop}\label{cut_it_out} Let $A=\OO[[x_1,\ldots, x_d]]$ and choose a homomorphism of local $\OO$-algebras
$A\rightarrow R_{\infty}$, that induces an isomorphism $R_p\wtimes_{\OO} A\cong R_{\infty}$.
Then there is an isomorphism in $\dualcat(A)$:
$$ M_{\infty}\cong \wP \wtimes_{\OO} A,$$
where $\wP\twoheadrightarrow \pi^{\vee}$ is a projective envelope of $\pi^{\vee}$ in $\dualcat(\OO)$.
\end{prop}
\begin{proof} Theorem \ref{thm: Minfty is projective} implies that $M_{\infty}$ is projective in $\dualcat(\OO)$.
	As we already noted in the proof of that theorem,
since $M_{\infty}^{\vee}$ is locally admissible,
it follows from Lemma~\ref{lem:socles are essential} that
$\soc_G M_{\infty}^{\vee}\hookrightarrow M_{\infty}^{\vee}$ is essential, and
hence that
$M_{\infty}\twoheadrightarrow \cosoc_{\dualcat(\OO)} M_{\infty}$ is essential.
Proposition~\ref{prop: abstract Serre weight freeness and Hecke
    compatibility}~(3) implies that all the irreducible subquotients of $\cosoc_{\dualcat(\OO)} M_{\infty}$
are isomorphic to $\pi^{\vee}$. It is therefore enough to show that
$$M_{\infty}(\pi):=\Hom_{\dualcat(\OO)}(M_{\infty}, \pi^{\vee})^{\vee}\cong \Hom_G(\pi, M_{\infty}^{\vee})^{\vee}$$
is a free $A/\varpi$-module of rank $1$, since the assertion then follows from Proposition~\ref{factoring patching variables over O}.

As in the proof of Proposition~\ref{prop: abstract Serre weight freeness and Hecke
    compatibility}, we have
$$M_{\infty}(\pi)\cong \F \otimes_{\cH(\sigmabar)} M_{\infty}(\sigmabar)\cong \F \otimes_{\cH(\sigmao)} M_{\infty}(\sigmao).$$
It follows from Proposition~\ref{prop: abstract Serre weight freeness and Hecke
    compatibility}~(1) that $M_{\infty}(\pi)$ is a free
$\F\otimes_{\cH(\sigmao)} R_{\infty}(\sigma)$-module of rank $1$. Since $R_{\infty}(\sigma)\cong R_p(\sigma)\wtimes_{\OO} A$
and the map $\cH(\sigmabar)\rightarrow R_{\infty}(\sigma)/\varpi$ factors
through $R_p(\sigma)/\varpi$
by Proposition~\ref{prop: abstract Serre weight freeness and Hecke
    compatibility}~(2), we conclude that the map $A\rightarrow
  R_{\infty}$ induces an isomorphism $$A/\varpi\cong
  \F\otimes_{\cH(\sigmao)}R_{\infty}(\sigma).$$ (Recall that $\F\otimes_{\cH(\sigmao)}R_p(\sigma)/\varpi
	= \F,$ by Lemma~\ref{lem:rbar has a unique Serre weight, and the deformation ring
    is smooth, and given by Tp}~(4).)
Thus $M_{\infty}(\pi)$ is a free $A/\varpi$-module of rank $1$, as required.
\end{proof}

\begin{cor}\label{cor_cut_it_out} Let $A\rightarrow R_{\infty}$ be as in Proposition \ref{cut_it_out} and let $x: A\rightarrow \OO$ be a homomorphism
of local $\OO$-algebras. Then $M_{\infty}\wtimes_{A,x} \OO$ is a projective envelope of $\pi^{\vee}$ in $\dualcat(\OO)$
with a continuous $R_p\cong R_{\infty}\otimes_{A, x} \OO$-action, which commutes with the action of $G$.
\end{cor}
\begin{proof} This follows from Corollary \ref{get_proj}.
\end{proof}

\subsection{Uniqueness of arithmetic actions}As in the statement of
Proposition~\ref{cut_it_out}, we let
$\wP\twoheadrightarrow \pi^{\vee}$ be a projective envelope of
$\pi^{\vee}$ in $\dualcat(\OO)$.

\begin{prop}\label{it_exists} $\wP$ can be endowed with an arithmetic
  action of $R_p$ {\em (}in the sense of Section \ref{subsec: axioms for arithmetic action}
  when $d=0${\em )}.
\end{prop}
\begin{proof}
  Making any choice of morphism $x:A\to \cO$ in
  Corollary~\ref{cor_cut_it_out}, we obtain an action of $R_p$ on
  $M_\infty\wtimes_{A,x} \OO\cong\wP$. Since the action of $R_\infty$
  on $M_\infty$ is an arithmetic action, it follows immediately from
  the definitions that this induced action of $R_p$ on $\wP$ is also
  an arithmetic action.
\end{proof}

 \subsection{Recapping capture}\label{subsec:capture}We now very briefly
recall the theory of capture from~\cite[\S2.4]{CDP}, specialised to
the case of interest to us. We note that analogues of these results are valid for
general choices of~$G$, and in particular do not use
either Colmez's functor nor $p$-adic Langlands correspondence for
$\GL_2(\Qp)$.

Let $M$ be a compact linear-topological
$\cO[[K]]$-module, and let $\{V_i\}_{i\in I}$ be a set of continuous
$K$-representations on finite-dimensional $E$-vector spaces.
\begin{defn}
  \label{defn: capture}We say that $\{V_i\}_{i\in I}$ \emph{captures}
$M$ if for any proper quotient $M\onto Q$, we have
$\Hom_{\OO[[K]]}^{\cont}(M,V_i^\ast)\ne\Hom_{\OO[[K]]}^{\cont}(Q,V_i^\ast) $
for some $i\in I$. \end{defn}
This definition is used only in the proof of the following result.
\begin{prop}\label{prop:capture implies unique action}
  Suppose that $\phi\in\End_{\OO[[K]]}^{\cont}(\wP)$
  kills each $\Hom_{\OO[[K]]}^{\cont}(\wP,\sigma_{a,b}^\ast)$ for $a\in \Z$ and $b\in \Z_{\ge 0}$. Then $\phi=0$.
\end{prop}
\begin{proof}
  Since $\wP$ is projective in $\dualcat(\cO)$, it follows from
  ~\cite[Prop.\ 2.12]{CDP} that  the set
$\{\sigma_{a,b}\}$ captures~$\wP$. The result follows from~\cite[Lem.\
2.9]{CDP} (that is, from an application of the definition of capture
to the cokernel of~$\phi$).
\end{proof}
Set $M(\sigma^\circ):=\left(\Hom^{\mathrm{cont}}_{\cO[[K]]}(\wP,(\sigma^\circ)^d)\right)^d$.
\begin{lem}\label{fibre_hecke_old} Let $\sigma=\sigma_{a,b}$ with $a\in \Z$ and $b\in \Z_{\ge 0}$ and let $\mm_y$ be a maximal ideal of $\cH(\sigma)$. Then
  $\kappa(y)\otimes_{\cH(\sigma)} M(\sigmao)[1/p]\neq 0$ if
  and only if $\mm_x:=\eta(\mm_y) R_p(\sigma)[1/p]$ is a maximal ideal
  of $R_p(\sigma)[1/p]$ in the support of $M(\sigma^\circ)[1/p]$ for
  some \emph{(}equivalently, any\emph{)} arithmetic action of $R_p$ on~$\wP$.
\end{lem}\begin{proof} Since $M(\sigmao)$ is a finitely generated $R_p(\sigma)$-module and the action of $\cH(\sigma)$
on $M(\sigmao)[1/p]$ factors through the action of $R_p(\sigma)[1/p]$ via $\eta$, we deduce that
$\kappa(y)\otimes_{\cH(\sigma)} M(\sigmao)[1/p]$ is a finitely generated $\kappa(y)\otimes_{\cH(\sigma)} R_p(\sigma)[1/p]$-module. If $\kappa(y)\otimes_{\cH(\sigma)} M(\sigmao)[1/p]\neq 0$ then we deduce from Lemma \ref{fibre_hecke} that
$\mm_x$ is a maximal ideal of $R_p(\sigma)[1/p]$ in the support of $M(\sigmao)$. Conversely, if $y$ is the image of $x$
then using Lemma \ref{fibre_hecke} we obtain $\kappa(x)=\kappa(y)\otimes_{\cH(\sigma)} R_p(\sigma)[1/p]$ and hence:
$$ \kappa(x)\otimes_{R_p(\sigma)[1/p]} M(\sigmao)[1/p]\cong \kappa(y)\otimes_{\cH(\sigma)} M(\sigmao)[1/p], $$
which implies that $\kappa(y)\otimes_{\cH(\sigma)} M(\sigmao)[1/p]$ is non-zero.
\end{proof}

\begin{thm}\label{unique_arithmetic} There is a unique arithmetic action of $R_p$ on $\wP$.
\end{thm}
\begin{proof} The existence of such an action follows from Proposition
  \ref{it_exists}.

Let $\sigma=\sigma_{a, b}$ with $a\in \Z$ and $b\in \Z_{\ge 0}$.
 Let $y\in \mSpec \cH(\sigma)$, such that $\kappa(y)\otimes_{\cH(\sigma)} M(\sigmao)[1/p]\neq 0$.
Lemma \ref{fibre_hecke_old} implies that $y$ is the image of $x\in \mSpec R_p(\sigma)[1/p]$, which
lies in the support of $M(\sigmao)$.
Proposition ~\ref{prop: generic fibre of cristabelline deformation ring} implies that $$\widehat{M(\sigmao)[1/p]}_{\mm_y}=\widehat{M(\sigmao)[1/p]}_{\mm_x},$$  as $\widehat{R_p(\sigma)[1/p]}_{\mm_x}$-modules.
Moreover, the action of  $\widehat{R_p(\sigma)[1/p]}_{\mm_x}$ on $\widehat{M(\sigmao)[1/p]}_{\mm_y}$ does not depend on a given arithmetic action of $R_p$ on $\wP$,
as it acts via the isomorphism in  Proposition~\ref{prop: generic fibre of cristabelline deformation ring}. If $M$ is a finitely generated module over a noetherian ring $R$ then we have injections:
$$ M\hookrightarrow \prod_{\mm} M_{\mm}\hookrightarrow \prod_{\mm} \widehat{M}_{\mm},$$
where the product is taken over all the maximal ideals in $R$. In fact it is enough to take the product over finitely many maximal ideals: if $\pp_1, \ldots, \pp_n$
are minimal associated primes of $M$, just pick any maximal ideals $\mm_1\in V(\pp_1), \ldots, \mm_n\in V(\pp_n)$. (The second injection follows from \cite[Thm. 8.9]{Matsumura}.)
This observation applied to $R=R_p(\sigma)[1/p]$ and $M=M(\sigmao)[1/p]$ together with Lemma \ref{fibre_hecke_old} implies that we have an injection of $R_p(\sigma)[1/p]$-modules:
$$ M(\sigmao)[1/p]\hookrightarrow \prod_{y\in \mSpec \cH(\sigma)}  \widehat{M(\sigmao)[1/p]}_{\mm_y}.$$
Since the map and the action of $R_p(\sigma)[1/p]$ on the right hand side are independent of the arithmetic action of $R_p$ on $\wP$, we deduce that
the action of $R_p(\sigma)[1/p]$ on $M(\sigmao)[1/p]$ is also
independent of the arithmetic action of $R_p$ on $\wP$.

If $\theta: R_p\rightarrow \End_{\dualcat(\OO)}(\wP)$ and
$\theta': R_p\rightarrow \End_{\dualcat(\OO)}(\wP)$ are two arithmetic actions  and $r\in R_p$ then it follows from the above that
$\theta(r)-\theta'(r)$ will annihilate $M(\sigma_{a,b}^{\circ})$ for all $a\in \Z$ and $b\in \Z_{\ge 0}$.
Proposition~\ref{prop:capture implies unique action} implies that $\theta(r)=\theta'(r)$.
\end{proof}

\begin{rem} A different proof of the Theorem could be given using \cite[Prop. 2.19]{CDP}.
\end{rem}

\begin{thm}\label{main} If $\wP\twoheadrightarrow \pi^{\vee}$ is a projective envelope of $\pi^{\vee}$ in $\dualcat(\OO)$ equipped with an arithmetic action
of $R_p$, then there is an isomorphism in $\dualcat(R_{\infty})$
$$ \wP\wtimes_{R_p} R_{\infty}\cong M_{\infty}.$$
\end{thm}
\begin{proof} Let $A=\OO[[x_1, \ldots, x_d]]$, and  choose a homomorphism of local $\OO$-algebras
$A\rightarrow R_{\infty}$ which induces an isomorphism $R_p\wtimes_{\OO} A\cong R_{\infty}$.
Proposition \ref{cut_it_out} implies that there is an isomorphism $\wP\wtimes_{\OO}A\cong M_{\infty}$
in $\dualcat(A)$; it is therefore enough to show that this isomorphism is $R_p$-linear. Any $\OO$-algebra homomorphism $x: A\rightarrow \OO$  induces an isomorphism $(\wP\wtimes_{\OO} A)\wtimes_{A,x} \OO \cong M_{\infty}\wtimes_{A, x}\OO$ in $\dualcat(\OO)$.
We get two actions of $R_p$ on $M_{\infty}\wtimes_{A, x}\OO$: one of them coming from the action of $R_p$ on $M_{\infty}$,
the other transported by the isomorphism. Both actions are arithmetic: the first one by Proposition \ref{it_exists}, the second one by assumption. Theorem \ref{unique_arithmetic} implies that  the two actions coincide; thus the isomorphism
$(\wP\wtimes_{\OO} A)\wtimes_{A,x} \OO \cong M_{\infty}\wtimes_{A,
  x}\OO$ is $R_p$-linear.

We have a commutative diagram in $\dualcat(A)$:
 \begin{displaymath}
 \xymatrix@1{\wP\wtimes_{\OO} A\ar[d]\ar[r]^{\cong} &   M_{\infty} \ar[d]\\
 \prod\limits_{x: A\rightarrow \OO} (\wP\wtimes_{\OO} A)\wtimes_{A, x} \OO \ar[r]^-{\cong}& \prod\limits_{x: A\rightarrow \OO} M_{\infty}\wtimes_{A, x} \OO }
\end{displaymath}
where the product is taken over all $\OO$-algebra homomorphisms $x: A\rightarrow \OO$. We know that both vertical and the lower horizontal arrows are
$R_p$-linear. The map $A\rightarrow \prod\limits_{x:A\rightarrow \OO}
\OO$ is injective (see for example \cite[Lem. 9.22]{paskunasimage}).
Since $\wP$ is $\OO$-torsion free the functor $\wP \wtimes_{\OO}-$ is exact, hence the first vertical arrow is injective, which implies that the top
horizontal arrow is $R_p$-linear.
\end{proof}

\begin{rem}
	The preceding result shows that,
  in particular, the construction of $M_{\infty}$ in \cite{Gpatch} is
  independent of the choices made in the case of~$\GL_2(\Qp)$. More precisely, the isomorphism $$
  \wP\wtimes_{R_p} R_{\infty}\cong M_{\infty}$$ exhibits $M_\infty$ as
  the extension of scalars of~$\wP$, and this latter object,
  with its arithmetic $R_p$-action, is independent of all
  choices by Theorem~\ref{unique_arithmetic}.  Thus, the only ambiguity in
  the construction of~$M_\infty$ is in the number of power series
  variables in~$R_\infty$, and in their precise action. As one is free
  to choose as many Taylor--Wiles primes as one wishes in the patching
  construction, and as the presentations of global deformation rings
  as quotients of power series rings over local deformation rings are
  non-canonical, it is evident that this is the exact degree of
  ambiguity that the construction of~$M_\infty$ is forced to
  permit.
 \end{rem}
\section{Unitary completions of principal series
  representations}\label{sec: Berger Breuil stuff}

In this section, we record some arguments related
to the paper~\cite{berger-breuil}, which proves using
$(\varphi,\Gamma)$-module techniques that the locally algebraic
representations associated to crystabelline Galois representations
admit a unique unitary completion.  We will use the machinery
developed in the previous sections, namely the projective envelope
$\wP$ and the purely local map $R_p\to \End_{\dualcat(\OO)}(\wP)$, to independently
deduce (without using $(\varphi,\Gamma)$-module techniques) that the
locally algebraic representations corresponding to crystalline types
of regular weight admit \emph{at most one} unitary completion
satisfying certain properties, and that such a completion comes from
some~$\wP$.

We will use this result in Section~\ref{sec: local global} below to show,
assuming the existence results of~\cite{berger-breuil}, that certain of these
representations occur in completed cohomology.  This gives an
alternative approach to proving modularity results in the crystalline
case.

As in Section~\ref{sec:Galois
  deformation rings and Hecke algebras}, we write
$\sigma=\sigma_{a,b}=\det^{a}\otimes\Sym^b E^2$. Let
$\theta:\cH(\sigma)\to E$ be a homomorphism, and set
$\Psi:=(\cInd_K^G\sigma)\otimes_{\cH(\sigma),\theta}E$; so $\Psi$ is a locally algebraic principal series representation of~$G$.

\begin{thm}
  \label{thm:Berger--Breuil completions via Minfty}If~$\Psi$ is irreducible, then $\Psi$ admits at most one non-zero admissible unitary completion
  $\widehat{\Psi}$ with the following property: for an open bounded  $G$-invariant lattice $\Theta$ in $\widehat{\Psi}$,
  $(\Theta/\varpi)\otimes_{\F} \Fpbar$ contains no  subquotient of the form $(\Ind_B^G \chi\otimes \chi \omega^{-1})_{\sm}$,
  for any character $\chi: \Qptimes\rightarrow \overline{\F}_p^\times$, and no special series or characters.

  If a completion satisfying this property exists, then it is
  absolutely irreducible.\end{thm}

 \begin{proof}
	 Let $\Pi$ be a non-zero  admissible unitary completion
   of $\Psi$ that satisfies the property in the statement of the theorem.
We will first show that $\Pi$ is absolutely irreducible (and, indeed, most
of the work of the proof will be in showing this). We note that in the course of the proof we are allowed to replace
$E$ by a finite field extension $E'$, since if $\Pi\otimes_E E'$ is an absolutely irreducible $E'$-Banach space representation of
$G$ then $\Pi$ is an absolutely irreducible $E$-Banach space representation of $G$.

   Since $\Pi$ is admissible, it will
 contain an irreducible closed sub-Banach space representation $\Pi_1$.
If we let $\Pi'$ denote the quotient $\Pi/\Pi_1$, then we must show
that $\Pi'$ is zero. For the moment, we note simply that if $\Pi'$
is non-zero, then since the composite $\Psi \to \Pi \to \Pi'$ has dense
image, we see that $\Pi'$ is another non-zero
admissible unitary completion of $\Psi$.

 Let $\Theta$ be an open bounded
 $G$-invariant lattice in $\Pi$, and let $\Theta_1:= \Pi_1\cap \Theta$. Since $\Pi_1$ is also admissible,
 $\Theta_1/\varpi$ will contain an irreducible subquotient $\pi$.  Since we are allowed to enlarge $E$, we may assume that $\pi$ is absolutely irreducible. Let $\wP\twoheadrightarrow \pi^{\vee}$
 be a projective envelope of $\pi^{\vee}$ in $\dualcat(\OO)$. By the
 assumption on the subquotients of $\Theta/\varpi$, there is a Galois
 representation~ $\rbar$ satisfying Assumption~\ref{assumption: rbar is generic enough}
that
 corresponds to $\pi$ via Lemma~\ref{lem:rbar has a unique Serre weight, and the deformation ring
    is smooth, and given by Tp}~(5). We equip $\wP$ with the
  arithmetic action of $R_p$ provided by Theorem~\ref{unique_arithmetic}.

 We let (for any admissible unitary $E$-Banach space representation $\Pi$ of $G$)
 $$M(\Pi):= \Hom_{\dualcat(\OO)}( \wP, \Theta^d)^d[1/p]\cong \Hom_G^{\cont}(\Pi, \Hom_{\OO}^{\cont}(\wP, E))^d. $$
 The projectivity of $\wP$ implies that $\Pi\mapsto M(\Pi)$ is an exact covariant functor from the category of admissible unitary $E$-Banach space representations of $G$ to the category of $R_p[1/p]$-modules.
 In particular, we have an injection $M(\Pi_1)\hookrightarrow
 M(\Pi)$. Since $\wP$ is projective and $\pi$ is a subquotient of
 $\Theta_1/\varpi$, we have $M(\Pi_1)\neq 0$ by \cite[Lem. 4.13]{paskunasimage}, and hence $M(\Pi)\neq 0$.

 Similarly, we let
 $$M(\Psi):=\Hom_G(\Psi, \Hom_{\OO}^{\cont}(\wP, E))^d\cong M(\sigmao)[1/p]\otimes_{\cH(\sigma), \theta} E. $$

 Recall that $\cH(\sigma)$ acts on $M(\sigmao)[1/p]$ through
the composite of  the homomorphism
$\cH(\sigma)\buildrel \eta \over \longrightarrow R_p(\sigma)[1/p]$ and the $R_p[1/p]$-action on $M(\sigmao)[1/p]$,
and that, by Lemma~\ref{fibre_hecke_old},
either the fibre module
$M(\sigmao)[1/p] \otimes_{\cH(\sigma),\theta} E$ vanishes,
or else  $\theta$ extends to a homomorphism
$\thetahat: R_p(\sigma)[1/p] \to E$
(so that then $\theta = \thetahat \circ \eta$),
in which case there is a natural isomorphism
$$M(\sigmao)[1/p] \otimes_{\cH(\sigma),\theta}E
\iso
M(\sigmao)[1/p] \otimes_{R_p(\sigma)[1/p],\thetahat}E$$
of $E$-vector spaces of dimension at most $1$.

 The map $\Psi\rightarrow \Pi$ induces a continuous homomorphism
 $$ \Hom_G^{\cont}(\Pi, \Hom_{\OO}^{\cont}(\wP, E))\rightarrow \Hom_G(\Psi, \Hom_{\OO}^{\cont}(\wP, E)).$$
 Since the image of $\Psi$ in $\Pi$ is dense this  map is injective,
 and by taking duals we obtain a surjection
 of $R_p[1/p]$-modules $$M(\Psi)\twoheadrightarrow M(\Pi).$$
 Since the target is non-zero, and the source is an $E$-vector space
 of dimension at most $1$,   this must be an isomorphism.
Since $M(\Pi_1)$ is a non-zero subspace of $M(\Pi)$,
we therefore have induced isomorphisms $M(\Pi_1) \iso M(\Pi)\iso M(\Psi)$.
Since, as was noted above, $M$ is an exact functor, we find that
$M(\Pi') = 0$.

We digress for a moment in order to establish that $\Pi_1$ is in fact
absolutely irreducible.  Since it is irreducible and admissible, its endomorphism ring
is a division algebra over $E$.  On the other hand, since $M$ is a functor,
and since $M(\Pi_1)$ is one-dimensional over $E$, we see that this division
algebra admits a homomorphism to $E$.  Thus this division algebra is in
fact equal to $E$, and this implies that $\Pi_1$ is absolutely irreducible by \cite[Lem. 4.2]{paskunasimage}.

Suppose now that $\Pi' \neq 0.$
We may then apply the above argument with $\Pi'$
in place of $\Pi,$ and find an absolutely irreducible subrepresentation $\Pi_2$ of $\Pi',$ a $G_{\Q_p}$-representation $\rbar'$,
an associated irreducible $\GL_2(\Q_p)$-representation $\pi'$,
for which the projective envelope $\wP_2$ of $(\pi')^{\vee}$
gives rise to an exact functor $M'$ such that $M'(\Psi) = M'(\Pi')
= M'(\Pi_2)$,
with all three being one-dimensional.

Since both $M(\Psi)$ and $M'(\Psi)$ are non-zero, we find that
each of $\rbar$ and $\rbar'$ admits a lift that is a lattice in
a crystalline representation~$V$ of Hodge--Tate weights  $(1-a, -(a+b))$ determined (via~Lemma~\ref{well_known}) by the homomorphism $\hat{\theta}: R_p[1/p]\rightarrow E$.
(Strictly speaking, for this to make sense we need to know that $V$ is
indecomposable; but this is automatic, since each of $\rbar$ and $\rbar'$ is indecomposable and has
non-scalar semisimplification.)

We note that $V$ is reducible if and only if $\Psi$ may be identified with the locally algebraic subrepresentation
of a continuous induction $(\Ind_B^G \chi)_{\cont},$ for some unitary character $\chi:T \to E^{\times}$, where $T$ denotes the
diagonal torus contained in the upper triangular Borel subgroup
$B$ of $G$.   In this case it is well-known that
this continuous induction is the universal unitary completion
of $\Psi$, see \cite[Prop. 2.2.1]{MR2667890}, and the theorem is true in this case.  Thus, for the remainder of the argument,
we suppose that $V$ is irreducible, or, equivalently,
that $\Psi$ does not admit an embedding into the continuous parabolic induction
of a unitary character.

We now consider separately two cases, according to whether or not $\rbar$
and $\rbar'$ are themselves isomorphic.
If they are, then $M$ and $M'$ are isomorphic
functors, and we obtain a contradiction from the fact that $M(\Pi') = 0$
while $M'(\Pi')$ is one-dimensional, implying that in fact $\Pi' = 0,$
as required.

If $\rbar$ and $\rbar'$ are not isomorphic, but have isomorphic
semi-simplifications, then they must each consist of an extension of
the same two characters, but in opposite directions.   In this case
$\pi\cong (\Ind_{B}^G \omega \chi_1\otimes \chi_2)_{\sm}$, $\pi'\cong  (\Ind_{B}^G \omega \chi_2\otimes \chi_1)_{\sm},$
for some smooth characters $\chi_1, \chi_2:\Qp^{\times}\rightarrow\F^{\times}$. In the terminology of \cite{MR2642409}, $\pi$ and
$\pi'$
 are the two constituents of an atome
automorphe, which by definition is the unique (up to isomorphism) non-split extension between $\pi$ and $\pi'$.

We first show that $\Pi_2 = \Pi'.$ To this end, set $\Pi'' := \Pi'/\Pi_2$. If $\Pi'' \neq 0,$ then,
running through the above argument another time,
we find $\rbar''$, etc., such that $(\rbar'')^{\ssg} \cong (\rbar')^{\ssg}
\cong \rbar^{\ssg}$, and such that $M''(\Pi'') \neq 0.$
But either $\rbar'' \cong \rbar$ or $\rbar'' \cong \rbar'$.  Thus
the functor $M''$ is isomorphic to either $M'$ or to $M$.
On the other hand, $M'(\Pi'') = 0$, and also $M(\Pi'') = 0$ (since
$\Pi''$ is a quotient of $\Pi'$ and $M(\Pi') = 0$).    This contradiction
shows that $\Pi'' = 0$, and thus that $\Pi' = \Pi_2$ is absolutely irreducible.

Since $M(\Pi') = 0,$ we see that $\Theta'/\unif$ does not contain
a copy of $\pi$ as a subquotient (here $\Theta'$ denotes
some choice of $G$-invariant lattice in $\Pi'$).  Since (as we have just shown)
$\Pi'$ is absolutely irreducible,
it follows
from \cite[Cor.~8.9]{paskunasimage} that $\Pi' \cong (\Ind_B^G \chi)_{\cont}$
for some unitary character $\chi:T \to E^{\times}$, where $T$ denotes the
diagonal torus contained in the upper triangular Borel subgroup
$B$ of $G$.   (The proof of this result uses various $\Ext$ computations
in $\Mod^{\ladm}_G(\OO)$, but does not use either of Colmez's functors.
The basic idea is that the extension of $\pi'$ by $\pi$ given
by the atome automorphe induces an embedding $\wP \hookrightarrow
\wP'$ whose cokernel is dual to
the induction from $B$ to $G$ of a character; since $M(\Pi') = 0$,
we see that $\Theta'$ embeds into this induction, and so is itself
such an induction.) Thus $\Psi$ admits an embedding into the continuous parabolic induction of
a unitary character, contradicting our hypothesis that $V$
is irreducible.   Thus we conclude that in fact $\Pi' = 0$, as required.

 If $\Psi$ were to admit two non-isomorphic admissible irreducible unitary completions $\Pi_1$ and $\Pi_2$ satisfying the
 assumptions of the theorem,  the image of the diagonal map $\Psi \rightarrow \Pi_1\oplus \Pi_2$ would be dense in $\Pi_1\oplus \Pi_2$. This yields a contradiction to what we have already proved, as $\Pi_1\oplus \Pi_2$ is not irreducible.
\end{proof}
\begin{rem}\label{rem: BB give completions}
It follows from the proof of Theorem~\ref{thm:Berger--Breuil
  completions via Minfty} that if a completion $\widehat{\Psi}$ of the
kind considered there exists, then there is a Galois representation
$\rbar$ satisfying Assumption~\ref{assumption: rbar is generic enough} such that~$\theta$ extends to a homomorphism
$\widehat{\theta}:R_p(\sigma)[1/p]\to E$. Furthermore, for any
$M_\infty$ as in Section~\ref{sec: recalling
  from Gpatching}, we have
$M_\infty(\sigma^\circ)[1/p]\otimes_{\cH(\sigma),\theta}E\ne 0$.

Conversely, if there is an
$\rbar$ satisfying Assumption~\ref{assumption: rbar is generic enough}
such that $\theta$ extends to a homomorphism
$\widehat{\theta}:R_p(\sigma)[1/p]\to E$, then the existence of a completion~ $\widehat{\Psi}$  is immediate from the main theorem
  of~\cite{berger-breuil} (which shows that some completion exists)
  together with the main theorem of~\cite{MR2642408} (which implies
  the required property of the subquotients). \end{rem}

\begin{rem}\label{we_can_do_crystabelline}
  \label{rem: crystabelline cases would be fine if we could be
    bothered to set up the notation}As is true throughout this paper,
  the results of this section are equally valid for crystabelline
  representations, but for simplicity of notation we have restricted ourselves
  to the crystalline case.
\end{rem}

\section{Comparison to the local approach}\label{sec:comparison}
We now examine the
compatibility of our constructions with those of~\cite{paskunasimage},
which work with fixed central characters. The arguments
of~\cite{paskunasimage} make use of Colmez's functor, and the results
of this section therefore also depend on this functor. In Section~\ref{subsec:rings} we briefly discuss how to prove some of the more
elementary statements in~\cite{paskunasimage} using the results of the
previous section (and in particular not using Colmez's
functor). We assume throughout this section that~$p\ge 5$, as this
assumption is made in various of the results of~\cite{paskunasimage}
that we cite.

There are two approaches that we could take to this comparison. One
would be to note that the axioms of Section~\ref{subsec: axioms for
  arithmetic action} and arguments of Section~\ref{sec: proof that
  arithmetic action is unique} admit obvious analogues in the setting
of a fixed central character, and thus show that if we pass to a
quotient of~$M_\infty$ with fixed central character, we obtain a
uniquely determined $p$-adic Langlands correspondence. These axioms
are satisfied by the purely local object constructed
in~\cite{paskunasimage} (which is a projective envelope of~$\pi^\vee$
in a category of representations with fixed central character), and
this completes the comparison.

While this route would be shorter, we have preferred to take the
second approach, and go in the opposite direction: we promote the
projective envelope from~\cite{paskunasimage} to a representation with
non-constant central character by tensoring with the universal
deformation of the trivial 1-dimensional representation (which has a
natural Galois action by local class field theory), and show that this
satisfies the axioms of Section~\ref{subsec: axioms for arithmetic
  action}. This requires us to make a careful study of various
twisting constructions; the payoff is that we prove a stronger result
than that which would follow from the first approach.

\subsection{Deformation rings and twisting}

Let $\Lambda$ be the universal deformation ring of the trivial $1$-dimensional representation of $G_\Qp$ and let $\Eins^{\univ}$ be the universal deformation.
Then $\Eins^{\univ}$ is a free $\Lambda$-module of rank $1$ with a continuous $G_\Qp$-action. We let $G$ act on $\Eins^{\univ}$
via the inverse determinant (composed with the Artin map). We let $\Lambda^{\ur}$ be the quotient of $\Lambda$ unramified deformations.
We note that  $\Lambda$ and $\Lambda^{\ur}$ are formally
smooth over $\cO$ of relative dimensions $2$ and~$1$, respectively,
and in particular
are $\OO$-torsion free.

Let $\psi: G_{\Qp}\rightarrow \cO^{\times}$ be a continuous character such that $\psi\varepsilon^{-1}$ is congruent to $\det \rbar$ modulo $\varpi$.
(By Remark~\ref{rem: comparting our convention on pi to the
  literature}, this implies that $\psi$ modulo $\varpi$ considered
as a character of $\Q_p^{\times}$ coincides with the central character
of $\pi$.) We let $R^{\psi}_p$ denote the quotient of $R_p$ parameterising deformations with determinant
$\psi\varepsilon^{-1}$. Let $r^{\univ, \psi}$ be the tautological deformation of $\rbar$ to $R^{\psi}_p$. Then
$r^{\univ, \psi}\wtimes_{\OO} \Eins^{\univ}$ is a deformation of
$\rbar$ to $R^{\psi}_p \wtimes_{\OO} \Lambda$. Since $p>2$, this induces an isomorphism of local $\OO$-algebras
\numequation\label{iso_twist_univ}R_p\overset{\cong}{\longrightarrow} R_p^{\psi}\wtimes_{\OO} \Lambda.\end{equation}

Let $R_p^\psi(\sigma)$ denote the quotient of $R_p(\sigma)$
corresponding to deformations of determinant
$\psi\varepsilon^{-1}$; note that $R_p^\psi(\sigma)=0$ unless
$\psi|_{\Z_p^{\times}}$ is equal to the central character of
$\sigma$. If $\psi|_{\Z_p^{\times}}$ is equal to the central character of
$\sigma$, then the isomorphism $R_p\overset{\cong}{\longrightarrow} R_p^{\psi}\wtimes_{\OO} \Lambda$ induces an isomorphism \numequation\label{iso_twist_pst}R_p(\sigma)\overset{\cong}{\longrightarrow} R_p^{\psi}(\sigma)\wtimes_{\OO} \Lambda^{\ur}.\end{equation}

Let $\delta: G_{\Qp}\rightarrow \OO^{\times}$ be a character that is trivial modulo $\varpi$.
Twisting by $\delta$ induces
isomorphisms of $\OO$-algebras
$$\tw_{\delta}: R_p\iso R_p, \quad \tw_{\delta}: R_p^{\psi\delta^2}\iso R_p^{\psi}.$$
(In terms of the deformation functor $D_{\rbar}$ pro-represented
by $R_p$, and the deformation functors $D^{\psi}_{\rbar}$
and $D^{\psi\delta^2}_{\rbar}$ pro-represented by $R_p^{\psi}$
and $R_p^{\psi \delta^2}$, these isomorphisms are induced by the
natural bijections $$D_{\rbar}(A)\iso D_{\rbar}(A),\quad D^{\psi}_{\rbar}(A)\iso D^{\psi\delta^2}_{\rbar}(A)$$ defined by $ r_A\mapsto r_A\otimes \delta$.)
Similarly we obtain
isomorphisms
$$\tw_{\delta}: \Lambda \iso \Lambda, \quad \tw_{\delta}:
\Lambda(\delta)^{\ur}\iso \Lambda^{\ur},$$and we have a commutative
diagram \numequation\label{yet_another_diagram}
 \xymatrix@1{R_p \ar[d]_-{\id} \ar[r]^-{\cong}_-{(\ref{iso_twist_univ})}&   R_p^{\psi\delta^2}\wtimes_{\OO} \Lambda \ar[d]^{\tw_{\delta}\wtimes \tw_{\delta^{-1}}}\\
 R_p\ar[r]^-{\cong}_-{(\ref{iso_twist_univ})}& R_p^{\psi}\wtimes_{\OO} \Lambda}
\end{equation}
\subsection{Unfixing the central character}
Let $\Mod^{\ladm, \psi}_G(\OO)$ be the full subcategory of $\Mod^{\ladm}_G(\OO)$,
consisting of those representations, where $Z$ acts by the central character $\psi$. Let $\dualcat^{\psi}(\OO)$
be the Pontryagin dual of $\Mod^{\ladm, \psi}_G(\OO)$, so that we can identify $\dualcat^{\psi}(\OO)$ with a
full subcategory of $\dualcat(\OO)$ consisting of those objects on which $Z$ acts by $\psi^{-1}$.

Let $\wP^{\psi}\twoheadrightarrow \pi^{\vee}$ be a projective envelope of $\pi^{\vee}$ in $\dualcat^{\psi}(\OO)$.
By \cite[Prop. 6.3, Cor. 8.7, Thm. 10.71]{paskunasimage} there is a natural isomorphism
\numequation\label{def_endo}
R^{\psi}_p\iso \End_{\dualcat(\OO)}(\wP^{\psi}).
\end{equation}

In Corollary~\ref{identify_ring} below we will prove a version of the
isomorphism~(\ref{def_endo}) for~$R_p$.

\begin{lem}\label{twist_and_shout_1} Let $\delta: G_{\Qp}\rightarrow \OO^{\times}$ be a character that is trivial modulo $\varpi$. There is
an isomorphism in $\dualcat(\OO)$:
$$ \varphi: \wP^{\psi\delta^2}\iso\wP^{\psi}\otimes \delta^{-1} \circ \det.$$
Moreover, the following diagram commutes:
\begin{displaymath}
 \xymatrix@1{R_p^{\psi\delta^2} \ar[d]_-{\eqref{def_endo}} ^-{\cong} \ar[r]^-{\tw_{\delta}}_-{\cong}&   R_p^{\psi}\ar[d]^-{\eqref{def_endo}}_-{\cong}
 \\
 \End_{\dualcat(\OO)}(\wP^{\psi\delta^2})\ar[r]^-{\cong}&  \End_{\dualcat(\OO)}(\wP^{\psi})}
\end{displaymath}
where the lower horizontal arrow is given by $\alpha \mapsto \varphi\circ \alpha \circ \varphi^{-1}$.
\end{lem}\begin{proof}
	The claimed isomorphism follows from the fact that twisting
	by $\delta^{-1}\circ \det$
	induces an equivalence of categories between $\dualcat^{\psi}(\OO)$
	and $\dualcat^{\psi\delta^2}(\OO)$,
	so that the twist of a projective envelope
	of $\pi^{\vee}$ in $\dualcat^{\psi}(\OO)$
	is a projective envelope of $\pi^{\vee}$ in $\dualcat^{\psi
		\delta^2}(\OO).$
	The commutativity of the diagram follows from the compatibility of
	the constructions of \cite{paskunasimage}, and in particular
	of the isomorphism~(\ref{def_endo}), with twisting. (This comes
	down to the compatibility of the functor $\cV$, discussed below,
	with twisting.)
\end{proof}

Let $\delta: G_{\Qp}\rightarrow \OO^{\times}$ be a character that is trivial modulo $\varpi$. There is
an evident isomorphism of pseudocompact $\OO[[G_{\Qp}]]$-modules:
$$ \theta: \Eins^{\univ}\otimes \delta  \iso \Eins^{\univ},$$
and a commutative diagram
\numequation\label{twist_and_shout_2}
 \xymatrix@1{\Lambda\ar[d]^-{\cong} \ar[r]^-{\tw_{\delta^{-1}}}_-{\cong}&   \Lambda\ar[d]_-{\cong}
 \\
 \End_{G_{\Qp}}^{\cont}(\Eins^{\univ})\ar[r]^-{\cong}&  \End_{G_{\Qp}}^{\cont}(\Eins^{\univ})}
\end{equation}
where the lower horizontal arrow is given by $\alpha \mapsto
\theta\circ \alpha \circ \theta^{-1}$.
(In terms of the deformation functor $D_{\Eins}$ pro-represented by $\Lambda$,
the isomorphism $\tw_{\delta^{-1}}$ is induced by the bijection
$D_{\Eins}(A) \iso D_{\Eins}(A)$ defined by $\chi \mapsto \chi\delta^{-1}$.)

\begin{lem}\label{lem:twisting} Let $\delta: \Q_p^{\times}\rightarrow \cO^{\times}$ be a character that is trivial modulo
$\varpi$.
Then
there is an isomorphism
$$\wP^{\psi\delta^2}\wtimes_{\OO} \Eins^{\univ}\iso \wP^{\psi}\wtimes_{\OO}\Eins^{\univ}$$
in the category $\dualcat(R_p)$, where $R_p$ acts on both sides by the isomorphism of~\emph{(\ref{iso_twist_univ})}.
\end{lem}
\begin{proof} Using Lemma \ref{twist_and_shout_1} and~(\ref{twist_and_shout_2}) we obtain isomorphisms
in $\dualcat(\OO)$:
\begin{displaymath}
\begin{split}
\wP^{\psi\delta^2}\wtimes_{\OO} \Eins^{\univ}\overset{\varphi\wtimes \id}{\longrightarrow}&
( \wP^{\psi}\otimes \delta^{-1}\circ\det)\wtimes_{\OO} \Eins^{\univ}\iso\\ & \iso\wP^{\psi}\wtimes_{\OO} (\Eins^{\univ}\otimes \delta\circ\det) \overset{\id\wtimes \theta}{\longrightarrow} \wP^{\psi}\wtimes_{\OO} \Eins^{\univ}.
\end{split}
\end{displaymath}
The composition of these isomorphisms  is equal to $\phi\wtimes \theta: \wP^{\psi\delta^2}\wtimes_{\OO} \Eins^{\univ}\iso \wP^{\psi}\wtimes_{\OO}\Eins^{\univ}$, which is an isomorphism in $\dualcat(\OO)$. It follows from Lemma
\ref{twist_and_shout_1} and (\ref{twist_and_shout_2}) that the following diagram commutes:
\begin{displaymath}
 \xymatrix@1{R_p^{\psi\delta^2}\wtimes_{\OO}\Lambda\ar[d] \ar[r]^-{\tw_{\delta}\wtimes\tw_{\delta^{-1}}}_-{\cong}&   R_p^{\psi}\wtimes_{\OO
}\Lambda\ar[d]
 \\
 \End_{\dualcat(\OO)}(\wP^{\psi\delta^2}\wtimes_{\OO} \Eins^{\univ})\ar[r]^-{\cong}&
 \End_{\dualcat(\OO)}(\wP^{\psi}\wtimes_{\OO} \Eins^{\univ})}
\end{displaymath}
where the lower horizontal arrow is given by $\alpha \mapsto
(\varphi\wtimes \theta)\circ \alpha\circ
(\varphi\wtimes\theta)^{-1}$. It follows from the above diagram and~ (\ref{yet_another_diagram}) that $\varphi\wtimes \theta$ is an
isomorphism in $\dualcat(R_p)$.
\end{proof}

\begin{prop}\label{AA2_holds} $\wP^{\psi}\wtimes_{\OO} \Eins^{\univ}$ is projective in the category of pseudocompact $\OO[[K]]$-modules.
\end{prop}
\begin{proof} Let $I_1:=\{ g\in K: g \equiv \bigl(\begin{smallmatrix}
    1 & \ast \\ 0 & 1 \end{smallmatrix}\bigr )\pmod{\varpi}\}$. Then
  $I_1$ is a pro-$p$ Sylow subgroup
of $K$ and it is enough to show that $\wP^{\psi}\wtimes_{\OO} \Eins^{\univ}$ is projective in the category of pseudocompact $\OO[[I_1]]$-modules.
Since $I_1$ is a pro-$p$ group there is only one indecomposable projective object in the category, namely $\OO[[I_1]]$, and thus
a pseudocompact $\OO[[I_1]]$-module is projective if and only if it is
pro-free, which means that it is isomorphic to $\prod_{j\in J} \OO[[I_1]]$
for some indexing set $J$.

Let $\Gamma:=1+p \Zp$. We identify $\Gamma$ with the image in $G_{\Qp}^{\ab}$ of the wild inertia subgroup of $G_{\Qp}$.
We may identify $\Lambda$ with the completed group algebra of the pro-$p$ completion of $G_{\Qp}^{\ab}$,
which is isomorphic to $\Gamma \times \Zp$. There is an isomorphism
of $\OO$-algebras  $\Lambda\cong \OO[[\Gamma]] [[x]]$, and in particular $\Eins^{\univ}$ is a pro-free and hence  projective $\OO[[\Gamma]]$-module.

The restriction of $\wP^{\psi}$ to $K$ is projective in the category of pseudocompact $\OO[[K]]$-modules on which $Z\cap K$ acts by the
central character $\psi^{-1}$, \cite[Corollary 5.3]{paskunasBM}. By restricting further to $I_1$, we deduce that
$\wP^{\psi}$ is projective in the category of pseudocompact $\OO[[I_1]]$-modules, where $Z_1:= I_1\cap Z$ acts by the central character
$\psi^{-1}$.

Since $p>2$ there is a character $\delta: \Gamma\rightarrow \OO^{\times}$, such that $\delta^2=\psi$. Twisting by characters preserves projectivity,
so that $\Eins^{\univ} \otimes \delta$ is a projective in the category of pseudocompact $\OO[[\Gamma]]$-modules and
$\wP^{\psi}\otimes (\delta\circ\det)$ is projective in the category of pseudocompact $\OO[[I_1]]$-modules on which $Z_1$ acts trivially. We may identify this last category with the category of pseudocompact $\OO[[I_1/Z_1]]$-modules.

We have an isomorphism of $I_1$-representations
$$\wP^{\psi}\wtimes_{\OO} \Eins^{\univ}\cong (\wP^{\psi}\otimes \delta
\circ \det)\wtimes_{\OO} ( \Eins^{\univ} \otimes_{\OO} \delta),$$ where
$I_1$ acts on $\Eins^{\univ} \otimes_{\OO} \delta$ via the homomorphism
$I_1\rightarrow \Gamma$, $g\mapsto (\det g)^{-1}$. We may therefore assume that $\psi|_{Z_1}$ is trivial, so that $\wP^{\psi}$ is projective in the category of pseudocompact
$\OO[[I_1/Z_1]]$-modules. Thus there are indexing sets $J_1$ and $J_2$ such that
$$ \wP^{\psi}|_{I_1} \cong \prod_{j_1\in J} \OO[[I_1/Z_1]], \quad \Eins^{\univ} \cong \prod_{j_2\in J_2} \OO[[\Gamma]]$$
where the first isomorphism is an isomorphism of pseudocompact
$\OO[[I_1]]$-modules, and the second isomorphism is an isomorphism of
pseudocompact $\OO[[\Gamma]]$-modules.
Since completed tensor products commute with products, we obtain
an isomorphism of pseudocompact $\OO[[I_1]]$-modules:
$$ (\wP^{\psi}\wtimes_{\OO} \Eins^{\univ})\cong \prod_{j_1\in J_1} \prod_{j_2\in J_2} \OO[[I_1/Z_1]]\wtimes_{\OO} \OO[[\Gamma]].$$
Since $p>2$, the determinant induces an isomorphism between $Z_1$ and $\Gamma$. Thus the map
$I_1\rightarrow I_1/Z_1 \times \Gamma$, $g \mapsto (g Z_1, (\det g)^{-1})$ is an isomorphism of groups. The isomorphism
induces a natural isomorphism of $\OO[[I_1]]$-modules, $\OO[[I_1]]\cong \OO[[I_1/Z_1]]\wtimes_{\OO} \OO[[\Gamma]]$.
We conclude that $\wP^{\psi}\wtimes_{\OO} \Eins^{\univ}$ is a pro-free and hence a projective $\OO[[I_1]]$-module.
\end{proof}

\begin{remark}
	We will see in Theorem~\ref{final_comparison} below that in fact
$\wP^{\psi}\wtimes_{\OO} \Eins^{\univ}$ is a projective object
of $\dualcat(\OO)$.   It is not so difficult to prove this directly,
but we have found it more convenient to deduce it as part of
the general formalism of arithmetic actions.\end{remark}

\begin{prop}\label{manipulate} Let $\sigma=\sigma_{a,b}$, and let $\delta: \Q_p^{\times}\rightarrow \cO^{\times}$ be a character
that is trivial modulo $\varpi$, chosen so that $\psi\delta^2|_{\Z_p^{\times}}$ is the central
character of $\sigmao$. Then there is a natural isomorphism of $R_p$-modules
$$ \Hom^{\cont}_{\OO[[K]]}(\wP^{\psi}\wtimes_{\OO}\Eins^{\univ}, (\sigmao)^d)^d\overset{\cong}{\longrightarrow}
\Hom^{\cont}_{\OO[[K]]}(\wP^{\psi\delta^2}, (\sigmao)^d)^d \wtimes_{\OO} \Lambda^{\ur},$$
where $R_p$ acts on the left hand side via the isomorphism $R_p\cong R_p^{\psi}\wtimes_{\OO} \Lambda$ and on the right hand side
via the isomorphism $R_p\cong R_p^{\psi\delta^2}\wtimes_{\OO} \Lambda$.
\end{prop}
\begin{proof} Using Lemma \ref{lem:twisting} we may assume that $\psi|_{K\cap Z}$ is the central character of
$\sigma$. We note that $\Eins^{\univ}\otimes_{\Lambda} \Lambda^{\ur}$ is the largest quotient of $\Eins^{\univ}$ on which $K\cap Z$ acts trivially. Since the central character of $(\sigmao)^d$ is $\psi^{-1}|_{\Z_p^{\times}}$, and the central character of
$\wP^{\psi}$ is $\psi^{-1}$, we have a natural isomorphism
$$ \Hom^{\cont}_{\OO[[K]]}(\wP^{\psi}\wtimes_{\OO}\Eins^{\univ}, (\sigmao)^d)^d\overset{\cong}{\longrightarrow}
\Hom^{\cont}_{\OO[[K]]}(\wP^{\psi}\wtimes_{\OO}(\Eins^{\univ}\otimes_{\Lambda} \Lambda^{\ur}), (\sigmao)^d)^d.$$
Since $K$ acts trivially on $\Eins^{\univ}\otimes_{\Lambda} \Lambda^{\ur}$, we have an isomorphism:
$$ \Hom^{\cont}_{\OO[[K]]}(\wP^{\psi}\wtimes_{\OO}(\Eins^{\univ}\otimes_{\Lambda} \Lambda^{\ur}), (\sigmao)^d)^d
\cong \Hom^{\cont}_{\OO[[K]]}(\wP^{\psi}\wtimes_{\OO}\Lambda^{\ur}, (\sigmao)^d)^d,$$
where $\Lambda^{\ur}$ carries a trivial $K$-action and $R_p$ acts by the isomorphism $R_p\cong R_p^{\psi}\wtimes_{\OO} \Lambda$.
To finish the proof
we need to construct a natural isomorphism of $R^{\psi}_p\wtimes_{\OO} \Lambda$-modules:
\numequation\label{delightful_duals}
\Hom^{\cont}_{\OO[[K]]}(\wP^{\psi}\wtimes_{\OO}\Lambda^{\ur}, (\sigmao)^d)^d\cong
\Hom^{\cont}_{\OO[[K]]}(\wP^{\psi}, (\sigmao)^d)^d\wtimes_{\OO} \Lambda^{\ur}.
\end{equation}
Both sides of \eqref{delightful_duals} are finitely generated $R_p^{\psi}\wtimes_{\OO} \Lambda^{\ur}$-modules. The $\mm$-adic topology on $R_p^{\psi}\wtimes_{\OO} \Lambda^{\ur}$ induces a
topology on them, and makes them into pseudo-compact, $\OO$-torsion free $\OO$-modules. It is
therefore enough to construct a natural  isomorphism between the Schikhof duals of both sides of
\eqref{delightful_duals}. To ease the notation we let $A=\wP^{\psi}$, $B=\Lambda^{\ur}$,
$C=(\sigmao)^d$. Since for a pseudo-compact $\OO$-module $D$, we have
$D^d=\Hom^{\cont}_{\OO}(D, \OO)\cong \varprojlim_{n}
\Hom_{\OO}^{\cont}(D, \OO/\varpi^n)$, using the adjointness between $\wtimes_{\OO}$ and $\Hom^{\cont}_{\OO}$ (see   \cite[Lem. 2.4]{MR0202790}), we obtain natural isomorphisms
 \begin{displaymath}
 \begin{split}
 (\Hom^{\cont}_\OO( A, C)^d \wtimes_{\OO} B)^{d}&\cong \Hom^{\cont}_{\OO}( B, (\Hom^{\cont}_\OO(A, C)^{d})^{d})\\
 &\cong
 \Hom^{\cont}_{\OO}(B, \Hom^{\cont}_{\OO}(A,C))\cong \Hom^{\cont}_{\OO}(A\wtimes_{\OO} B, C),
 \end{split}
\end{displaymath}
and hence a natural isomorphism
$$ \Hom^{\cont}_{\OO}(A\wtimes_{\OO} B, C)^d \cong \Hom^{\cont}_{\OO}(A, C)^d\wtimes_{\OO} B.$$
Since the isomorphism is natural and $K$ acts trivially on $B$ we obtain a natural  isomorphism
$$ \Hom^{\cont}_{\OO[[K]]}(A\wtimes_{\OO} B, C)^d \cong \Hom^{\cont}_{\OO[[K]]}(A, C)^d\wtimes_{\OO}B.$$
Since the action of $R^{\psi}_p$ commutes with the action of $K$ and the isomorphism is natural, we deduce that \eqref{delightful_duals} holds.
\end{proof}

\begin{prop}\label{M_sigma_CM_loc_free} For any~$\sigma$, the action of $R_p\cong R_p^{\psi}\wtimes_{\OO} \Lambda$ on
$$M'(\sigmao):=\Hom^{\cont}_{\OO[[K]]}(\wP^{\psi}\wtimes_{\OO} \Eins^{\univ}, (\sigmao)^d)^d$$ factors through $R_p(\sigma)$. Moreover, $M'(\sigmao)$ is a maximal Cohen--Macaulay $R_p(\sigma)$-module and $M'(\sigmao)[1/p]$ is a locally free $R_p(\sigma)[1/p]$-module of rank $1$.
\end{prop}
\begin{remark}\label{M'(sigmao)=M(sigmao)}
  Recall that $M(\sigmao)$ was defined in the last section. Soon we will see from Theorem \ref{final_comparison} below that $\wP \iso \wP^{\psi}\cotimes_{\OO} \Eins^{\univ}$, thus also that $M(\sigmao)\simeq M'(\sigmao)$.
\end{remark}

\begin{proof} If the central character of $\sigma$ is not congruent to $\psi|_{\Z_p^{\times}}$ modulo $\varpi$, then
both $R_p(\sigma)$ and $M'(\sigmao)$ are zero. Otherwise, there is a character $\delta: \Qp^{\times}\rightarrow \OO^{\times}$ trivial modulo $\varpi$
such  that $(\psi \delta^2)|_{\Z_p^{\times}}$ is equal to the central character of $\sigma$. Proposition \ref{manipulate} gives us an isomorphism
of $R_p$-modules
\numequation\label{decompose_module}
 M'(\sigmao)\cong M^{\psi\delta^2}(\sigmao)\wtimes_{\OO} \Lambda^{\ur},
 \end{equation}
where  $M^{\psi\delta^2}(\sigmao):= \Hom^{\cont}_{\OO[[K]]}(\wP^{\psi\delta^2}, (\sigmao)^d)^d$, and the action of $R_p$ on the right hand side
is given by $R_p\cong R_p^{\psi\delta^2}\wtimes_{\OO} \Lambda$. The action of $R_p^{\psi\delta^2}$ on $M^{\psi\delta^2}(\sigmao)$ factors through
the action of $R_p^{\psi\delta^2}(\sigma)$ and makes it into a maximal Cohen--Macaulay $R_p^{\psi\delta^2}(\sigma)$-module
\cite[Cor. 6.4, 6.5]{paskunasBM}. Since $\Lambda^{\ur}\cong \OO[[x]]$, we conclude that $M^{\psi\delta^2}(\sigmao)\wtimes_{\OO} \Lambda^{\ur}$
is a maximal Cohen--Macaulay $R_p^{\psi\delta^2}(\sigma)\wtimes_{\OO}
\Lambda^{\ur}$-module. Using~(\ref{iso_twist_pst}) we see that $M'(\sigmao)$ is a
maximal Cohen--Macaulay $R_p(\sigma)$-module. Since $R_p(\sigma)[1/p]$
is a regular ring a standard argument using the Auslander--Buchsbaum theorem shows that
$M'(\sigmao)[1/p]$ is a locally free $R_p(\sigma)[1/p]$-module. It follows from \cite[Prop. 4.14, 2.22]{paskunasBM} that
$$\dim_{\kappa(x)} M^{\psi\delta^2}(\sigmao)\otimes_{R^{\psi\delta^2}_p} \kappa(x)=1, \quad \forall x\in \mSpec R^{\psi\delta^2}_p(\sigma)[1/p].$$ This together
with \eqref{decompose_module} gives us $$\dim_{\kappa(x)} M'(\sigmao)\otimes_{R_p} \kappa(x)=1,  \quad \forall x\in \mSpec R_p(\sigma)[1/p].$$
Hence, $M'(\sigmao)[1/p]$ is a locally free $R_p(\sigma)[1/p]$-module of rank $1$.
\end{proof}

The natural action of $\cH(\sigmao)$ on $M'(\sigmao)$ commutes with the action of $R_p(\sigma)$, and hence induces an action of $\cH(\sigma)$ on $M'(\sigmao)[1/p]$. Since $M'(\sigmao)[1/p]$ is
 locally free of rank $1$ over $R_p(\sigma)[1/p]$ by Proposition \ref{M_sigma_CM_loc_free}, we obtain
 a homomorphism of $E$-algebras:
 $$ \alpha: \cH(\sigma)\rightarrow \End_{R_p(\sigma)[1/p]}( M'(\sigmao)[1/p])\cong R_p(\sigma)[1/p].$$

\begin{prop}\label{AA4_holds} The map $\alpha: \cH(\sigma)\rightarrow R_p(\sigma)[1/p]$ coincides with the map
$\eta: \cH(\sigma)\rightarrow R_p(\sigma)[1/p]$ constructed in \cite[Thm. 4.1]{Gpatch}.
\end{prop}
\begin{proof} It is enough to show that the specialisations of $\alpha$ and $\eta$ at $x$ coincide for $x$ in  a Zariski dense subset of  $\mSpec R_p(\sigma)[1/p]$.  The isomorphism
$R_p\cong R_p^{\psi} \wtimes_{\OO} \Lambda$ maps $x$ to a pair $(y, z)$,
where $y\in \mSpec R_p^{\psi}[1/p]$ and $z\in \mSpec \Lambda[1/p]$, so that
if $r_x^{\univ}$, $r_y^{\univ, \psi}$ and $\Eins_z^{\univ}$ are Galois representations
corresponding to $x$, $y$ and $z$ respectively then
$$ r_x^{\univ}\cong r_y^{\univ,\psi} \otimes \Eins_z^{\univ}.$$
Let
$$\Pi_x:=\Hom^{\cont}_{\OO}( ( \wP^{\psi}\wtimes_{\OO} \Eins^{\univ})\wtimes_{R_p} \OO_{\kappa(x)}, E),$$

$$\Pi_y:=\Hom^{\cont}_{\OO}(  \wP^{\psi}\wtimes_{R_p^{\psi}} \OO_{\kappa(y)}, E).$$
Then both are unitary $G$-Banach space representations and we have
$$\Pi_x\cong \Pi_y \otimes (\Eins^{\univ}_z \circ \det).$$
It follows from \cite[Prop. 2.22]{paskunasBM}
that $\Hom_K(\sigma, \Pi_x)$ is a one-dimensional $\kappa(x)$-vector space, and the action
of $\cH(\sigma)$ on it coincides with the specialisation of $\alpha$ at $x$, which can be written
as the composite $$\cH(\sigma)\rightarrow \End_{\kappa(x)}(M'(\sigmao)\otimes_{R_p} \kappa(x))\cong
\kappa(x).$$ Since $\sigma$ is an algebraic representation of $K$, we have
$$\Hom_K(\sigma, \Pi_x)\cong \Hom_K(\sigma, \Pi_x^{\lalg}),$$ where
$\Pi_x^{\lalg}$ is the subspace of locally algebraic vectors in $\Pi_x$.

It follows from the main result of \cite{paskunasimage} that the  specialisation $\Pi_x$ of $\wP^{\psi}$
coincides with the Banach space representation attached to $r_x^{\univ}$
via the $p$-adic local Langlands correspondence.
It is then a
consequence of the construction of the appropriate cases of
the $p$-adic local Langlands correspondence (that is, the construction
of~$\Pi_x$) that there is an embedding
$$\pi_{\sm}(r_x^{\univ})\otimes \pi_{\alg}(r_x^{\univ})
\hookrightarrow \Pi_x^{\lalg},$$
where
$\pi_{\sm}(r_x^{\univ})=r_p^{-1}(\WD(r_x^{\univ})^{F-\mathrm{ss}})$
 is the smooth representation of $G$
corresponding to the Weil--Deligne representation associated to $r_x^{\univ}$ by the classical Langlands
correspondence~$r_p$ (normalised as in~Section~\ref{subsec:notation}),
and $\pi_{\alg}(r_x^{\univ})$ is the algebraic
representation of $G$ whose restriction to $K$ is equal to $\sigma$.
Indeed, if the representation~$r_x^\univ$ is irreducible,
then (for a Zariski dense set of~$x$)~$\Pi_x$ is a completion of~$\pi_{\sm}(r_x^{\univ})\otimes
\pi_{\alg}(r_x^{\univ})$ by the main result of~\cite{berger-breuil}
(see in particular~\cite[Thm.\ 4.3.1]{berger-breuil}), while in the
case that~$r_x^\univ$ is reducible, the result follows from the
explicit description of $\Pi_x$ in~\cite{MR2667890}.

  Since we have already noted that $\Hom_K(\sigma, \Pi_x)$ is one-dimensional,
  we find that in fact
$$\Hom_K(\sigma, \Pi_x)\cong
\Hom_K\bigl(\sigma,  \pi_{\sm}(r_x^{\univ})\otimes \pi_{\alg}(r_x^{\univ})\bigr)
\cong \pi_{\sm}(r_x^{\univ})^K,$$
and the right-hand side of this isomorphism is indeed a one-dimensional
vector space on which $\cH(\sigma)$ acts via the specialisation of $\eta$.
  \end{proof}

Let~$R_p$ act on $\wP^{\psi}\wtimes_{\OO} \Eins^{\univ}$ via the
isomorphism $R_p\cong R_p^{\psi}\wtimes_{\OO} \Lambda$, where (as
throughout this section) the action
of $R_p^{\psi}$ on $\wP^{\psi}$ is via the isomorphism
$R_p^{\psi}\cong \End_{\dualcat^{\psi}(\OO)}(\wP^{\psi})$ constructed in \cite{paskunasimage}.
\begin{thm}\label{final_comparison} $\wP^{\psi}\wtimes_{\OO} \Eins^{\univ}$ is a projective envelope of $\pi^{\vee}$ in
$\dualcat(\OO)$, and the action of $R_p$ on $\wP^{\psi}\wtimes_{\OO} \Eins^{\univ}$ is arithmetic.
\end{thm}
\begin{proof}
We will show that the action of $R_p$ on $\wP^{\psi}\wtimes_{\OO}
\Eins^{\univ}$ satisfies the axioms (AA1) -- (AA4) with $d=0$; then
the action is arithmetic by definition, and $\wP^{\psi}\wtimes_{\OO}
\Eins^{\univ}$ is a projective envelope of $\pi^{\vee}$
in $\dualcat(\OO)$ by Theorem~\ref{main} (applied with $d= 0 $
and with $M_{\infty}$ taken to be $\wP^{\psi}\cotimes_{\OO} \Eins^{\univ}$).

It is shown in \cite[Propositition 6.1]{paskunasBM} that
$\F\wtimes_{R_p^{\psi}}\wP^{\psi}$ is a finitely generated
$\OO[[K]]$-module, so the topological version of Nakayama's
lemma implies that $\wP^{\psi}$ is a finitely generated $R_p^{\psi}[[K]]$-module. Since $\Eins^{\univ}$ is a
free $\Lambda$-module of rank $1$, $\wP^{\psi}\wtimes_{\OO}\Eins^{\univ}$ is a finitely generated module over
$(R_p^{\psi}\wtimes_{\OO} \Lambda)[[K]]$, and so (AA1) holds. Proposition \ref{AA2_holds}
implies that (AA2) holds. Proposition \ref{M_sigma_CM_loc_free}
implies that (AA3) holds (indeed, it shows that the support
of~$M'(\sigma^\circ)$ is all of~$R_p(\sigma)[1/p]$).
Proposition \ref{AA4_holds} implies that (AA4) holds.\end{proof}

  Recall that for each fixed
central character $\psi: Z\rightarrow \OO^{\times}$, there is an exact functor
$\cV$ from $\dualcat^{\psi}(\OO)$ to the category of continuous $G_{\Qp}$-representations
on compact $\OO$-modules. This is a modification of the functor introduced by Colmez
in \cite{MR2642409}, see \cite[\S 5.7]{paskunasimage} for details (we additionally have to twist
the functor in \cite{paskunasimage} by the inverse of the cyclotomic character to get the
desired relationship between the determinant of the Galois representations and the
central character of $\GL_2(\Qp)$-representations.) If $\Pi$ is an admissible unitary $E$-Banach space representation
of $G$ with central character $\psi$, and $\Theta$ is an open bounded $G$-invariant lattice
in $\Pi$ then the Schikhof dual of $\Theta$  is an
object of $\dualcat^{\psi}(\OO)$ and $\cV(\Pi):=\cV(\Theta^d)[1/p]$
does not depend on  the choice of $\Theta$.

The representation $\wP^{\psi}$ satisfies the conditions (N0), (N1), (N2) of
\cite[\S4]{paskunasBM} by \cite[Prop. 6.1]{paskunasBM}. In particular, we have
$\cV(\wP^{\psi})\cong r^{\univ, \psi}$ as $R^{\psi}_p[[G_{\Qp}]]$-modules.

Let $R_{\infty}=R_p[[x_1, \ldots, x_d]]$ and let $M_{\infty}$ be an $R_{\infty}[G]$-module satisfying
the axioms (AA1)--(AA4). To $x\in \mSpec R_{\infty}[1/p]$ we associate a unitary $\kappa(x)$-Banach
space representation of $G$,
$\Pi_{\infty, x}:=\Hom_{\OO}^{\cont}(M_{\infty}\wtimes_{R_{\infty}} \OO_{\kappa(x)}, E).$
The map $R_p\rightarrow R_{\infty}$ induces a map $R_p\rightarrow \kappa(x)$ and we let
 $r_x^{\univ}:= r^{\univ}\otimes_{R_p} \kappa(x)$.

\begin{cor}\label{specialize} We have an isomorphism of Galois representations $\cV(\Pi_{\infty,x})\cong r_x^{\univ}$.
In particular, $\Pi_{\infty, x}\neq 0$ for all $x\in \mSpec R_{\infty}[1/p]$.
\end{cor}
\begin{proof} Theorem \ref{main} allows us to assume that $R_{\infty}=R_p$ and
$M_{\infty}=\wP$ is a projective envelope of $\pi^{\vee}$ in $\dualcat(\OO)$. Since
an arithmetic action of $R_p$ on $\wP$ is unique by  Theorem \ref{unique_arithmetic},
using Theorem \ref{final_comparison} we may assume that
$R_{\infty}=R_p^{\psi}\wtimes_{\OO} \Lambda$ and $M_{\infty}=\wP^{\psi}\wtimes_{\OO} \Eins^{\univ}$.
Then, with the notation introduced in the course of the proof of Proposition \ref{AA4_holds},  $x$ corresponds to a pair $(y, z)$,
$\Pi_{\infty, x}= \Pi_x\cong \Pi_y \otimes (\Eins^{\univ}_z\circ \det)$
as in the proof of Proposition \ref{AA4_holds}. It follows from
\cite[Lem. 4.3]{paskunasBM} that $\cV(\Pi_y)\cong r_y^{\univ, \psi}$. Since $\cV$ is compatible
with twisting by characters we have $\cV(\Pi_x)\cong \cV(\Pi_y)\otimes \Eins^{\univ}_z\cong
r^{\univ}_x$, as required.
\end{proof}

\begin{cor} $R_{\infty}$ acts faithfully on $M_{\infty}$.
\end{cor}
\begin{proof} It follows from Corollary \ref{specialize} that $M_{\infty}\otimes_{R_{\infty}}\kappa(x)$
is non-zero for all $x\in \mSpec R_{\infty}[1/p]$. Since $R_p$ and hence $R_{\infty}$ are reduced
we deduce that the action is faithful.
\end{proof}

We now use the results of~\cite{paskunasimage} to
describe~$\F\wtimes_{R_p} \wP$. We will use this result in Corollary~\ref{cor:
  the m-torsion in completed cohomology} below to describe the
$\m$-torsion in the completed cohomology of a modular curve,
where~$\m$ is a maximal ideal in a Hecke algebra.

\begin{prop}
  \label{prop: an atome automorphe by any other name would smell as sweet}
  The representation~$\pi^\vee$ occurs as a subquotient
  of~$\F\wtimes_{R_p} \wP$ with multiplicity one. More precisely, if we let
  $\kappa(\rbar):=(\F\wtimes_{R_p} \wP)^{\vee}$ then the $G$-socle filtration of $\kappa(\rbar)$ is described as follows:
   \begin{enumerate}
\item If $\rbar$ is irreducible then
  $\kappa(\rbar)\cong \pi$.
\item  If $\rbar$ is a generic
  non-split extension of characters \emph{(}so the ratio of the two
  characters is not $1, \omega^{\pm 1}$\emph{)}, the $G$-socle
  filtration of $\kappa(\rbar)$ has length two, with graded pieces
  consisting of $\pi$ and of the other principal series representation
  in the block of $\pi$.
\item  If
  $\rbar\cong \bigl ( \begin{smallmatrix} 1 & \ast \\ 0 &
    \omega \end{smallmatrix}\bigr)\otimes\chi$ then the $G$-socle
  filtration of $\kappa(\rbar)$ has length three and
  the graded pieces are $\pi$, the twist by $\chi\circ \det$ of the
  Steinberg representation, and two copies of the one-dimensional
  representation $\chi\circ\det $.
\end{enumerate}
\end{prop}
\begin{proof} We choose any continuous character $\psi$ such that $\psi\varepsilon^{-1}$ lifts
$\det \rbar$. It follows from Theorem \ref{final_comparison} that
$\F\wtimes_{R_p} \wP\cong \F \wtimes_{R^{\psi}_p} \wP^{\psi}$. Since we can identify
the endomorphism ring of $\wP^{\psi}$ with $R^{\psi}_p$, see \eqref{def_endo}, it follows from
Lemma 3.7 in \cite{paskunasimage} applied with $S=\pi^{\vee}$ that any $Q$ in $\dualcat^{\psi}(\OO)$
satisfying the hypotheses (H1)--(H4) of \cite[\S 3]{paskunasimage} is isomorphic to
$\F\wtimes_{R_p^{\psi}} \wP^{\psi}$, so that $\kappa(\rbar)\cong Q^{\vee}$. (We leave the reader to check that (H5) is not used to prove this part of Lemma 3.7 in \cite{paskunasimage}.) In all the cases $Q$ has been constructed explicitly
in \cite{paskunasimage} and it is immediate from the construction of $Q$ that the assertions about the socle filtration hold,
see \cite[Propositions 6.1, 8.3, Remark 10.33]{paskunasimage}.
\end{proof}

\begin{rem}\label{compare_to_atome_automorphe} In the first two cases of Proposition
\ref{prop: an atome automorphe by any other name would smell as sweet}, $\kappa(\rbar)$ coincides
with what Colmez calls the \textit{atome automorphe} in \cite[\S VII.4]{MR2642409}. In the last case, $\kappa(\rbar)$
has an extra  copy of $\chi\circ \det$. This has to do with the fact that Colmez requires that his atome automorphe lift to irreducible unitary Banach space representations of $G$.
\end{rem}

\begin{cor}\label{identify_ring} There is a natural isomorphism $R_p\iso \End_{\dualcat(\OO)}(\wP)$.
\end{cor}\begin{proof}
 Theorems~\ref{unique_arithmetic} and~\ref{final_comparison} yield an isomorphism
$\wP \iso \wP^{\psi}\cotimes_{\OO} \Eins^{\univ}$ as $R_p[G]$-modules via
the isomorphism $R_p \iso R_p^{\psi} \cotimes_{\OO} \Lambda$ given by~(\ref{iso_twist_univ}).
Now $\End_{G}^{\cont}(\Eins^{\univ}) = \Lambda$,
while~(\ref{def_endo})
gives a natural isomorphism
$R_p^{\psi} \iso \End_{\dualcat(\OO)}(\wP^{\psi}).$
Thus the corollary amounts to proving that the natural homomorphism
$$\End_{\dualcat(\OO)}(\wP^{\psi}) \cotimes_{\OO} \End_{G}^{\cont}(\Eins^{\univ})
\to \End_{\dualcat(\OO)}(\wP^{\psi}\cotimes_{\OO} \Eins^{\univ})$$
is an isomorphism. The map is an injection of pseudo-compact $\OO$-algebras. This makes the ring 
$\End_{\dualcat(\OO)}(\wP)$ into a compact $R_p$-module.
By the topological version of Nakayama's lemma, in order to show that
the map is surjective it is enough to show that $\F\wtimes_{R_p} \End_{\dualcat(\OO)}(\wP)$
is a one dimensional $\F$-vector space. Since $\wP$ is projective, we have an isomorphism:
$$\F\wtimes_{R_p} \End_{\dualcat(\OO)}(\wP) \cong \Hom_{\dualcat(\OO)}( \wP, \F\wtimes_{R_p} \wP).$$

By Proposition~\ref{prop: an atome automorphe by any other name would smell as
    sweet}, $\pi^{\vee}$ occurs as a subquotient of~$\F\wtimes_{R_p} \wP$ with multiplicity one. Since $\wP$ is a projective envelope of $\pi^{\vee}$
and $\End_{\dualcat(\OO)}(\pi^{\vee})=\F$,  this implies that
$\Hom_{\dualcat(\OO)}( \wP, \F\wtimes_{R_p} \wP)$ is a one dimensional
$\F$-vector space, as required.
\end{proof}

\subsection{Endomorphism rings and deformation rings}
\label{subsec:rings}
We maintain the notation of the previous sections; in particular, $\tP$
denotes a projective envelope of $\pi^{\vee}$ in $\dualcat(\cO)$.
If we write $\tR := \End_{\dualcat(\OO)}(\tP)$, then
the arithmetic action of $R_p$ on $\tP$ provided by
Theorem~\ref{unique_arithmetic} gives a morphism $R_p \to \tR$,
which Corollary~\ref{identify_ring} shows is an isomorphism.
The proof of that Corollary uses the analogous statement proved
in~\cite{paskunasimage} (when the central character is fixed),
a key input to the proof of which
is Colmez's functor from $\GL_2(\Q_p)$-representation to Galois representations.
It is natural to ask (especially in light of possible
generalizations) whether this isomorphism
can be proved using just the methods of the present paper,
without appealing to Colmez's results.
In this subsection we address this question, to the extent that we can.

We begin by noting that
since $\tP$ is a projective envelope
of the absolutely irreducible representation $\pi^\vee$, the ring $\tR$ is
a local ring.
We will furthermore give a proof that it is commutative,
from the perspective of this paper.
As already noted,
this result is not new.   Indeed, in addition to being
a consequence of Corollary~\ref{identify_ring} (and thus,
essentially, of the results of \cite{paskunasimage}),
another proof is given in \cite{CDP} (see Cor.~2.22 of that
paper).   This latter proof uses the capturing techniques that
we are also employing in the present paper, and (since it
is easy to do so) we present a slightly
rephrased version of the argument here, in order to illustrate how
it fits naturally into our present perspective.

\begin{prop}
	\label{prop:commutative}
The ring $\tR$ is commutative.
\end{prop}

\begin{proof}We first prove that the image of $R_p$ in $\tR$ lies in the
	centre of $\tR$.  To see this, suppose that $\phi \in \tR$.
	By Proposition~\ref{prop:capture implies unique action}, to show
	that $\phi$ commutes with the action of $R_p$, it suffices to
	show that $\phi$ commutes with the action of $R_p(\sigma)[1/p]$
	on $M(\sigmao)[1/p]$ for each $\sigma$.
	Since the action of $\cH(\sigma)$ on $M(\sigmao)[1/p]$
	depends only on the $G$-action on $\tP$, we see that $\phi$
	commutes with the $\cH(\sigma)$-action on $M(\sigmao)[1/p]$.
	The desired result then follows from Proposition~\ref{prop: generic
	fibre of cristabelline deformation ring}.

        To see that $\tR$ is commutative, we again apply
	Proposition~\ref{prop:capture implies unique action}, by which
	it suffices to show that $\tR$ acts on each $M(\sigmao)[1/p]$
	through a commutative quotient.  This follows from the
	fact that each $M(\sigmao)[1/p]$ is locally free of rank one
	over its support in $\Spec R_p(\sigma)[1/p],$ and the fact
	that (by the result of the previous paragraph) the $\tR$-action
	commutes with the $R_p$-action.
\end{proof}

\begin{remark}
	As for proving the stronger result that the canonical map $R_p \to \tR$ is an isomorphism,
	in the forthcoming paper~\cite{EmertonPaskunas} two of us (ME and VP) will
	establish the {\em injectivity} of the morphism $R_p \to \tR$. (In fact
        we will prove a result in the more general context of \cite{Gpatch};
	in particular, our arguments won't rely on any special aspects
	of the $\GL_2(\Q_p)$ situation, such as the existence of Colmez's
	functors.)
	However, proving the {\em surjectivity} of this morphism seems to be more
	difficult, and we currently don't know a proof of this surjectivity
	that avoids appealing to the theory of Colmez's functor from
	$\GL_2(\Q_p)$-representations to $G_{\Q_p}$-representations.
\end{remark}

\subsection{Speculations in the residually scalar semi-simplification
  case}\label{subsec: wild speculations} Suppose for the rest of this section that $\rbar\cong
\bigl( \begin{smallmatrix}  \chi & \ast\\ 0 &
  \chi\end{smallmatrix}\bigr)$ for some $\chi$; so in particular
$\rbar$ does not satisfy Assumption~\ref{assumption: rbar is generic
  enough}.
It is natural to ask what the modules $M_{\infty}$ constructed in
\cite{Gpatch} look like in this case;
we give a speculative answer below. By twisting we may assume that $\chi$ is the trivial character.
Let $\pi= (\Ind_B^G \omega\otimes \Eins)_{\sm}$, and let $\wP$ be  a projective envelope of $\pi^{\vee}$.
We first give a conjectural description of $\End_{\dualcat(\OO)}(\wP)$,
under the assumption $p>2$.

Let $D^{\ps}$ be a functor from the category $\mathfrak A$ of complete
local noetherian
$\OO$-algebras
with residue field $\F$ to the category of
sets, that assigns to $A\in \mathfrak A$ the set of pairs of  functions
 $(t, d): \gal\rightarrow A$, where:
 \begin{itemize}
 \item $d:\gal \rightarrow A^{\times}$ is a continuous group
   homomorphism, congruent to $\det \rbar$ modulo $\mm_A$,
 \item    $t: \gal\rightarrow A$ is a continuous function with
   $t(1)=2$, and,
 \item  for all $g, h\in \gal$, we have:
   \begin{enumerate}
   \item $t(g)\equiv \tr \rbar(g) \pmod{\mm_A}$;
   \item $t(gh)=t(hg)$;
   \item $d(g) t(g^{-1}h)-t(g)t(h)+ t(gh)=0$.
   \end{enumerate}
 \end{itemize}
(The ``ps'' is for ``pseudocharacter''. By~\cite[Lem.\ 1.9]{chenevierpseudo},
$D^{\ps}(A)$ is the set of pseudocharacters deforming the
pseudocharacter~$(\tr\rbar,\det\rbar)$ associated to~$\rbar$.)
This functor is representable by a complete local noetherian $\OO$-algebra $R^{\ps}$. Let
$(t^{\univ}, d^{\univ}): G_{\Qp}\rightarrow R^{\ps}$ be the universal object.  We expect that there is a
natural isomorphism of $\OO$-algebras
\numequation\label{conjectural_description}
\wE:= \End_{\dualcat(\OO)}(\wP)\cong (R^{\ps}[[G_{\Qp}]]/J)^{\op},
 \end{equation}
where $J$ is the closed two sided ideal of $R^{\ps}[[G_{\Qp}]]$ generated by all the elements of the form
$g^2-t^{\univ}(g)g + d^{\univ}(g)$ for all $g\in G_{\Qp}$, and the
superscript $\op$ indicates the opposite algebra. We note that such an isomorphism has been established in \cite[\S 9]{paskunasimage}, when
the central character is fixed, and we expect that one can deduce \eqref{conjectural_description}
from this using the twisting techniques of the previous subsection.

Let $R_p^{\square}$ be the framed deformation ring of $\rbar$ and let $M_{\infty}$ be the patched
module constructed in \cite{Gpatch} (or the variant for the completed
cohomology of modular curves that we briefly discuss in Section~\ref{sec:
  local global} below) and let $R_{\infty}$ be the patched ring. Then $R_{\infty}$ is an
$R_p^{\square}$-algebra, and the map $R_p^{\square}\rightarrow R_{\infty}$ gives rise to a Galois
representation $r_{\infty}: G_{\Qp}\rightarrow \GL_2(R_{\infty})$ lifting $\rbar$. The pair
$(\tr r_{\infty}, \det r_{\infty})$ gives us a point in $D^{\ps}(R_{\infty})$ and hence a map
$R^{\ps}\rightarrow R_{\infty}$. Hence we obtain a homomorphism of $R^{\ps}$-algebras
$R^{\ps}[[G_{\Qp}]]\rightarrow M_{2\times 2}(R_{\infty})$.

The Cayley--Hamilton theorem implies that this
map is zero on $J$, so we obtain a left action of $R^{\ps}[[G_{\Qp}]]/J$ on
the standard module $R_{\infty}\oplus R_{\infty}$. If we admit \eqref{conjectural_description}
then we get a right action of $\wE$ on $R_{\infty}\oplus R_{\infty}$. We expect that there
are isomorphisms in $\dualcat(R_{\infty})$
\numequation\label{conj_des}
M_{\infty}\cong (R_{\infty}\oplus R_{\infty})\wtimes_{\wE} \wP\cong
R_{\infty}\wtimes_{R_p^{\square}}(R_p^{\square}\oplus R_p^{\square})\wtimes_{\wE} \wP.
\end{equation}
We note that the representation appearing on the right  hand side of this equation has been studied
by Fabian Sander in his thesis \cite{fabian}, in the setting where the central character is fixed. Motivated by
\cite[Thm. 2]{fabian} we expect  $(R_{\infty}\oplus R_{\infty})\wtimes_{\wE} \wP$
to be  projective in the category of pseudocompact $\OO[[K]]$-modules.
We do not expect $(R_{\infty}\oplus R_{\infty})\wtimes_{\wE} \wP$ to be projective
in $\dualcat(\OO)$, so the methods of Section~\ref{sec:
  proof that arithmetic action is unique} cannot directly be applied
to this case.
However, it might
be possible to prove \eqref{conj_des} using Colmez's functor. This would show that $M_{\infty}$
does not depend on the choices made in the patching process.

\section{Local-global compatibility}\label{sec: local global}In this final section, we briefly explain how the results of this
paper give a simple new proof of the local-global compatibility
theorem of~\cite{emerton2010local} (under the hypotheses that we have
imposed in this paper, which differ a little from those
of~\cite{emerton2010local}:
locally at $p$, we have excluded the case
of split $\rbar$, and have allowed a slightly different collection
of indecomposable reducible $\rbar$'s; and in the global context
we consider below, we exclude the possibility of so-called {vexing
	primes}). Applying these considerations to the
patched modules constructed in~\cite{Gpatch} allows us to prove a
local-global compatibility result for the completed cohomology of a
compact unitary group, but for ease of comparison
to~\cite{emerton2010local}, we instead briefly discuss the output of
Taylor--Wiles patching for modular curves.

Patching in this context goes back to~\cite{MR1333036}, but the
precise construction we need is not in the literature. It is, however,
essentially identical to that of~\cite{Gpatch} (or the variant for
Shimura curves presented in~\cite{scholze}), so to keep this paper at
a reasonable length we simply recall the output of the construction
here.

Let $\rhobar:G_\Q\to\GL_2(\F)$ be an absolutely irreducible odd (so
modular, by Serre's conjecture) representation, and assume that $p\ge
5$ and that $\rhobar|_{G_{\Q(\zeta_p)}}$
is irreducible. Write $\rbar:=\rhobar|_{G_{\Qp}}$, and assume that
$\rbar$ satisfies Assumption~\ref{assumption: rbar is generic
  enough}. Write $r^{\univ}:G_{\Qp}\to\GL_2(R_p)$ for the universal
deformation of~$\rbar$. Let~$N(\rhobar)$ be the
prime-to-$p$ conductor of~$\rhobar$; that is, the level of~$\rhobar$
in the sense of~\cite{MR885783}. We assume that if~$q|N(\rhobar)$
with~$q\equiv -1\pmod{p}$ and~$\rhobar|_{G_{\Q_q}}$ is irreducible
then~$\rhobar|_{I_{\Q_q}}$ is also irreducible.

\begin{rem}
  \label{rem: avoiding vexing primes} The last condition we have imposed excludes the
  so-called \emph{vexing primes}~$q$. The assumption that there are no
  vexing primes means that the Galois representations associated to
  modular forms of level~$N(\rhobar)$ are necessarily minimally
  ramified. This assumption can be removed by considering inertial
  types at such primes as in~\cite{MR1639612}. Since the arguments
  using types are standard and are orthogonal to the main concerns of
  this paper, we restrict ourselves to this simple case.

  We caution the reader that while the use of types would also allow
  us to work at certain non-minimal levels, the most naive analogues
  of Theorem~\ref{thm:local global compatibility for modular curves}
  fail to hold at arbitrary tame levels. It seems that to formulate a
  clean statement, one should pass to infinite level at a finite set
  of primes, and formulate the compatibility statement in terms of the local
  Langlands correspondence in families of~\cite{MR3250061}, as is
  done in~\cite{emerton2010local}.

  However, it does not seem to be
  easy to prove this full local-global compatibility statement using
  only the methods of the present paper; indeed, the proof
  in~\cite{emerton2010local} ultimately makes use of mod~$p$
  multiplicity one theorems that rely on $q$-expansions, whereas in
  our approach, we are only using multiplicity one theorems that
  result from our patched modules being Cohen--Macaulay, and certain
  of our local
  deformation rings being regular (namely the minimal deformation rings at places not dividing~$p$, and the deformation rings
  considered in Lemma~\ref{lem:rbar has a unique Serre weight, and the deformation ring is smooth, and given by Tp}). 
  Note that in general the (non-minimal) local
  deformation rings at places away from~$p$ need not be regular (even
  after inverting~$p$), so that carrying out the patching construction
  below would result in a ring~$R_\infty$ that was no longer formally
  smooth over~$R_p$, to which the results of Section~\ref{sec:
  proof that arithmetic action is unique} would not apply.
\end{rem}

Let $\TT$ be the usual Hecke algebra acting on (completed) homology and
cohomology
of modular curves with (tame)
level $\Gamma_1(N(\rhobar))$ and $\cO$-coefficients; so
$\TT$ is an $\cO$-algebra, generated by the operators $T_l$, $S_l$
with $l\nmid Np$. Let $\m(\rhobar)$ be the maximal ideal of $\TT$ corresponding
to $\rhobar$ (so that $T_l-\tr\rhobar(\Frob_l)$ and
$lS_l-\det\rhobar(\Frob_l)$ are both zero in $\TT/\m(\rhobar)$). Let $R_{\Q,N(\rhobar)}$ be the universal
deformation ring for deformations of~$\rhobar$ that are minimally ramified
at primes~$l\ne p$, in the sense that they have the same conductor
as~$\rhobar|_{G_{\Ql}}$ (and in particular are unramified if~$l\nmid N(\rhobar)$). Let
$\rho^{\univ}_{\Q,N(\rhobar)}: G_{\Q} \to \GL_2(R_{\Q,N(\rhobar)})$
denote the corresponding universal deformation of $\rhobar$.

We now use the notation introduced in Section~\ref{sec: recalling
  from Gpatching}, so that in particular we write
$R_\infty:=R_p\widehat{\otimes}_{\cO}\cO[[x_1,\dots,x_d]]$ for some
$d\ge 0$. Patching the completed \'etale homology of the modular curves $Y_1(N(\rhobar))$
  (and using an argument of Carayol~\cite{MR1279611}, as
  in~\cite[\S 5.5]{emerton2010local} and~\cite[\S\S6.2, 6.3]{emertongeesavitt},
  to factor out the Galois action
  on the completed cohomology; see also~\cite[\S 9]{scholze} for the
  analogous patching construction for Shimura curves),
  we obtain (for some~$d\ge 0$) an $R_\infty[G]$-module $M_\infty$ with an arithmetic action,
with the further property that there is an ideal $\mathfrak{a}_\infty$ of $R_\infty$, an
  isomorphism of local $\cO$-algebras $R_\infty/\mathfrak{a}_\infty\isoto
  R_{\Q,N(\rhobar)}$, and an isomorphism of $R_{\Q,N(\rhobar)}[G\times
  G_{\Q}]$-modules \numequation\label{eqn: M infinity gives completed cohomology }(M_\infty/\mathfrak{a}_\infty)\otimes_{R_{\Q,N(\rhobar)}}
    (\rho^{\univ}_{\Q,N(\rhobar)})^{\ast}
    \isoto
  \widetilde{H}_{1,\et}\bigl(Y_1(N(\rhobar)),\cO\bigr)_{\m(\rhobar)}.
\end{equation}
Here $\widetilde{H}_{1,\et}\bigl(Y_1(N(\rhobar)),\cO\bigr)$ denotes completed
\'etale homology, as described for example in~\cite{MR2905536}. The
action of~$G_{\Q}$
    on the left hand side is via its action on
    $(\rho^{\univ}_{\Q,N(\rhobar)})^{\ast}$, which as in~(\ref{subsubsec:duals}) 
 denotes the
    $R_{\Q,N(\rhobar)}$-linear dual of
    $\rho^{\univ}_{\Q,N(\rhobar)}$.

\begin{rem}
  \label{rem: minimally ramified}Here we have used implicitly that the
  minimally ramified local (framed) deformation rings are all smooth,
  which follows for example from~\cite[Lem.\ 2.4.19]{cht}; this
  ensures that the ring~$R_\infty$ occurring in the patching argument
  is formally smooth over~$R_p$.
\end{rem}

\begin{thm}\label{thm:local global compatibility for modular curves}Let $p>3$
  be prime, and let $\rhobar:G_\Q\to\GL_2(\F)$ be an absolutely irreducible
  odd representation, with the property that
  $\rhobar|_{G_{\Q(\zeta_p)}}$ is irreducible, $N(\rhobar)$ is not
  divisible by any vexing primes, and
  $\rhobar|_{G_{\Qp}}$ satisfies Assumption~{\em \ref{assumption: rbar is
    generic enough}}. Then
  there is an isomorphism of
  $R_{\Q,N(\rhobar)}[G\times G_{\Q}]$-modules
  \[\widetilde{H}_{1,\et}\bigl(Y_1(N(\rhobar)),\cO\bigr)_{\m(\rhobar)}\isoto\wP\,
	  \cotimes_{R_p}
	  \, (\rho^{\univ}_{\Q,N(\rhobar)})^{\ast},\]
where the completed tensor product on the right-hand side is computed by regarding
$(\rho^{\univ}_{\Q,N(\rhobar)})^{\ast}$ as a $G_{\Q}$-representation
on an $R_p$-module via the natural morphism $R_p \to R_{\Q,N(\rhobar)}$.
\end{thm}
\begin{proof}As noted before Remark \ref{rem: avoiding vexing primes}, 
$\rhobar$ is modular by Serre's
  conjecture, so in particular $M_\infty$ is not zero.
By Theorem~\ref{main} there is an isomorphism
of $R_\infty[G]$-modules \[M_\infty\cong
  \wP\wtimes_{\cO}\cO[[x_1,\dots,x_d]]
  \cong \wP\wtimes_{R_p} R_{\infty}
  .\]
  Quotienting out
by~$\mathfrak{a}_\infty$ yields an isomorphism
$$ M_{\infty}/\mathfrak{a}_{\infty} \cong \wP\wtimes_{R_p} R_{\Q,N(\rhobar)}.$$
The result now follows by tensoring both sides
with $(\rho^{\univ}_{\Q,N(\rhobar)})^{\ast}$ and applying~(\ref{eqn: M
  infinity gives completed cohomology }).\end{proof}

We now show how to compute the $\mathfrak{m}(\rhobar)$-torsion in the completed \'etale
cohomology of modular curves $\widetilde{H}^1_{\et}\bigl(Y_1(N(\rhobar)),\cO\bigr)$ as
a $\GL_2(\Qp)$-representation.

\begin{cor}\label{cor: the m-torsion in completed cohomology}Under the
 assumptions of Theorem~{\em \ref{thm:local global compatibility for
    modular curves}}, we have an isomorphism of $\F[G\times G_\Q]$-modules
$$ \widetilde{H}^1_{\et}\bigl(Y_1(N(\rhobar)),\F\bigr)[\m(\rhobar)]\simeq
 \kappa(\rhobar|_{G_{\Qp}})\otimes_{\F} \rhobar,
$$ where~$\kappa(\rhobar|_{G_{\Qp}})$ is the representation defined in
Proposition~{\em \ref{prop: an atome automorphe by any other name would smell as
    sweet}}.\end{cor}

\begin{proof} The Pontryagin dual of the left hand side is $\widetilde{H}_{1,\et}\bigl(Y_1(N(\rhobar)),\cO\bigr)\wtimes_{R_{\Q, N(\rhobar)}}\mathbb{F}$.
By Theorem~\ref{thm:local global compatibility for modular curves}, we have an isomorphism of $\F[G \times G_{\Q}]$-modules
$$\widetilde{H}_{1,\et}\bigl(Y_1(N(\rhobar)),\cO\bigr)\wtimes_{R_{\Q, N(\rhobar)}}
\mathbb{F}\simeq  (\F \wtimes_{R_p} \wP) \wtimes_{\F} \rhobar^{\ast}
\cong
\kappa(\rhobar|_{G_{\Qp}})^{\vee}\wtimes_{\F} \rhobar^{\ast},$$the last isomorphism following
from the definition of~$\kappa(\rhobar|_{G_{\Qp}})$.
Passing to Pontryagin duals (and noting that since~$\rhobar$ is an
$\F$-representation, we have $\rhobar^{\ast}=\rhobar^\vee$) gives the result.
\end{proof}

\begin{rem} Corollary~\ref{cor: the m-torsion in completed cohomology} together with
Proposition~ \ref{prop: an atome automorphe by any other name would smell as
    sweet} gives a description of the $G$-socle filtration of $\widetilde{H}^1_{\et}\bigl(Y_1(N(\rhobar)),\F\bigr)[\m(\rhobar)]$. Even more is true.

Since in Corollary \ref{identify_ring} we have identified the endomorphism ring of $\wP$ with $R_p$ and $\kappa(\rbar)$ is by definition $\F\wtimes_{R_p} \wP$,
a completely formal argument
(see the proof of Proposition 2.8 in \cite{paskunas2}) shows that $\kappa(\rbar)$ is up to isomorphism the unique representation in
$\Mod^{\mathrm{l.adm}}_{G}(\cO)$ that is maximal with respect to the following two properties:
\begin{enumerate}
\item the socle of $\kappa(\rbar)$ is $\pi$;
\item $\pi$ occurs as a subquotient of $\kappa(\rbar)$ with multiplicity $1$.
\end{enumerate}
(It is maximal in the sense that it cannot be embedded into any other strictly larger representation in
$\Mod^{\mathrm{l.adm}}_{G}(\cO)$ satisfying these two
properties.)

 Corollary~\ref{cor: the m-torsion in completed cohomology} shows
 that (after
 factoring out the $G_\Q$-action) the same characterisation
 carries over to
 $\widetilde{H}^1_{\et}\bigl(Y_1(N(\rhobar)),\F\bigr)[\m(\rhobar)]$. However, we warn the reader that a
 simple-minded application of  this recipe will not work in general,
 and in particular it fails if $\rhobar|_{G_{\Qp}}\cong
 \bigl(\begin{smallmatrix} \omega & \ast \\ 0 &
   1\end{smallmatrix}\bigr)\otimes \chi$.

More precisely, if  Assumption \ref{assumption: rbar is generic enough}
 is in force then~$\rbar$ is determined up to isomorphism by the data
 of its determinant and its unique irreducible subrepresentation. This information can be recovered from $\pi$, which in turn determines $\kappa(\rbar)$.
 If on the other hand $\rbar\cong \bigl(\begin{smallmatrix} \omega & \ast \\ 0 & 1\end{smallmatrix}\bigr)\otimes \chi$
  and $\End_{G_{\Qp}}(\rbar)=\F$ then the $G$-socle of the atome automorphe associated by Colmez
  in \cite[\S VII.4]{MR2642409} to $\rbar$, which we still call $\pi$, is the Steinberg representation twisted by $\chi\circ \det$.
  This representation still carries the information about the irreducible subrepresentation of $\rbar$ and
  the determinant of $\rbar$ but it does not determine $\rbar$ up to isomorphism, as it
  does not carry the information about the extension class
  in $\Ext^1_{G_{\Qp}}(\chi, \chi\omega)$ corresponding to $\rbar$. The maximal representation
  satisfying (1) and (2) above will contain the atome automorphe corresponding to $\rbar$ as a
  subrepresentation, but it will be strictly bigger. In fact it can be shown that it is the smallest
  representation that contains all the atomes automorphes
  corresponding to different non-zero
  extensions in $\Ext^1_{G_{\Qp}}(\chi, \chi\omega)$.
 \end{rem}

\begin{rem}
  \label{rem: we've done something new}Theorem~\ref{thm:local global
    compatibility for modular curves} and Corollary~\ref{cor: the m-torsion in
    completed cohomology}, when
  combined with Theorems~\ref{final_comparison} and~\ref{main} (which together
  show that~$\tP$ realises the usual $p$-adic local Langlands
  correspondence for~$\GL_2(\Qp)$), prove a local-global compatibility
  result for completed cohomology. They are new in the case that $\rhobar|_{G_{\Qp}}\cong \bigl ( \begin{smallmatrix} 1 & \ast \\ 0 &
    \omega \end{smallmatrix}\bigr)\otimes\chi$. In particular, they answer a question raised in Remark 1.2.9 in \cite{emerton2010local} by confirming the expectation of Remark 6.1.23 of \textit{loc.cit.}. In the other cases,
  Theorem~\ref{thm:local global  compatibility for modular curves} can be deduced from  Theorem 6.4.6
   in~\cite{emerton2010local} with $\wP$ replaced by a deformation of $\kappa(\rbar)^{\vee}$ to $R_p$,
   such that one obtains the universal deformation of $\rbar$ after applying Colmez's functor $\cV$ to it.
   If $\rhobar|_{G_{\Qp}}\cong \bigl ( \begin{smallmatrix} 1 & \ast \\ 0 &
    \omega \end{smallmatrix}\bigr)\otimes\chi$ then it can be shown that $\wP$ is not flat over $R_p$, and that
    is why the approach of \cite{emerton2010local} does not work in this case.
   \end{rem}

\begin{rem}\label{rem: division algebra case}
  Theorem~\ref{thm:local global compatibility for modular curves}
  and Corollary~\ref{cor: the m-torsion in completed cohomology} have analogues
  in more general settings when the group at $p$ is essentially $\GL_2(\Q_p)$
  (or a product of copies of $\GL_2(\Q_p)$). For example, taking~$M_\infty$ to be the
  patched module of Section~\ref{subsec: M infty}
  (as constructed in~\cite{Gpatch} for $n=2$),
  we obtain statements about the completed cohomology of unitary groups that
  are compact at infinity.

  Perhaps a case of greater interest is that of the completed
  cohomology of definite quaternion algebras over totally real fields. (One reason
  for this case to be of interest is its relationship to the cohomology of the
  Lubin-Tate tower as in~\cite[Thm. 6.2]{scholze}.) We
  expect that our results can be extended to this setting, although
  there is one wrinkle: in order to carry out Taylor--Wiles patching,
  we need to fix a central character, and as a consequence, our
  patched module has a fixed central character, and no longer
  satisfies the axioms of Section~\ref{sec: recalling
  from Gpatching}. One approach to this difficulty would be to formulate analogues of those axioms with an arbitrary fixed central
character, making use of the twisting constructions of
Section~\ref{sec:comparison} and ``capture" arguments of  \cite[\S2.1]{paskunas2}, but this leads to ugly statements.

Instead, we content ourselves with considering the case that the fixed
central character is the trivial character. In this case we can think
of our patched modules as modules for~$\PGL_2(\Qp)$, and natural
analogues of our axioms can be formulated in this setting; this is
carried out in~\cite[\S 5]{geenewtonderivedpatching}, where an
analogue of the results of Section~\ref{sec:
  proof that arithmetic action is unique} is proved. In fact, the
arguments there allow us to consider modules for a product of copies
of~$\PGL_2(\Qp)$, which is convenient when there is more than one
place lying over~$p$; accordingly, we work below with cohomology that
is completed at all primes above~$p$. (Of course, the case of
cohomology that is completed at a single prime above~$p$ can be
deduced from this by returning to finite level via taking appropriate
locally algebraic vectors.)

As explained in Remark~\ref{rem: avoiding vexing primes}, one has to
take some care with ramification at places away from~$p$, and we
therefore content ourselves with considering quaternion algebras that
split at all finite places. The reader wishing to prove extensions of
these results to more general quaternion algebras is advised to
examine the patching arguments of~\cite[\S 4]{geekisin}, which work in
this setting. We further caution the reader that we have not attempted
to check every detail of the expected result explained below.

Let $F$ be a totally real field in which $p\ge 5$ splits completely, and let~$D$ be
a quaternion algebra over $F$ that is split at all finite places and  definite
at all infinite places (note that in particular this requires~$[F:\Q]$
to be even).

Let~$\rhobar:G_F\to\GL_2(\F)$ be absolutely irreducible, and assume
that~$\rhobar|_{G_{F(\zeta_p)}}$ is irreducible, and that
$\det\rhobar=\omega^{-1}$. Suppose that~$\rhobar$ has no vexing
primes; that is, if~$v\nmid p$ is a finite place at which~$\rhobar$ is
ramified and $\mathbf{N}v\equiv -1\pmod{p}$, and~$\rhobar|_{G_{F_v}}$
is irreducible, then~$\rhobar|_{I_{F_v}}$ is also
irreducible. Finally, suppose that  for each
place~$v|p$, $\rhobar|_{G_{F_{v}}}$ satisfies
Assumption~\ref{assumption: rbar is generic enough}.

Since~$D$ splits at all finite places, we can consider the tame level subgroup $U_1(N(\rhobar))\subset \PGL_1(D\otimes\mathbb{A}_{F}^{\infty,p})\simeq \PGL_2(\A_F^{\infty,p})$
given by the image of those matrices in $\GL_2(\A_F^{\infty,p})$ that are unipotent and upper triangular modulo $N(\rhobar)$.
Let $\widetilde{H}_0(U_1(N(\rhobar)),\cO)$ (resp.\ $\widetilde{H}^0(U_1(N(\rhobar)),\cO)$) denote
the completed homology (resp.\ cohomology) of the tower of locally
symmetric spaces associated to $\PGL_1(D)$ with tame
level~$U_1(N(\rhobar))$. (Note that at finite level, the locally symmetric spaces
are just finite sets of points.)

We assume that~$\rhobar$ is modular, in the following sense: $\rhobar$
determines a maximal ideal $\m(\rhobar)$ in the spherical
Hecke algebra (generated by Hecke operators at places not dividing~$p$
at which~$\rhobar$ is unramified) acting on
$\widetilde{H}_0(U_1(N(\rhobar)),\cO)$ and we assume that
$\widetilde{H}_0(U_1(N(\rhobar)),\cO)_{\m(\rhobar)}\not = 0$.

Let $R_{\rhobar}$ be the universal deformation ring
for deformations of $\rhobar$ that are minimally ramified at places
not dividing~$p$, and which have
determinant~$\varepsilon^{-1}$. For each place~$v|p$, let~$R_v$ be the
universal deformation ring for deformations of~$\rhobar|_{G_{F_v}}$
with determinant~$\varepsilon^{-1}$, and set~$R_p:=\wtimes_{v|p}R_v$.

Set~$G=\prod_{v|p}\PGL_1(D_v)$, which we identify with $\prod_{v|p}\PGL_2(\Q_p)$ via a fixed isomorphism.
By patching the completed homology
$\widetilde{H}_0(U_1(N(\rhobar)),\cO)_{\m(\rhobar)}$ (with a variation of the argument in~\cite[\S
9]{scholze}), we obtain\footnote{We caution the reader that one should carefully check this claim and we have not done so.} a ring~$R_\infty$ which is a power series ring
over~$R_p$, and an $R_\infty[G]$-module $M_\infty$ with an arithmetic action
in the sense of~\cite[\S 5.2]{geenewtonderivedpatching},
together
with an ideal $\ga_\infty\subset R_\infty$ such that $ R_\infty/\ga_\infty \cong R_{\rhobar}$
as local $\cO$-algebras and $M_\infty/\ga_\infty \cong \widetilde{H}_0(U_1(N(\rhobar)),\cO)_{\m(\rhobar)}$
as $R_{\rhobar}[G]$-modules.
(Note that we can ensure the smoothness of~$R_\infty$ over~$R_p$
since we are assuming that $D$ splits at all finite places and that there are no vexing primes.)

Applying~\cite[Prop.\ 5.2.2]{geenewtonderivedpatching} now gives a $G$-equivariant isomorphism
\[ \widetilde{H}_0(U_1(N(\rhobar)),\cO)_{\m(\rhobar)}\simeq (\wtimes_{v|p}\wP^1(\rhobar|_{G_{F_v}}))\wtimes_{R_p}R_{\rhobar},\]  where~$\wP^1(\rhobar|_{G_{F_v}})$ denotes the
projective envelope considered in Section~\ref{sec:comparison} in the
case that~$\rbar=\rhobar|_{G_{F_v}}$ and $\psi=1$. Arguing as in the proof of Corollary~\ref{cor: the
m-torsion in completed cohomology}, we obtain a $G$-equivariant isomorphism
\[ \widetilde{H}^0(U_1(N(\rhobar)),\F)[\m(\rhobar)]\simeq
  \wtimes_{v|p}\kappa(\rhobar|_{G_{F_v}}),\] where
$\kappa(\rhobar|_{G_{F_v}})$ denotes the representation  defined in Proposition~\ref{prop: an atome automorphe by any other name would smell as
    sweet}.

One can obtain analogous results in the case where~$[F:\Q]$ is odd and $D$ is split at one infinite place of $F$ and ramified at
all the others. In this case, one works with a tower of Shimura curves.
The main difference to the argument is to note that $\rhobar$ only contributes to completed
homology (or cohomology) in degree $1$ (since the $D^\times (\mathbb{A}^\infty)$-action factors through the reduced norm in degree $0$),
and  the $G_F$-action can be factored out by the same argument as for modular curves.
\end{rem}

\acknowledgements{This paper has its germ in conversations between three of us (M.E.,
T.G., V.P.) during the 2011 Durham Symposium on Automorphic Forms and
Galois Representation, and we would like to thank the organizers Fred
Diamond, Payman Kassaei and Minhyong Kim as well as Durham University,
EPSRC and the LMS for providing a fertile atmosphere for discussion.
The ideas of the paper were developed further when the six of us
participated in focussed research groups on ``The $p$-adic Langlands
program for non-split groups'' at the Banff Centre and AIM; we
would like to thank AIM and BIRS for providing an excellent working
atmosphere, and for their financial support. We would also like to
thank the anonymous referees for their comments on the paper.}

\bibliographystyle{amsalpha}
\bibliography{breuil-schneider}
\end{document}